\documentclass[letter,final,oneside,notitlepage,10pt]{article}

\usepackage{amssymb}
\usepackage{amsmath}
\usepackage{amstext}
\usepackage[]{geometry}
\usepackage{mathrsfs}
\usepackage{euscript}
\usepackage[all]{xy}
\usepackage{xstring}
\usepackage{slashed}
\usepackage{mathtools}

\def\N {\mathbb{N}}
\def\Z {\mathbb{Z}}

\def\R {\mathbb{R}}

\def\im{\mathrm{i}}
\def\id{\mathrm{id}}

\def\hc#1{\mathrm{h}_{#1}}
\def\h {\mathrm{H}}

\def\subset{\subseteq}

\renewcommand{\varepsilon}{\epsilon}
\renewcommand{\to}{\nobr\!\xymatrix@R=0cm@C=1.4em{\ar[r] &}\nobr}
\renewcommand{\mapsto}{\!\xymatrix@R=0cm@C=1.4em{\ar@{|->}[r] &}\!}
\renewcommand{\Rightarrow}{\!\xymatrix@R=0cm@C=1.4em{\ar@{=>}[r] &}\!}
\renewcommand{\Leftarrow}{\!\xymatrix@R=0cm@C=1.4em{\ar@{<=}[r] &}\!}
\newcommand{\incl}{\!\xymatrix@R=0cm@C=1.4em{\ar@{^(->}[r] &}\!}

\renewcommand\Leftrightarrow{\!\xymatrix@R=0cm@C=1.4em{\ar@{<=>}[r] &}\!}

\def\erf#1{(\ref{#1})}

\def\stackerf#1#2{\stackrel{\text{\erf{#1}}}{#2}}

\def\brackets#1{\IfStrEq{#1}{-}{}{(#1)}}
\def\subindex#1{\IfStrEq{#1}{-}{}{_{#1}}}
\def\trivtech{\mathcal{T}\hspace{-0.34em}r\hspace{-0.06em}i\hspace{-0.07em}v}

\newcommand{\alxydim}[2]{\begin{aligned}\xymatrix#1{#2}\end{aligned}}

\makeatletter
\def\bigset#1#2{\left\lbrace\;\begin{minipage}[c]{#1}\begin{center}#2\end{center}\end{minipage}\;\right\rbrace}

\makeatother

\newlength{\myl}

\def\ddt#1#2#3{\left.\frac{\mathrm{d}^{\IfStrEq{#1}{1}{}{#1}}}{\mathrm{d}#2}\IfStrEq{#2}{#3}{\right.}{\right|_{#3}}}

\def\nobr{\hspace{-0.2em}~\hspace{-0.2em}}
\def\maps{\nobr\colon\thinspace\nobr}
\def\df{\nobr := \nobr}
\def\eq{\nobr = \nobr}
\let\Oldin\in\renewcommand{\in}{\nobr\Oldin\nobr}
\let\Oldtimes\times\renewcommand{\times}{\nobr\Oldtimes\nobr}
\let\Oldotimes\otimes\renewcommand{\otimes}{\nobr\Oldotimes\nobr}

\newcommand{\ueins}{{\mathrm{U}}(1)}

\newcommand{\spin}[1]{{\mathrm{Spin}}\brackets{#1}}

\newcommand{\so}[1]{{\mathrm{SO}}\brackets{#1}}

\def\man{\mathcal{M}\!a\!n}

\def\diff{\mathcal{D}\!i\!f\!\!f}
\def\frech{\mathcal{F}\hspace{-0.12em}r\hspace{-0.1em}e\hspace{-0.05em}c\hspace{-0.05em}h}

\def\hom{\mathcal{H}\!om}
\def\homcon{\mathcal{H}\!om^{\!\nabla}\!}

\def\triv#1{\trivtech\brackets{#1}}
\def\trivcon#1{\trivtech^{\nabla\!}\brackets{#1}}

\def\hol#1#2{\mathrm{Hol}_{#1}(#2)}
\def\ev{\mathrm{ev}}

\def\pr{{\mathrm{pr}}}

\newlength{\widthtmp}
\def\length#1{\settowidth{\widthtmp}{#1}\the\widthtmp}

\def\buntech#1#2{\mathcal{B}\hspace{-0.01em}un_{\hspace{0.05em}#1}^{#2}}
\def\trivlin{\mathbf{I}}

\def\bun#1#2{\buntech{#1}{}\brackets{#2}}
\def\buncon#1#2{\buntech{#1}{\nabla}\hspace{-0.05em}\brackets{#2}}
\def\bunconflat#1#2{\buntech#1{\nabla_{\!0}}\hspace{-0.05em}\brackets{#2}}

\def\ubun#1{\bun\relax{#1}}
\def\ubunth#1{\bunth\relax{#1}}
\def\ubuncon#1{\buncon\relax{#1}}
\def\ubunconflat#1{\bunconflat\relax{#1}}

\def\ufusbun#1{\mathcal{F}\!us\buntech{}{}(#1)}
\def\ufusbunth#1{\mathcal{F}\!us\bunth{}{#1}}
\def\ufusbuncon#1{\mathcal{F}\!us\buncon{}{#1}}

\def\ufusbunconsf#1{\mathcal{F}\!us\bunconsf{}{#1}}

\def\bunconsf#1#2{\buntech{#1}{\!\nabla}{}^{_{\!\!s\!f}}\hspace{-0.15em}\brackets{#2}}

\def\bunth#1#2{\buntech{#1}{\!\!\!\;\;t\hspace{-0.03em}h}\hspace{-0.05em}\brackets{#2}}

\def\grbtech#1{\mathcal{G}\hspace{-0.06em}r\hspace{-0.06em}b_{\hspace{-0.07em}{#1}}}
\def\grb#1#2{\grbtech{#1}\brackets{#2}}
\def\grbcon#1#2{\grbtech{#1}^{\nabla\!}\brackets{#2}}

\def\ugrb#1{\grb{\,}{#1}}
\def\ugrbcon#1{\grbcon\relax{#1}}

\usepackage{amsthm}
\usepackage{graphicx}
\usepackage[margin=2cm,font=normalsize,labelfont=bf]{caption}
\usepackage{verbatim}
\usepackage{enumerate}
\usepackage{fancyhdr}
\usepackage[]{geometry}
\usepackage[normalem]{ulem}
\usepackage{xstring}
\usepackage{needspace}
\usepackage{color}
\usepackage{amstext}

\makeatletter

\newcounter{denseversion}
\newcounter{authorcounter}
\newcounter{adresscounter}

\def\title#1{\gdef\@title{#1}}
\def\@title{}
\def\subtitle#1{\gdef\@subtitle{#1}}
\def\@subtitle{}

\def\authortagsused{0}
\def\adresstag#1{\if!#1!\else$^{\;#1\;}$\fi}

\renewcommand{\author}[2][]{
  \stepcounter{authorcounter}
  \if!#1!\else\gdef\authortagsused{1}\fi
  \ifnum\value{authorcounter}=1
    \def\@authorstringa{#2\adresstag{#1}}
    \def\@authorstringb{#2}
    \def\@authorstringc{#2\adresstag{#1}}
  \else
    \g@addto@macro\@authorstringa{\ and #2\adresstag{#1}}
    \g@addto@macro\@authorstringb{\ and #2}
    \g@addto@macro\@authorstringc{\\#2\adresstag{#1}}
  \fi}
\def\@author{\ifnum\value{denseversion}=0\@authorstringa\else\@authorstringb\fi}

\def\@adressstringa{}
\def\@adressstringb{}
\newcommand{\adress}[2][]{
  \stepcounter{adresscounter}
  \ifnum\value{adresscounter}=1
    \g@addto@macro\@adressstringa{\ifnum\authortagsused=0\def\br{\\}\else\def\br{, }\fi\adresstag{#1}#2}
    \g@addto@macro\@adressstringb{\def\br{\\}\adresstag{#1}\parbox[t]{14cm}{#2}}
  \else
    \g@addto@macro\@adressstringa{\\[\bigskipamount]\adresstag{#1}#2}
    \g@addto@macro\@adressstringb{\\[\medskipamount]\adresstag{#1}\parbox[t]{14cm}{#2}}
  \fi}

\def\preprint#1{\gdef\@preprint{#1}}
\def\@preprint{}
\def\keywords#1{\gdef\@keywords{#1}}
\def\@keywords{}
\def\msc#1{\gdef\@msc{#1}}
\def\@msc{}
\def\email#1{
   \gdef\@email{#1}
   \g@addto@macro\@authorstringc{ {\it (#1)}}}
\def\@email{}
\def\dedication#1{\gdef\@dedication{#1}}
\def\@dedication{}

\def\mybaselinestretch#1{\gdef\@mybaselinestretch{#1}}
\def\@mybaselinestretch{}

\mybaselinestretch{1.2}
\renewcommand{\baselinestretch}{\@mybaselinestretch}

\def\denseversion{
  \setcounter{denseversion}{1}
  \newgeometry{left=3cm,right=3cm,top=3cm}
  \mybaselinestretch{1.1}
  \renewcommand{\baselinestretch}{\@mybaselinestretch}
  \normalfont
  \def\possiblelinebreak{}
  \fancyfoot[C]{\itshape{--$\,\,$\thepage$\,\,$--}}}

\newlength{\myparskip}
\newlength{\myproofparskip}

\def\possiblelinebreak{\\}

\renewcommand{\emph}[1]{\def\reserved@a{it}\ifx\f@shape\reserved@a\uline{#1}\else\textit{#1}\fi}

\newcommand{\mytableofcontents}{
   \ifnum\value{denseversion}=0
     \renewcommand{\baselinestretch}{1}
     \normalfont
     \tableofcontents
     \renewcommand{\baselinestretch}{\@mybaselinestretch}
     \normalfont
   \else
     \renewcommand{\baselinestretch}{0.5}
     \normalfont
     \tableofcontents
     \renewcommand{\baselinestretch}{\@mybaselinestretch}
     \normalfont
   \fi}

\newlength{\zeilenlaenge}
\def\putindent#1{
  \settowidth{\zeilenlaenge}{#1}
  \ifnum\zeilenlaenge>\textwidth
    #1
  \else
    \noindent #1
  \fi
}

\def\href#1#2{#2}

\def\kohyp{
  \usepackage{hyperref}
  \hypersetup{
    linktocpage = true,
    pdftitle = {\@title},
    pdfauthor = {\@author},
    pdfkeywords = {\@keywords},    
    bookmarksopen = true,
    bookmarksopenlevel = 1
  }}  
\def\showkeywords{\begin{flushleft}\footnotesize\textbf{Keywords}: \@keywords\end{flushleft}}
\def\showmsc{\begin{flushleft}\footnotesize\textbf{MSC 2010}: \@msc\end{flushleft}}

\newcounter{mythm}[subsection]
\newcounter{mainthm}

\def\setsecnumdepth#1{
  \setcounter{secnumdepth}{#1}
  \setcounter{mythm}{0}
  \ifnum \c@secnumdepth >0
    \ifnum \c@secnumdepth >1
      \def\themythm{\thesubsection.\arabic{mythm}}
      \numberwithin{equation}{subsection}
      \renewcommand\theequation{\thesubsection.\arabic{equation}}
    \else
      \def\themythm{\thesection.\arabic{mythm}}
      \numberwithin{equation}{section}
      \renewcommand\theequation{\thesection.\arabic{equation}}
    \fi
  \else
    \def\themythm{\arabic{mythm}}
  \fi}

\newenvironment{mythmenv}{\strut\ \setlength{\parskip}{\myproofparskip}}{\setlength{\parskip}{\myparskip}}

\newlength{\mythmskip}
\newlength{\mythmtopskip}
\newtheoremstyle{mythmstylea}{\mythmtopskip}{\mythmskip}{\it}{}{\bf}{.}{0em}{}
\newtheoremstyle{mythmstyleb}{\mythmtopskip}{\mythmskip}{}{}{\bf}{.}{0em}{}

\theoremstyle{mythmstylea}
\newtheorem{mytheorem}[mythm]{\nameTheorem}
\newtheorem{mydefinition}[mythm]{\nameDefinition}
\newtheorem{mycorollary}[mythm]{\nameCorollary}
\newtheorem{myproposition}[mythm]{\nameProposition}
\newtheorem{mylemma}[mythm]{\nameLemma}

\newtheorem{mymaintheorem}[mainthm]{\nameTheorem}
\newtheorem{mymaincorollary}[mainthm]{\nameCorollary}
\newtheorem{mymainproposition}[mainthm]{\nameProposition}
\newtheorem{mymaindefinition}[mainthm]{\nameDefinition}

\theoremstyle{mythmstyleb}

\newtheorem{myremark}[mythm]{\nameRemark}
\newtheorem{myproblem}[mythm]{\nameProblem}
\newtheorem{myexample}[mythm]{\nameExample}
\newtheorem{myexercise}[mythm]{\nameExercise}

\newenvironment{theorem}[1][]{\begin{mytheorem}[#1]\begin{mythmenv}}{\end{mythmenv}\end{mytheorem}}
\newenvironment{definition}[1][]{\begin{mydefinition}[#1]\begin{mythmenv}}{\end{mythmenv}\end{mydefinition}}
\newenvironment{corollary}[1][]{\begin{mycorollary}[#1]\begin{mythmenv}}{\end{mythmenv}\end{mycorollary}}
\newenvironment{proposition}[1][]{\begin{myproposition}[#1]\begin{mythmenv}}{\end{mythmenv}\end{myproposition}}
\newenvironment{lemma}[1][]{\begin{mylemma}[#1]\begin{mythmenv}}{\end{mythmenv}\end{mylemma}}
\newenvironment{remark}[1][]{\begin{myremark}[#1]\begin{mythmenv}}{\end{mythmenv}\end{myremark}}

\newenvironment{example}[1][]{\begin{myexample}[#1]\begin{mythmenv}}{\end{mythmenv}\end{myexample}}

\newenvironment{maintheorem}[1]{\def\themainthm{#1}\begin{mymaintheorem}\begin{mythmenv}}{\end{mythmenv}\end{mymaintheorem}}

\newenvironment{maincorollary}[1]{\def\themainthm{#1}\begin{mymaincorollary}\begin{mythmenv}}{\end{mythmenv}\end{mymaincorollary}}

\renewenvironment{proof}[1][\nameProof]{\noindent #1. \begin{mythmenv}}{\hphantom{$\square$}\hfill$\square$\end{mythmenv}\medskip}

\def\tocsection#1{\section*{#1}\addcontentsline{toc}{section}{#1}}

\def\mytitle{}
\def\zmptitle{
  \begin{tabular}{cc}
    \begin{minipage}[c]{0.4\textwidth}
      \begin{flushleft}
        \includegraphics[width=110pt]{../../tex/zmp}
      \end{flushleft}  
    \end{minipage}&
    \begin{minipage}[c]{0.55\textwidth}
      \begin{flushright}
      {\small\sf\@preprint}
      \end{flushright}
    \end{minipage}
  \end{tabular}
  \vskip 2cm}

\def\maketitle{
  \setlength{\parskip}{\myparskip}  
  \newpage
  \noindent
  \mytitle
  \begin{center}
    \LARGE\@title\\
    \if!\@subtitle!\else\smallskip\LARGE\@subtitle\\\fi
    \bigskip
    \if!\@author!\else\bigskip\large\@author\\\fi
    \ifnum\value{denseversion}=0
      \if!\@adressstringa!\else\bigskip\normalsize\@adressstringa\\\fi
      \if!\@email!\else\ifnum\value{authorcounter}=1\bigskip\normalsize\textit{\@email}\\\else\fi\fi
    \else
    \fi
    \if!\@dedication!\else\bigskip\normalsize{\@dedication}\\\fi
  \end{center}
  \ifnum\value{denseversion}=0\vskip 1.5cm\else\vskip0.5cm\fi
  \if!\@draft!\else\thispagestyle{empty}\fi}

\def\kobib#1{
  \begin{raggedright}
  \ifnum\value{denseversion}=0\else\small\fi
  \Oldbibliography{#1/kobib}
  \bibliographystyle{#1/kobib}
  \end{raggedright}
  \ifnum\value{denseversion}=0\else
      \noindent
      \if!\@authorstringc!\else
        \ifnum\authortagsused=0\ifnum\value{authorcounter}>1\normalsize\@authorstringc\\[\medskipamount]\else\fi\else\normalsize\@authorstringc\\[\medskipamount]\fi
      \fi
      \if!\@adressstringb!\else\normalsize\@adressstringb\\{}\fi
      \ifnum\authortagsused=0
        \ifnum\value{authorcounter}=1
          \if!\@email!\else\linebreak\normalsize\textit{\@email}\\{}\fi
        \else
        \fi
      \else
      \fi
  \fi
  }

\let\Oldbibliography\bibliography
\def\bibliography#1{
  \begin{raggedright}
  \ifnum\value{denseversion}=0\else\small\fi
  \Oldbibliography{#1}
  \end{raggedright}
  \ifnum\value{denseversion}=0\else
      \noindent
      \if!\@authorstringc!\else
        \ifnum\authortagsused=0\ifnum\value{authorcounter}>1\normalsize\@authorstringc\\[\medskipamount]\else\fi\else\normalsize\@authorstringc\\[\medskipamount]\fi
      \fi
      \if!\@adressstringb!\else\normalsize\@adressstringb\\{}\fi
      \ifnum\authortagsused=0
        \ifnum\value{authorcounter}=1
          \if!\@email!\else\linebreak\normalsize\textit{\@email}\\{}\fi
        \else
        \fi
      \else
      \fi
  \fi
}

\newenvironment{commentfigure}{\begin{comment}}{\end{comment}}
\newenvironment{sidewayscommentfigure}{\begin{minipage}}{\end{minipage}}

\def\draft#1#2#3#4{
  \ifnum#4=0
    \def\showcomments{ - Comments are not displayed}
  \else
    \renewenvironment{comment}{\begin{list}{}{\rightmargin=1cm\leftmargin=1cm}\item\sf\begin{small}}{\end{small}\end{list}}

    \def\showcomments{ - Comments are displayed}
  \fi
  \gdef\@draft{DRAFT - Version #1 - Last edited on #2 - Last edited by #3\showcomments}
  \fancyhead[C]{\footnotesize\tt\textcolor{red}{\@draft}}}
\def\@draft{}

\makeatother

\def\thinpairs#1{#1^2_{\text{\tiny \it thin}}}

\def\p{P}

\def\ev{\mathrm{ev}}

\def\hc#1{\mathrm{h}_{#1}}
\def\pcomp{\nobr\star\nobr}

\def\prev#1{\overline{#1}}

\def\un{\mathscr{R}}

\def\tr{\mathscr{T}}

\def\ufusbun#1{\mathcal{F}\!us\buntech{}{}(#1)}
\def\ufusbunth#1{\mathcal{F}\!us\bunth{}{#1}}
\def\ufusbuncon#1{\mathcal{F}\!us\buncon{}{#1}}

\def\ufusbunconsf#1{\mathcal{F}\!us\bunconsf{}{#1}}

\def\bunconsf#1#2{\buntech{#1}{\!\nabla}{}^{_{\!\!s\!f}}\hspace{-0.15em}\brackets{#2}}

\def\bunth#1#2{\buntech{#1}{\!\!\!\;\;t\hspace{-0.03em}h}\hspace{-0.05em}\brackets{#2}}

\usepackage[latin1]{inputenc}
\usepackage[english]{babel}

\def\quot#1{``#1''}

\def\quand{\quad\text{ and }\quad}
\def\quomma{\quad\text{, }\quad}

\def\nameTheorem{Theorem}
\def\nameDefinition{Definition}
\def\nameCorollary{Corollary}
\def\nameProposition{Proposition}
\def\nameLemma{Lemma}
\def\nameRemark{Remark}
\def\nameProblem{Problem}
\def\nameExample{Example}
\def\nameExercise{Exercise}
\def\nameProof{Proof}

\title{String geometry vs. spin geometry on loop spaces}
\author{Konrad Waldorf}
\adress{Ernst-Moritz-Arndt-Universität Greifswald\\
Institut für Mathematik und Informatik\\
Walther-Rathenau-Str. 47\\
D-17487 Greifswald}
\email{konrad.waldorf@uni-greifswald.de}
\msc{}

\kohyp
\usepackage{amstext}
\usepackage[normalem]{ulem}

\def\xypicst{3em}

\def\lspin#1{L\spin #1}
\def\lspinhat#1{\widetilde{L\spin #1}}

\def\inf#1{\EuScript{#1}}

\def\trunc{_1}
\def\ttrunc{_2}
\def\trivcontrunc#1{\mathrm{t}\trunc{\trivcon{#1}}}
\def\trivconttrunc#1{\mathrm{t}\ttrunc{\trivcon{#1}}}
\def\trivtrunc#1{\mathrm{t}\trunc{\triv{#1}}}
\def\trivttrunc#1{\mathrm{t}\ttrunc{\triv{#1}}}
\def\sc{\mathcal{S}\hspace{-0.08em}t\hspace{-0.065em}r\mathcal{C}\hspace{-0.04em}l}
\def\sccon{\sc^{\hspace{-0.04em}\nabla}\!}

\def\split{s}
\def\gbas{\mathcal{G}_{bas}}
\def\lghat{\widetilde{LG}}
\def\lop{\cup\,}
\def\ddt{\left . \frac{\mathrm{d}}{\mathrm{d}t} \right |_0}
\def\corr{spin}
\def\q{P}

\def\spintech{\mathcal{S}\hspace{-0.07em}pi\hspace{-0.06em}n}
\def\stringtech#1{\mathcal{S}\hspace{-0.09em}t\hspace{-0.07em}r\hspace{-0.06em}i\hspace{-0.07em}n\hspace{-0.09em}g_{#1}}
\def\fustech{f\!u\hspace{-0.09em}s}
\def\thtech{t\hspace{-0.05em}h}
\def\spst{\spintech(LM)}
\def\spstcon{\spintech^{\!\nabla\hspace{-0.2em}}(LM)}
\def\spstconsf{\spintech^{\!\nabla_{\!\!s\!f\!}}(LM)}

\def\spstconsffus{\spintech_{\hspace{-0.1em}\fustech\!}^{\!\nabla_{\!\!s\!f\!}}(LM)}
\def\spstth{\spintech^{th\hspace{-0.08em}}(LM)}
\def\spstfus{\spintech_{\hspace{-0.1em}\fustech\hspace{-0.08em}}(LM)}
\def\spstthfus{\spintech^{\hspace{0.1em}\thtech}_{\hspace{-0.1em}\fustech\hspace{-0.1em}}(LM)}
\def\stst#1#2{\stringtech#1(#2)}
\def\ststcon#1#2{\stringtech#1^{\nabla\!}(#2)}
\def\trivthfus#1{\trivtech_{\fustech}^{\,\thtech}(#1)}

\usepackage{amsmath}
\usepackage{amssymb}
\usepackage{graphicx}

\begin{document}

%\draft{1.0}{22.3.2014}{Konrad}{0}
\denseversion

\maketitle

\begin{abstract}
We introduce various versions of spin structures on  free loop spaces of smooth manifolds, based on  a classical notion due to Killingback, and additionally coupled to two relations between loops: thin homotopies and loop fusion. 
The central result of this article is an equivalence between these enhanced versions of spin structures  on the loop space and string structures on the manifold itself. The equivalence exists in two  settings:   in a purely topological  one and a in geometrical one that includes spin connections and string connections. Our results   provide a consistent, functorial, one-to-one dictionary between string geometry and spin geometry on loop spaces.  
%%\showkeywords
\end{abstract}

{
\small
\mytableofcontents
}
%\noarxiv
%\nolinks

\setsecnumdepth{1}

\section{Introduction}

One perspective to  classical two-dimensional field theories on a Riemannian manifold $M$, also known as sigma models, is to regard them as  one-dimensional field theories on the free loop space $LM$: the points of $LM$ are the \quot{closed strings} in $M$. For example, if we want to understand the coupling of strings to gauge fields, this perspective  makes us study principal bundles with connections over $LM$. And if we want to understand fermions, it lets us ask for  spin structures on loop spaces.

In order to study fermions on an oriented, $n$-dimensional Riemannian manifold, one has to lift the structure group of the frame bundle of $M$ from $\so n$ to a covering group that admits appropriate unitary representations, $\spin n$. Analogous steps on the loop space require, in the first place, to choose an orientation: a reduction of the structure group of $LM$, namely  $L\so n$, to the connected component of the identity,   $L\spin n$, see \cite{atiyah2,mclaughlin1}. Such a reduction can, for instance,  be induced from a spin structure on $M$. 
In the next step, one observes that  $L\spin n$ has no appropriate unitary representations. It only has projective ones, i.e. representations of its universal central extension,
\begin{equation}
\label{eq:ext}
1 \to \ueins \to \lspinhat n \to \lspin n \to 1\text{.}
\end{equation}
Thus, we require a lift of the structure group of  $LM$ from $\lspin n$ to this central extension; such a lift  is called a \emph{spin structure} on $LM$ \cite{killingback1}. 
We remark that an important difference to ordinary spin structures is that the central subgroup of the extension \erf{eq:ext} is  the \emph{continuous} group $\ueins$ instead of the \emph{discrete} group $\Z/2\Z$. One effect of this difference is that it is non-trivial to lift a given connection on the frame bundle of $LM$ to a connection on the lifted bundle, a \emph{spin connection} \cite{Coquereaux1998}. Every spin structure on $LM$ admits a spin connection \cite{manoharan1}, but  there might be non-equivalent choices.

\subsubsection*{Deficits of the loop space theory}

Returning to the attempt to understand  the coupling of strings to gauge fields via, say, principal $\ueins$-bundles with connection over loop spaces, one soon encounters the problem that not all aspects of  the two-dimensional theory  can be described in terms of such bundles. For  example, if two strings join in form of a pair of pants, there is no sensible way to describe the gauge field coupling of this process solely in terms of a bundle over $LM$. This deficit of the loop space theory has lead to the development of   \emph{B-fields}, structure defined on the manifold itself that fulfills all requirements for a gauge field for strings. Nowadays it is well understood that a B-field is a \emph{$\ueins$-gerbe with connection} \cite{gawedzki3,brylinski1}. 
The relation between gerbes over $M$ and bundles over $LM$ can be understood on a cohomological level in terms of a transgression homomorphism
\begin{equation}
\label{eq:tau}
\tau: \h^n(M,\Z) \to \h^{n-1}(LM,\Z)\text{,}
\end{equation}
which for $n=3$ takes the Dixmier-Douady class of a gerbe over $M$ to the first Chern class of a principal $\ueins$-bundle over $LM$. Various differential-geometric versions of the transgression homomorphism have been developed that also include connections on both sides \cite{brylinski1,gawedzki1,waldorf5}.
A general fact is that  transgression is not injective; this loss of information explains precisely the above-mentioned deficit of the loop space theory for gauge fields.

A similar phenomenon has been observed  for \emph{geometric spin structures} on loop spaces, i.e. spin structures with spin connections.  For the consistency of the fermionic theory (a version of  the supersymmetric sigma model)  it is necessary to trivialize a certain \emph{Pfaffian line bundle} over the mapping spaces of closed spin surfaces into $M$ \cite{freed4,freed5}. A spin structure on the loop space only provides such  trivializations over the mapping space of \emph{genus one} surfaces. In order to remedy this deficit (among other issues) Stolz and Teichner have proposed a notion of a  \emph{geometric string structure} on $M$, consisting of a \emph{string structure} and a \emph{string connection} \cite{stolz1}. Spin manifolds that admit such structures are called \emph{string manifolds}; they are characterized by the vanishing of the first fractional Pontryagin class $\frac{1}{2}p_1(M)$.
The proof that a geometric string structure indeed provides trivializations of the Pfaffian line bundle over mapping spaces of \emph{arbitrary} surfaces  was provided later by Bunke \cite{bunke1} based on a gerbe-theoretical formulation of geometric string structures  introduced  in  \cite{waldorf8}. Additionally, that formulation allows to define a transgression procedure for geometric string structures on $M$, analogous to the homomorphism \erf{eq:tau}, that results in spin structures on $LM$ \cite{Waldorfa}.  This transgression procedure is again afflicted with a loss of information \cite{Pilch1988}, explaining the limitation of the loop space theory to genus one surfaces.

We remark that several other aspects are not yet understood, neither in terms of spin geometry on $LM$ nor in terms of string geometry on $M$. Examples are the Dirac operator on $LM$ postulated by Witten \cite{witten2}, or the Höhn-Stolz conjecture \cite{stolz4}.  The quest for methods to attack problems like these is the motivation for studying relations between geometry on $M$ and geometry on $LM$. The purpose of the present article is to contribute a new instance of such relations: an equivalence between (geometric) string structures on $M$ and a version of (geometric) spin structures on $LM$.

\subsubsection*{Thin homotopy and loop fusion}

We return to the above-mentioned transgression of gerbes (with connection) over $M$ to principal $\ueins$-bundles (with connection) over $LM$, suffering from a loss of information. It turns out that one can equip $\ueins$-bundles over loop spaces with additional structures, in such a way that an inverse of transgression can be defined, and an equivalence between  gerbes over $M$ and versions of $\ueins$-bundles over $LM$ is achieved; see  \cite{waldorf9,waldorf10,waldorf11} or \cite{Kottke,Kottkea} for an alternative approach. The relevant additional structures couple $\ueins$-bundles over $LM$ to two operations that only exist in loop spaces (rather than in general manifolds): thin homotopies and loop fusion.
\begin{figure}[h]
\begin{center}
\begin{tabular}{ccc}
\includegraphics[height=7em]{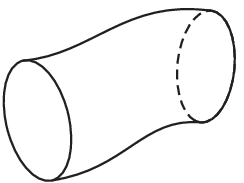}   && 
\includegraphics[height=7em]{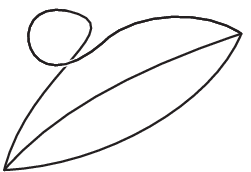} \\
(a)
\begin{minipage}[t]{0.3\textwidth}
\raggedright
A  homotopy between loops, regarded as a tube.
\end{minipage} && 
(b) 
\begin{minipage}[t]{0.3\textwidth}
\raggedright
Two loops with a common segment. 
\end{minipage}
\end{tabular}
\end{center}
\vspace{-1.5em}
\end{figure}
Roughly speaking, \emph{thin homotopies} are homotopies between loops that have \quot{zero area} when regarded as tubes in $M$ as shown in Figure (a). The relevance of thin homotopies has been noticed in axiomatic approaches to the parallel transport of connections on bundles and gerbes \cite{barret1,schreiber3}.  A \emph{thin structure} on a principal $\ueins$-bundle $P$ over $LM$ is a way to identify consistently the fibres of $P$ over thin homotopic loops. A connection on $P$ is called \emph{superficial}, if such a thin structure can be induced by parallel transport along a thin homotopy (independently of the choice of the thin homotopy). 

The second operation, \emph{loop fusion}, joins two loops along a common segment, see Figure (b). A \emph{fusion product} on a principal $\ueins$-bundle $P$ over $LM$ is a structure that lifts loop fusion to the fibres of  $P$. These additional structures   furnish a  category $\ufusbunth {LM}$ of principal $\ueins$-bundles over $LM$ equipped with fusion products and thin structures, and another category $\ufusbunconsf {LM}$ of principal $\ueins$-bundles over $LM$ equipped with fusion products and superficial connections. These two categories are \quot{loop space duals} of the bicategories $\ugrb M$ of gerbes over $M$ and $\ugrbcon M$ of gerbes with connection over $M$, respectively. These dualities  can be expressed in terms of a commutative diagram
\begin{equation}
\label{eq:duality}
\alxydim{@=\xypicst}{\ufusbunconsf{LM} \ar[r] \ar[d] & \hc 1 \ugrbcon M \ar[d] \\ h\ufusbunth{LM} \ar[r] & \hc 1 \ugrb M }
\end{equation}
of monoidal categories and  functors, in which the horizontal arrows are equivalences and the vertical arrows describe the passage from the setting \quot{with connections} to the one \quot{without connections}. The symbol $\hc 1$  stands for the truncation of a bicategory down to a category 
\begin{comment}
(with the same objects and 2-isomorphism classes of 1-morphisms)
\end{comment}
and the symbol $h$   stands for the homotopy category (where bundle morphisms become identified if they are homotopic). The horizontal functors in the diagram are called \emph{regression} as they are inverse to transgression; we refer to \cite{waldorf11} for a more detailed exposition.

The equivalence on top of the diagram  explains  how the deficit of the loop space theory of gauge fields for strings has to be compensated, namely by the addition of a fusion product and the requirement that the connection be superficial. Indeed, a fusion product provides exactly the structure needed in order to account for the joining of two strings in form of a pair of pants, see the discussion in \cite[Section 5.3]{waldorf10}.

\subsubsection*{Results of the present article}

In the present article we discuss an equivalence between the string geometry on $M$ and  versions of spin geometry on the loop space $LM$. The first part of this article is concerned with determining how exactly these versions have to be defined, and the second part is concerned with the proof that they serve their purpose and yield  the claimed equivalence.

We introduce two versions of spin structures on loop spaces: a category $\spstthfus$ of \emph{thin fusion spin structures} (Definition \ref{def:thinfusionspin}) and another category $\spstconsffus$ of \emph{superficial geometric fusion spin structures} (Definition \ref{def:geometricfusion}). As the terminology suggests, our strategy is to equip Killingback's original spin structures with structures that have already proved themselves: fusion products, thin structures, and superficial connections. The main issue is to connect these structures correctly  to action of the central extension $\lspinhat n$ on the spin structure.  Therefore, we start the first part of the present article by revisiting loop group geometry through transgression of multiplicative gerbes, bringing fusion products, thin structures, and superficial connections in context with central extensions of loop groups.

The categories $\spstthfus$ and $\spstconsffus$ are related, respectively, to the bicategories $\stst \relax M$ and $\ststcon \relax M$ of string structures and geometric string structures  introduced in \cite{waldorf8} as bicategory of trivializations of the Chern-Simons 2-gerbe. The relation is established  by regression functors that are inverse to the above-mentioned transgression procedure for geometric string structures. The main result of this article is
the following.

\needspace{10em}
\begin{maintheorem}{A}
\label{th:main}
Let $M$ be a connected spin manifold of dimension $n=3$ or $n>4$. There is a  commutative diagram of categories and functors,
\begin{equation*}
\alxydim{@=\xypicst}{\spstconsffus  \ar[r] \ar[d] & \hc 1 \ststcon \relax M \ar[d] \\ h\spstthfus \ar[r] &  \hc 1 \stst \relax M\text{.}}
\end{equation*}
If $M$ is  string, all  categories in the diagram are non-empty, and the following results hold:
\begin{enumerate}[(i)]

\item 
The horizontal functors are equivalences of categories, and the vertical functors are essentially surjective.

\item 
The diagram is a torsor over the diagram \erf{eq:duality} in the sense that each category is a torsor over the monoidal category in the corresponding corner of \erf{eq:duality}, and each functor is equivariant along the corresponding functor in \erf{eq:duality}.    

\end{enumerate}
If $M$ is not string, then all four categories in the diagram are empty.
\end{maintheorem}

Here, a category is a torsor over a monoidal category if it is a module for that monoidal category and the action is free and transitive in a sense  explained later.

We spell out explicitly what Theorem \ref{th:main} implies upon passing to  isomorphism classes of objects, an operation that we denote by the symbol $\hc 0$.  The set $\hc 0 \stst \relax M$ can be identified with a set $\sc (M)$ of \emph{string classes} \cite{redden1,waldorf8}; these can easily be described as cohomology classes $\xi\in \h^3(FM,\Z)$ on the total space of the spin-oriented frame bundle $FM$ of $M$ that restrict over each fibre to a generator of $\h^3(\spin n,\Z)\cong \Z$. We introduce  an analogous description  of \emph{geometric} string structures in terms of  \emph{differential} cohomology, which we call \emph{differential string classes} (Definition \ref{def:diffstringclass}). A differential string class is a differential cohomology class $\hat\xi\in\hat\h^3(FM)$, subject to a condition in the differential cohomology of $FM \times \spin n$ that involves a certain 2-form known from classical Chern-Simons theory.    We prove that the set $\sccon(M)$ of differential string classes can be identified with the set $\hc 0 \ststcon\relax M$ of isomorphism classes of geometric string structures (Theorem \ref{th:differentialstringclasses}). Under these identifications, Theorem \ref{th:main} implies the following statement.

\begin{maincorollary}{B}
\label{co:main}
Let $M$ be a connected string manifold of dimension $n=3$ or $n>4$. There is a commutative diagram
\begin{equation*}
\alxydim{@=\xypicst}{\hc 0 \spstconsffus \ar[r] \ar[d] & \sccon(M) \ar[d] \\ \hc 0 \spstthfus  \ar[r] & \sc(M)\text{.} }
\end{equation*}
The map in the first row is an equivariant bijection between torsors over the differential cohomology group $\hat\h^3(M)$, and the map in the second row is an equivariant bijection between torsors over the ordinary cohomology group $\h^3(M,\Z)$.  Moreover, the vertical maps are surjective and  equivariant along the projection $\hat\h^3(M) \to \h^3(M,\Z)$. In particular, the fibres of the vertical maps are torsors over the group $\Omega^2(M)/\Omega^2_{cl,\Z}(M)$ of 2-forms modulo closed 2-forms with integral periods.
\end{maincorollary}

The last statement follows because the group $\Omega^2(M)/\Omega^2_{cl,\Z}(M)$ is precisely  the kernel of the projection $\hat\h^3(M) \to \h^3(M)$ from differential to ordinary cohomology.

Summarizing, either in the categorical or in the set-theoretical setting, we provide a consistent dictionary between string geometry and spin geometry on loop spaces. We remark that in a first approximation of such an equivalence, Witten proposed to impose that spin structures be equivariant under the rotation action of $\ueins$ on $LM$ \cite{witten2}. In a previous article \cite{Waldorfa} I have considered a version of spin structures with fusion products (but without thin structures), and proved that such spin structures exist if and only if $M$ is a string manifold. Recently, Kottke and Melrose  introduced a version of spin structures that combines  fusion products and equivariance under a group of reparameterizations of $S^1$ (including rotations) \cite{Kottke}. This version achieved, on the level of equivalence classes, a bijection with the set of string classes. The results of the present article improve that bijection in two aspects: we upgrade it to an equivalence of categories and amend it by a second equivalence in the setting \quot{with connections}. 

\begin{comment}
The spin structures of Kottke and Melrose, though formally similar, are a bit different from the ones used in this article. One difference is the way how loop fusion is realized as an operation on \emph{smooth} loops. In the present article  we use   smooth paths with \quot{sitting instants} and the rather primitive theory of diffeological spaces. Kottke and Melrose develop a more elegant approach  using piecewise-smooth paths, energy loop spaces, and lithe smoothness. A second difference is that the thin structures we use here require equivariance under \emph{all} thin homotopies, whereas Kottke and Melrose manage to require only equivariance under reparameterizations. I have included a couple of comments about the relation to \cite{Kottke} at the end of Section \ref{sec:fusionspin}.
\end{comment}

\subsubsection*{Method of proof and organization of the paper}

For the proof of Theorem \ref{th:main} we will collect various partial results throughout this article; in the final Section \ref{sec:proof} we summarize these and show that the theorem is fully proved. The main tool in the proof is \emph{lifting gerbe theory} over the loop space, which  allows us to split the work into two parts. The first part (Sections  \ref{sec:fusionextensions},  \ref{sec:spinstructures},  \ref{sec:geometricspin}) is to reformulate spin structures and all additional structures in terms of trivializations of the  \emph{spin lifting gerbe} over $LM$ (Proposition \ref{prop:equivthinfusion} and Corollary \ref{co:equivsupfusion}).  This reformulation is based on work of Murray \cite{murray}, Gomi \cite{gomi3}, and previous work \cite{Waldorfa}. A crucial new aspect we encounter here is that the standard theory for connections on lifting gerbes  must be refined in a certain way in order to take thin homotopies into account (Proposition \ref{prop:chi}).  

The second part (Sections  \ref{sec:stringstructures}, \ref{sec:geomstringstructures}, \ref{sec:transgression})  is concerned with the problem to tailor the  bicategories $\ststcon \relax M$ and $\stst \relax M$ of (geometric) string structures into a form that allows a direct application of the duality between gerbes and bundles over loop spaces. The resulting loop space structure can then be identified with exactly those trivializations of the spin lifting gerbe that we identified in the first part as reformulations of the categories $\spstthfus$ and $\spstconsffus$, see Theorems \ref{th:trequivcon} and \ref{th:trequiv}.
The tailoring of the bicategories involves a general decategorification procedure for trivializations of bundle 2-gerbes. A key result that we prove is that in case of the Chern-Simons 2-gerbe,  this decategorification procedure is an equivalence of categories (Theorems \ref{th:truncation} and \ref{th:truncationcon}). 

\begin{comment}
The organization of this article is as follows. In Section \ref{sec:fusionextensions} we revisit central extensions of loop groups using the theory of multiplicative bundle gerbes. We collect additional structures on the loop group to be used in the following sections. In Section \ref{sec:spinstructures} we introduce the definition of thin fusion spin structures, and develop their reformulation in terms of the spin lifting gerbe. In Section \ref{sec:geometricspin} we do the same in the setting with connections: we introduce the definition of superficial geometric fusion spin structures, and develop their reformulation in terms of a certain connection on the spin lifting gerbe. In Sections \ref{sec:stringstructures} and \ref{sec:geomstringstructures} we review the bicategory of (geometric) string structures, and introduce various equivalent, but decategorified versions, of which the last is the set of (differential) string classes.  In Section \ref{sec:transgression} we turn on the transgression-regression machine and obtain equivalences between string structures and trivializations of the spin lifting gerbe. In Section \ref{sec:proof}  we summarize the various results and explain how they combine to a complete proof of Theorem \ref{th:main}. After Section \ref{sec:proof} we  include for the convenience of the reader a table of notations. 
\end{comment}

\paragraph{Acknowledgements.} This work is supported by the DFG network \quot{String Geometry} (project code 594335), and by the Erwin-Schrödinger Institute for Mathematical Physics in Vienna.

\setsecnumdepth 2

\section{Loop group geometry via multiplicative gerbes}

\label{sec:fusionextensions}

In this section we explore the geometry of central extensions of the loop group $LG$ of a Lie group $G$ via multiplicative bundle gerbes over $G$. The goal is to construct models for central extensions with specific additional structures: superficial connections, thin structures, and fusion products. The results of this section will be applied in the  sequel to $G=\spin n$. 

\subsection{Transgression and central extensions}

\label{sec:centralextension}

We  use the theory of bundle gerbes (with structure group $\ueins$) and connections on those. Introductions can be found in \cite{murray,carey2,murray3,waldorf1}. We  denote by $\ugrb X$ and $\ugrbcon X$ the bicategories of bundle gerbes and bundle gerbes with connection over a smooth manifold $X$, respectively. The 1-morphisms are called (connection-preserving) \emph{isomorphisms}, and the 2-morphisms are called (connection-preserving) \emph{transformations}.  The operation of \quot{forgetting the connection} is a surjective, but neither full nor faithful 2-functor
\begin{equation}
\label{eq:forgetconnections}
\ugrbcon X \to \ugrb X\text{.}
\end{equation}
\begin{comment}
This means that every bundle gerbe admits a connection, and it means that \quot{being connection-preserving} is \emph{additional structure} for isomorphisms, which may or may not exist, and if it exists, it might not be unique. 
\end{comment}

Let $G$ be a Lie group with Lie algebra $\mathfrak{g}$, and let $\left \langle -,-  \right \rangle$ be a symmetric  invariant bilinear form on $\mathfrak{g}$. There is a canonical, left-invariant closed 3-form $H\in \Omega^3(G)$ whose value at the identity is given by $H_1(X,Y,Z) = \left \langle X,[Y,Z]  \right \rangle$. In terms of the  left-invariant Maurer-Cartan form $\theta$ on $G$ it is  given by
\begin{equation}
\label{eq:H}
H=\frac{1}{6}\left \langle \theta\wedge [\theta\wedge \theta]  \right \rangle\text{.}
\end{equation}
\begin{comment}
So we use the standard convention for plugging into a wedge product:
\begin{equation*}
(\alpha\wedge\beta)(X_1,...,X_p,X_{p+1},...,X_{p+q})=\sum_{\sigma\in Sh_{p,q}} (-1)^{\sigma}\alpha(X_{\sigma(1)},...,X_{\sigma(p)})\beta(X_{\sigma(p+1)},...,X_{\sigma(p+q)})\text{,}
\end{equation*}
where $Sh_{p,q}\subset S_{p+q}$ contains only those permutations with
\begin{equation*}
\sigma(1)<\sigma(2)<...<\sigma(p)
\quand
\sigma(p+1)<\sigma(p+2)<...<\sigma(p+1)\text{.}
\end{equation*}
\end{comment}
We fix a bundle gerbe $\mathcal{G}$ over $G$ with connection of curvature $H$. Such a bundle gerbe exists if and only if $H$ has integral periods, in which case $H$ represents the Dixmier-Douady class $\mathrm{DD}(\mathcal{G}) \in \h^3(G,\Z)$ in real cohomology. Different choices of possible bundle gerbes with connection (up to connection-preserving isomorphisms) are parameterized by $\h^2(G,\ueins)$.

\begin{example}
Suppose $G$ is compact, simple  and simply-connected, for example,  $G\eq\spin n$ for $n=3$ or $n>4$. Then, $\left \langle -,-  \right \rangle$ is a multiple of the Killing form, and it can be normalized such that $H$ has integral periods and represents a generator $\gamma \in \h^3(G,\Z) \cong \Z$. We have $\h^2(G,\ueins)=0$. Thus, there exists a (up to connection-preserving isomorphisms) unique bundle gerbe $\mathcal{G}$ with connection of curvature $H$. Its Dixmier-Douady class is  the generator $\gamma$. This bundle gerbe $\mathcal{G}$ is called the \emph{basic gerbe} over $G$, and it will be denoted by $\gbas$. There exist  Lie-theoretical models for $\gbas$  \cite{gawedzki1,meinrenken1}.
\end{example}

The double group $G^2$ carries a canonical 2-form 
\begin{equation}
\label{eq:rho}
\rho := \left \langle  \pr_1^{*}\theta\wedge \pr_2^{*}\bar\theta  \right \rangle \in \Omega^2(G^2)\text{,}
\end{equation}
with $\theta$, $\bar\theta$ the left- and right-invariant Maurer-Cartan forms on $G$, respectively. 
\begin{comment}
The 2-form $\rho$ captures an \quot{error} in the multiplicativity of the 3-form $H$. 
\end{comment}
Let $m,\pr_1,\pr_2\maps G^2 \to G$ denote the multiplication and the two projections, respectively. Then 
\begin{equation}
\label{eq:multH}
\pr_1^{*}H + \pr_2^{*}H=m^{*}H + \mathrm{d}\rho
\end{equation}
as 3-forms on $G^2$.
\begin{comment}
This follows from the multiplication rule for the Maurer-Cartan-form,
\begin{equation*}
m^{*}\theta = \mathrm{Ad}_{\pr_2}^{-1}(\pr_1^{*}\theta) + \pr_2^{*}\theta\text{,}
\end{equation*}
and the formula $\mathrm{Ad}(\theta)=\bar\theta$.
From this we get
\begin{equation*}
m^{*}H=\pr_1^{*}H + \pr_2^{*}H+\frac{1}{2}\left \langle \pr_1^{*}\theta\wedge \pr_2^{*}[\bar\theta \wedge \bar\theta]  \right \rangle + \frac{1}{2}\left \langle \pr_1^{*}[\theta\wedge\theta]\wedge \pr_2^{*}\bar\theta  \right \rangle\text{.}
\end{equation*}
The Maurer-Cartan-equations are
\begin{equation*}
\mathrm{d}\theta +\frac{1}{2} [\theta\wedge\theta]=0
\quand
\mathrm{d}\bar\theta - \frac{1}{2}[\bar\theta \wedge\bar\theta]=0\text{.}
\end{equation*}
For if $X$ and $Y$ are left invariant vector field, then
\begin{equation*}
\mathrm{d}\theta(X,Y)=X\theta(Y)-Y\theta(X)-\theta([X,Y])=-\theta([X,Y])=-[\theta(X),\theta(Y)]=-\frac{1}{2}[\theta\wedge\theta](X,Y)\text{.}
\end{equation*}
Anyway, we have
\begin{equation*}
m^{*}H=\pr_1^{*}H + \pr_2^{*}H + \left \langle \pr_1^{*}\theta \wedge \mathrm{d}\pr_2^{*}\bar\theta  \right \rangle - \left \langle  \mathrm{d}\pr_1^{*}\theta\wedge\pr_2^{*}\bar\theta \right \rangle\text{.}
\end{equation*}
The last two terms are $-\mathrm{d}\rho$.
\end{comment}
\begin{comment}
On the triple group $G ^{3}$, the 2-form $\rho$ satisfies a cocycle condition.
\end{comment}
We have four maps $\pr_{12},\pr_{23},m_{23},m_{12}:G^3 \to G^2$, where $\pr_{12}$ and $\pr_{23}$ project to the indexed components, and $m_{23}$ and $m_{12}$ multiply the indexed components. Then, 
\begin{equation}
\label{eq:rhococycle}
\pr_{23}^{*}\rho + m_{23}^{*}\rho = \pr_{12}^{*}\rho + m_{12}^{*}\rho\text{.}
\end{equation}
We may re-interpret Equations \erf{eq:multH} and \erf{eq:rhococycle} by considering the simplicial manifold $BG$, see \cite{waldorf5}.
\begin{comment}
Its $n$-th manifold is $G^{n}$, and its face maps $\Delta_i^q:G^q\to G^{q-1}$ in degrees $q\eq2,3$ are precisely the maps that appear in \erf{eq:multH} and \erf{eq:rhococycle}.
\end{comment}
Denoting by $\Delta: \Omega^{*}(G^{q-1})\to\Omega^{*}(G^{q})$ the alternating sum over the pullbacks along the face maps, Equations \erf{eq:multH} and \erf{eq:rhococycle} become
\begin{equation}
\label{eq:simpderham}
\Delta H=\mathrm{d}\rho
\quand
\Delta\rho=0\text{.}
\end{equation}

A bundle gerbe $\mathcal{G}$ with connection of curvature $H$ can be seen as a lift from a differential form setting to a  cohomological setting. A corresponding lift of equations  \erf{eq:multH} and \erf{eq:rhococycle} is called a \emph{multiplicative structure} on $\mathcal{G}$. We recall that  2-forms are  connections on trivial gerbes; thus, we have a bundle gerbe $\mathcal{I}_{\rho}$ over $G^{2}$ -- it has  vanishing Dixmier-Douady class and curvature $\mathrm{d}\rho$. A multiplicative structure on $\mathcal{G}$ consists of a  connection-preserving isomorphism
\begin{equation*}
\mathcal{M}:\pr_1^{*}\mathcal{G} \otimes \pr_2^{*}\mathcal{G} \to m^{*}\mathcal{G} \otimes \mathcal{I}_{\rho}
\end{equation*}
of bundle gerbes over $G^{2}$, and of a connection-preserving transformation
\begin{equation*}
\alxydim{@R=\xypicst@C=4.5em}{\pr_1^{*}\mathcal{G} \otimes \pr_2^{*}\mathcal{G} \otimes \pr_3^{*}\mathcal{G} \ar[r]^-{\pr_{12}^{*}\mathcal{M} \otimes \id} \ar[d]_{\id \otimes \pr_{23}^{*}\mathcal{M}} & \pr_{12}^{*}\mathcal{G} \otimes \pr_3^{*}\mathcal{G} \otimes \mathcal{I}_{\pr_{12}^{*}\rho} \ar[d]^{m_{12}^{*}\mathcal{M} \otimes \id} \ar@{=>}[dl]|*+{\alpha} \\ \pr_1^{*}\mathcal{G} \otimes \pr_{23}^{*}\mathcal{G} \otimes \mathcal{I}_{\pr_{23}^{*}\rho} \ar[r]_-{m_{23}^{*}\mathcal{M} \otimes \id} & \mathcal{G}_{123} \otimes \mathcal{I}_{\rho_{\Delta}}}
\end{equation*}
between isomorphisms over $G^3$, where $\rho_{\Delta}$ is either side of \erf{eq:rhococycle}.
The transformation $\alpha$ has to satisfy a pentagon axiom over $G^4$. Multiplicative bundles gerbes (without connections) have been introduced in \cite{carey4}, the theory of connections is developed in \cite{waldorf5}.
\begin{comment}
Another perspective is to view the isomorphism $\mathcal{M}$  as a lift of the group structure from $G$ to the bundle gerbe $\mathcal{G}$, and to view the transformation $\alpha$ as an \quot{associator} for that lifted group structure.
\end{comment}

The quadruple $(H,\rho,0,0)$ is a degree 4 chain in the de Rham complex of the simplicial manifold $BG$.
\begin{comment}
That is, in the total complex of the double complex $\Omega^{*}(G^{*})$, with differentials $\mathrm{d}$ and $\Delta$. 
\end{comment}
Closedness of $H$ together with Equations \erf{eq:simpderham} show that it is a cocycle, and thus represents an element in $\h^4(BG,\R)$. Multiplicative bundle gerbes with connection relative to the differential forms $H$ and $\rho$ exist if and only if that  class is integral.
Different choices are parameterized by $\h^3(BG,\ueins)$, see \cite[Proposition 2.4]{waldorf5}.

\begin{example}
If $G$ is compact and simple, then $\h^3(BG,\ueins)=0$, so that multiplicative gerbes with connection are (up to connection-preserving isomorphisms compatible with the multiplicative structure) uniquely determined by $H$ and $\rho$, hence by $\left \langle -,-  \right \rangle$. If $G$ is in addition simply connected, every bundle gerbe with connection of curvature $H$ admits a multiplicative structure relative to the 2-form $\rho$ \cite[Example 1.5]{waldorf5}. In particular, the basic gerbe $\gbas$ over a compact, simple and simply-connected Lie group has a unique multiplicative structure.  
Based on  explicit models for the  basic gerbe over such groups, it is possible to construct this unique multiplicative structure  \cite{Waldorf}.
\end{example}

In the following we continue with a fixed multiplicative bundle gerbe $\mathcal{G}$ with connection over a general Lie group $G$, relative to the differential forms $H$ and $\rho$ of Equations \erf{eq:H} and \erf{eq:rho}. 
\begin{comment}
We describe how to construct central extensions of the loop group.
\end{comment}

For every smooth manifold $X$, there is a transgression functor
\begin{equation}
\label{eq:tr}
 \hc 1 \ugrbcon X \to \ubun{LX}:\mathcal{G}\mapsto \tr_{\mathcal{G}}
\end{equation} 
with target the category of Fréchet principal $\ueins$-bundles over the free loop space $LX
\eq C^{\infty}(S^1,X)$. The  category $\hc 1 \ugrbcon X$ is obtained from the  bicategory $\ugrbcon X$ by identifying 2-isomorphic isomorphisms.

Transgression for gerbes has first been defined by Gaw\c edzki  in terms  of cocycles for Deligne cohomology \cite{gawedzki3}, and by Gaw\c edzki-Reis  for bundle gerbes \cite{gawedzki1}.  Brylinski   has defined transgression in terms of Dixmier-Douady sheaves of categories \cite{brylinski1}.  The functor  \erf{eq:tr} that we use here is defined in \cite{waldorf5}. It is monoidal with respect to the tensor product of bundle gerbes and principal $\ueins$-bundles,  it is natural with respect to smooth maps $f\maps X \to X'$ between smooth manifolds and the induced maps $Lf\maps LX \to LX'$ between their loop spaces, and it sends  trivial bundle gerbes $\mathcal{I}_{\rho}$ to canonically trivializable bundles. Furthermore, it satisfies \begin{equation}
\label{eq:trprop}
\mathrm{c}_1(\tr_{\mathcal{G}}) = -\tau(\mathrm{DD}(\mathcal{G}))
\end{equation}
for all bundle gerbes $\mathcal{G}$ with connection over $X$, where $\mathrm{c}_1$ denotes the first Chern class of a principal $\ueins$-bundle, and  $\tau$ is the transgression homomorphism \erf{eq:tau}, see \cite{waldorf5}.

Applying the transgression functor to the bundle gerbe $\mathcal{G}$ over  $G$, we obtain a Fréchet principal $\ueins$-bundle $\tr_{\mathcal{G}}$ over the loop group $LG$. Because transgression is functorial and monoidal, the multiplicative structure $\mathcal{M}$ on  $\mathcal{G}$ transgresses  to a bundle isomorphism
\begin{equation*}
\alxydim{@C=1.6cm}{\pr_1^{*}\tr_{\mathcal{G}} \otimes \pr_2^{*}\tr_{\mathcal{G}} \ar[r]^-{\tr_{\mathcal{M}}} & m^{*}\tr_{\mathcal{G}} \otimes \tr_{\mathcal{I}_{\rho}} \cong m^{*}\tr_{\mathcal{G}}}
\end{equation*}
over $LG \times LG$, inducing a binary operation on the total space  $\tr_{\mathcal{G}}$ that covers the group structure of $LG$.  The mere existence of the associator $\alpha$ for the multiplicative structure $\mathcal{M}$ implies the commutativity of a diagram in the  category $\hc 1 \ugrbcon-(G^3)$, which implies under transgression the associativity of the binary operation $\tr_{\mathcal{M}}$. 

\begin{theorem}[{{\cite[Theorem 3.1.7]{waldorf5}}}]
The associative binary operation $\tr_{\mathcal{M}}$ equips $\tr_{\mathcal{G}}$  with the structure of a Fréchet Lie group,   making up a central extension
\begin{equation*}
1 \to \ueins \to \tr_{\mathcal{G}} \to LG \to 1\text{.}
\end{equation*}
\end{theorem}

\begin{example}
Consider again a compact, simple and simply-connected Lie group $G$, equipped with the basic gerbe $\gbas$ and  its unique multiplicative structure. We get from Equation \erf{eq:trprop} 
\begin{equation*}
\mathrm{c}_1(\tr_{\gbas})=-\tau(\mathrm{DD}(\gbas))=-\tau(\gamma)\text{.}
\end{equation*}
This means that $\tr_{\gbas}$ is the universal central extension of $LG$, see \cite{pressley1}.
\end{example}

In the following two subsections we  discuss additional structures on the central extension $\tr_{\gbas}$, which we  find by analyzing the image of the transgression functor $\tr$.

\subsection{Connections and splittings}

\label{sec:connsplits}

The principal $\ueins$-bundles in the image of the transgression functor $\tr$ of \erf{eq:tr} are canonically equipped with connections \cite{brylinski1}. In other words, transgression is actually a functor
\begin{equation}
\label{eq:trcon}
\tr: \hc 1 \ugrbcon X \to \ubuncon {LX}
\end{equation}
to the category of Fréchet principal $\ueins$-bundles \emph{with connection}. It satisfies  
\begin{equation}
\label{eq:trcurv}
\mathrm{curv}(\tr_{\mathcal{G}}) =- \tau_{\Omega}(\mathrm{curv}(\mathcal{G}))\text{,}
\end{equation}
where $\tau_{\Omega}$ is the differential form counterpart of the transgression homomorphism \erf{eq:tau}:
\begin{equation}
\label{eq:difftrans}
\tau_{\Omega}: \Omega^n(X) \to \Omega^{n-1}(LX): \omega \mapsto \int_{S^1}\ev^{*}\omega\text{;}
\end{equation}
it integrates the pullback of a differential form along the evaluation map $\ev:S^1 \times LX \to X$ over the factor $S^1$.
If $\omega\in\Omega^2(X)$ is a 2-form, and $\mathcal{I}_{\omega}$ is the associated trivial bundle gerbe with connection, then the above-mentioned canonical trivialization of $\tr_{\mathcal{I}_{\omega}}$ has covariant derivative $\tau_{\Omega}(\omega)\in\Omega^1(LX)$ \cite[Lemma 3.6]{waldorf13}. 

We continue the analysis of the central extension $\tr_{\mathcal{G}}$ of $LG$ obtained by transgression of a multiplicative bundle gerbe $\mathcal{G}$ over a Lie group $G$ with connection relative to the differential forms $H$ and $\rho$. As we have now lifted the transgression functor $\tr$ to the category of bundles \emph{with connection}, it follows that the central extension $\tr_{\mathcal{G}}$ has  a connection.

In the following we denote the central extension  by  $\lghat$, and we denote the connection  by $\nu\in\Omega^1(\lghat)$. According to Equation \erf{eq:trcurv} it has curvature $\mathrm{curv}(\nu)\eq-\tau_{\Omega}(H)$.
\begin{comment}
Since $H$ represents the generator $\gamma\in \h^3(G,\Z)$, $\mathrm{curv}(\nu)$ represents $\mathrm{c}_1(\lghat )\in\h^2(LM,\Z)$. 
\end{comment}
For us, the most important feature of the connection $\nu$  is that it is \emph{not strictly compatible} with the group structure of $\lghat$. Indeed, looking again at the transgression of the isomorphism $\mathcal{M}$, but now in the setting with connections, we obtain a connection-preserving bundle isomorphism
\begin{equation*}
\alxydim{@C=1.6cm}{\pr_1^{*}\tr_{\mathcal{G}} \otimes \pr_2^{*}\tr_{\mathcal{G}} \ar[r]^-{\tr_{\mathcal{M}}} & m^{*}\tr_{\mathcal{G}} \otimes \tr_{\mathcal{I}_{\rho}} \cong m^{*}\tr_{\mathcal{G}} \otimes \trivlin_{\varepsilon_{\nu}}\text{,}}
\end{equation*}
where $\trivlin_{\varepsilon_{\nu}}$ is the trivial $\ueins$-bundle over $LG$ equipped with the connection 1-form $\varepsilon_{\nu} \df \tau_{\Omega}(\rho)\in \Omega^1(LG \times LG)$. In terms of the connection 1-form $\nu$ and the group structure defined by the underlying bundle morphism, this can be expressed~as
\begin{equation}
\label{eq:eps}
 \nu_{\tilde\tau_1}(\tilde X_1) + \nu_{\tilde \tau_2}(\tilde X_2)=\nu_{\tilde \tau_1 \tilde \tau_2}(\tilde \tau_1 \tilde X_2 + \tilde X_1 \tilde \tau_1) + \varepsilon_{\nu}|_{\tau_1,\tau_2}(X_1,X_2)
\end{equation}
for elements $\tilde\tau_1,\tilde\tau_2\in \lghat$ projecting to loops $\tau_1,\tau_2\in LG$, and tangent vectors $\tilde X_1\in T_{\tilde\tau_1}\lghat$ and $\tilde X_2\in T_{\tilde\tau_2}\lghat$ projecting to $X_1 \in T_{\tau_1}LG$ and $X_2\in T_{\tau_2}LG$, respectively. The 1-form $\varepsilon_{\nu}$ can be computed explicitly from the given 2-form $\rho$ of \erf{eq:rho},
\begin{equation}
\label{eq:epsilonex}
\varepsilon_{\nu}|_{\tau_1,\tau_2}(X_1,X_2)= \int_{S^1} \left\lbrace \left \langle   \tau_1(z)^{-1}\partial_z\tau_1(z), X_2(z)\tau_2(z)^{-1} \right \rangle   - \left \langle  \tau_1(z)^{-1}X_1(z), \partial_z\tau_2(z)\tau_2(z)^{-1}  \right \rangle \right\rbrace \mathrm{d}z\text{.} 
\end{equation}
\begin{comment}
The full calculation is 
\begin{eqnarray}
&&\hspace{-0.8cm}\varepsilon_{\nu}|_{\tau_1,\tau_2}(X_1,X_2) \\&=& \int_{S^1} \rho_{\tau_1(z),\tau_2(z)}((\partial_z \tau_1(z),\partial_z\tau_2(z)),(X_1(z),X_2(z)))\mathrm{d}z
\nonumber\\ &=& \int_{S^1} \left \langle   \tau_1(z)^{-1}\partial_z\tau_1(z), X_2(z)\tau_2(z)^{-1} \right \rangle \mathrm{d}z -\int_{S^1} \left \langle  \tau_1(z)^{-1}X_1(z), \partial_z\tau_2(z)\tau_2(z)^{-1}  \right \rangle \mathrm{d}z\text{.} \nonumber
\end{eqnarray}
\end{comment}
Here, and in the following, we regard a tangent vector $X \in T_{\tau}LG$ as a section of $TG$ along $\tau$, i.e. as a smooth map $X:S^1 \to TG$ such that $X(z)\in T_{\tau(z)}G$, see \cite{pressley1}. 
\begin{comment}
We  also need the formulas
\begin{equation}
\label{eq:nugh}
\nu_{gh}(gv)=\nu_h(v)-\varepsilon_{\nu}|_{g,h}(0,v)
\end{equation}
and
\begin{equation*}
\nu_{gh}(wh)=\nu_g(w)-\varepsilon_{\nu}|_{g,h}(w,0)\text{.}
\end{equation*}
These imply
\begin{eqnarray*}
\nu_{ghg^{-1}}(gvg^{-1})&=&\nu_{hg^{-1}}(vg^{-1}) - \varepsilon_{\nu}|_{g,hg^{-1}}(0,vg^{-1})\\&=&\nu_h(v)-\varepsilon_{\nu}|_{h,g^{-1}}(v,0)- \varepsilon_{\nu}|_{g,hg^{-1}}(0,vg^{-1})
\end{eqnarray*}
\end{comment}

In general, a connection on a central extension induces a splitting of the  Lie algebra extension
\begin{equation}
\label{eq:liealgext}
\alxydim{}{0  \ar[r] & \R \ar[r] & \raisebox{0em}[0em][0.34em]{} \widetilde{L\mathfrak{g}}  \ar[r]^-{p_{*}} & L\mathfrak{g} \ar[r] & 0 \text{,}} 
\end{equation} 
i.e. a linear map $\split: L\mathfrak{g} \to \widetilde{L\mathfrak{g}}$ such that $p_{*}\circ \split=\id_{L\mathfrak{g}}$.
Indeed, the connection $\nu$ determines a horizontal subspace $H_1^{\nu}\lghat\subset T_1\lghat$ such that $p_{*}\maps H_1^{\nu}\lghat \to L\mathfrak{g}$ is an isomorphism. For $X\in L\mathfrak{g}$ we let $\split(X)\in H^{\nu}_1\lghat$ be its preimage under $p_{*}$. An equivalent definition that uses the connection 1-form $\nu$ directly is to first choose any lift $\tilde X\in \widetilde{L\mathfrak{g}}$ of $X$ and then define $s(X) := \tilde X - \nu(\tilde X)$. 
\begin{comment}
If $\hat X'$ is another lift, then $\hat X'=\hat X+Y$ for $Y\in \R$. Then,
\begin{equation*}
\hat X' - \nu(\hat X') = \hat X+Y-\nu(\hat X)-\nu(Y)=\hat X - \nu(\hat X)\text{,}
\end{equation*}
since $\nu(Y)=Y$ because $\nu$ is a connection and $Y$ is vertical. Thus, the definition of $\split$ is independent of the choice of $\hat X$. It is clearly linear, and $p_{*}(\split(X))=X$. Notice that $\nu(\split(X))=0$, so that $\split$ coincides with the horizontal lift of $X$. \end{comment}

Given the splitting $\split$ determined by the connection $\nu$ one can define the map
\begin{equation*}
Z: LG \times L\mathfrak{g} \to \R
\quomma
Z(\tau,X) := \mathrm{Ad}_{\tau}^{-1}(\split(X)) - s(\mathrm{Ad}_{\tau}^{-1}(X))\text{.}
\end{equation*}
\begin{comment}
It measures the error for $s$ being an intertwiner for the adjoint actions of $LG$ on $L\mathfrak{g}$ and $\widetilde{L\mathfrak{g}}$. 
\end{comment}
\begin{comment}
We have $Z(1,X)=0$ and 
\begin{eqnarray*}
Z(\tau_1\tau_2,X) &=& \mathrm{Ad}_{\tau_1\tau_2}^{-1}(\split(X)) - s(\mathrm{Ad}_{\tau_1\tau_2}^{-1}(X)) 
\\
&=& \mathrm{Ad}^{-1}_{\tau_2} (\mathrm{Ad}_{\tau_1}^{-1}(\split(X))) - s(\mathrm{Ad}^{-1}_{\tau_2}(\mathrm{Ad}_{\tau_1}^{-1}(X))) \\
&=& \mathrm{Ad}^{-1}_{\tau_2} (Z(\tau_1,X)+s(\mathrm{Ad}^{-1}_{\tau_1}(X))) - s(\mathrm{Ad}^{-1}_{\tau_2}(\mathrm{Ad}_{\tau_1}^{-1}(X))) \\
&=& Z(\tau_1,X)+\mathrm{Ad}^{-1}_{\tau_2} (s(\mathrm{Ad}^{-1}_{\tau_1}(X))) - s(\mathrm{Ad}^{-1}_{\tau_2}(\mathrm{Ad}_{\tau_1}^{-1}(X))) \\
&=& Z(\tau_1,X)+ Z(\tau_2,\mathrm{Ad}^{-1}_{\tau_1}(X))\text{.} \end{eqnarray*}
In particular
\begin{equation*}
0=Z(1,X)=Z(\tau\tau^{-1},X)=Z(\tau,X)+ Z(\tau^{-1},\mathrm{Ad}^{-1}_{\tau}(X))\text{,}
\end{equation*}
i.e.
\begin{equation}
\label{eq:Zinv}
Z(\tau^{-1},X) = -Z(\tau, \mathrm{Ad}_{\tau}(X))\text{.}
\end{equation}
\end{comment}

\begin{lemma}
\label{lem:Z}
$\displaystyle Z(\tau,X) = 2\int_{S^1}\left \langle \tau^{*}\bar\theta,X  \right \rangle$.
\end{lemma}

\begin{proof}
We express $Z$  in terms of the error 1-form $\varepsilon_{\nu}$ of the connection $\nu$. 
Consider $\tilde\tau\in \lghat$ and $\tilde X \in \widetilde{L\mathfrak{g}}$ projecting to $\tau\in LG$ and $X\in L\mathfrak{g}$, respectively. Then,
\begin{multline*}
Z(\tau,X) = \mathrm{Ad}_{\tilde \tau}^{-1}(\split(X))-\split(\mathrm{Ad}_\tau^{-1}(X))
= \tilde \tau^{-1}(\tilde X - \nu(\tilde X))\tilde \tau - \tilde \tau^{-1}\tilde X\tilde \tau + \nu(\tilde \tau^{-1}\tilde X\tilde \tau)
\\= - \nu(\tilde X) + \nu(\tilde\tau ^{-1}\tilde X\tilde \tau)
\stackrel{\erf{eq:eps}}{=} -\varepsilon_{\nu}|_{1,\tau}(X,0) -\varepsilon_{\nu}|_{\tau ^{-1},\tau}(0, X\tau)\text{.}
\end{multline*}
With \erf{eq:epsilonex} we see that these two terms are equal and add up to the claimed formula.
\begin{comment}
\begin{eqnarray*}
\varepsilon_{\nu}|_{1,\tau}(X,0) &=&-\int_0^1 \left \langle  X(z), \partial_z\tau(z)\tau(z)^{-1}  \right \rangle \;\mathrm{d}z
\\
\varepsilon_{\nu}|_{\tau^{-1},\tau}(0, X\tau) &=& \int_0^1 \left \langle \tau(z)\partial_z\tau(z)^{-1}, X(z) \right \rangle\;\mathrm{d}z
\end{eqnarray*}
\end{comment}
\begin{comment}
Explicitly, $Z(\tau,X)= 2 \int_{0}^1 \left \langle  \partial_z\tau(z)\tau(z)^{-1},X(z)  \right \rangle \mathrm{d}z$.
\end{comment}
\end{proof}

The formula
\begin{equation*}
\omega(X,Y) := \left.\frac{\mathrm{d}}{\mathrm{d}t}\right|_0 Z(e^{-tX},Y) = [\split(X),\split(Y)] - \split([X,Y])
\end{equation*}
defines a 2-cocycle $\omega$  for the Lie algebra cohomology of $L\mathfrak{g}$ with coefficients in the trivial module $\R$, and classifies the  Lie algebra extension. From Lemma \ref{lem:Z} we get the  following.

\begin{lemma}
\label{lem:omega}
$\displaystyle\omega(X,Y) =2\int_{S^1}\left \langle X,\mathrm{d}Y  \right \rangle$.
\end{lemma}

\begin{comment}
The calculation is:
\begin{multline}
\label{eq:omega}
\omega(X,Y) =  2 \left.\frac{\mathrm{d}}{\mathrm{d}t}\right|_0  \int_{0}^1 \left \langle  -t\partial_z X(z),Y(z)  \right \rangle \mathrm{d}z\\=-2 \int_{0}^1 \left \langle   \partial_z X(z),Y(z) \right \rangle \mathrm{d}z=2 \int_{0}^1 \left \langle   X(z),\partial_z Y(z) \right \rangle \mathrm{d}z =2\int_{S^1}\left \langle X,\mathrm{d}Y  \right \rangle\text{.}
\end{multline}
\end{comment}

Up to the prefactor (which can always be absorbed into the normalization of the bilinear form $\left \langle  -,- \right \rangle$) this is the standard cocycle on the loop algebra, see \cite[Section 4.2]{pressley1}.
Note that the cocycle $\omega$ is not invariant; instead it satisfies \cite[Lemma 5.8 (b)]{gomi3}
\begin{equation}
\label{eq:omegainv}
\omega(\mathrm{Ad}_{\tau}^{-1}(X),\mathrm{Ad}_{\tau}^{-1}(Y))=
\omega(X,Y) + Z(\tau,[X,Y])\text{.}
\end{equation}
\begin{comment}
This is an abstract calculation:
\begin{eqnarray*}
\omega(\mathrm{Ad}_{\tau}^{-1}(X),\mathrm{Ad}_{\tau}^{-1}(Y)) &=&
[\split(\mathrm{Ad}_{\tau}^{-1}(X)),\split(\mathrm{Ad}_{\tau}^{-1}(Y))] - \split([\mathrm{Ad}_{\tau}^{-1}(X),\mathrm{Ad}_{\tau}^{-1}(Y)]) \\
&=& 
[\mathrm{Ad}_{\tau}^{-1}(\split (X))-Z(\tau,X),\mathrm{Ad}_{\tau}^{-1}(\split(Y))-Z(\tau,Y)] \\&&
- \split(\mathrm{Ad}_{\tau}^{-1}([X,Y])) \\&=&
\mathrm{Ad}_{\tau}^{-1}([\split (X),\split(Y)])- \mathrm{Ad}_{\tau}^{-1}(\split([X,Y]))+Z(\tau,[X,Y])
\\&=&
\omega(X,Y) + Z(\tau,[X,Y])\text{.} 
\end{eqnarray*}
\end{comment}

It is well-known that a given splitting $\split$ of a Lie algebra extension  \erf{eq:liealgext} induces, conversely, a connection $\nu_\split$ on the central extension $\lghat$, given by the formula 
\begin{equation*}
\nu_{\split} = \tilde\theta - \split \left ( p^{*}\theta  \right ) \in\Omega^1(\lghat)\text{,}
\end{equation*}
where $\tilde\theta$ stands for the left-invariant Maurer-Cartan form on $\lghat$.
Its curvature is given by $-\frac{1}{2}\omega(\theta\wedge\theta)\in \Omega^2(LG)$,
see e.g. \cite[Lemma 5.4]{gomi3}. We will later have to compare the original connection $\nu$ with the connection $\nu_{\split}$   determined by $\split$ and hence indirectly by $\nu$. For this purpose, we consider the 1-form $\beta\in \Omega^1(LG)$ given by the formula 
\begin{equation*}
\beta_{\tau}(X) :=\int_0^1 \left \langle \tau(z)^{-1}\partial_z\tau(z), \tau(z)^{-1}X(z)   \right \rangle\;\mathrm{d}z
\end{equation*} 
for $\tau\in LG$ and $X\in T_{\tau}LG$.

\begin{lemma}
\label{lem:beta}
The connection $\nu_{\split}$ is obtained by shifting the connection $\nu$ by  $\beta$, i.e.
$\nu_{\split}  = \nu + \beta$. In particular, the curvatures obey the following relation: 
\begin{equation}
\label{eq:curvatures}
-\frac{1}{2}\omega(\theta\wedge\theta)=\mathrm{curv}(\nu)+\mathrm{d}\beta\text{.}
\end{equation}
\end{lemma}

\begin{proof}
For a tangent vector $\tilde X \in T_{\tilde\tau}\lghat$ we obtain from the definitions of the connection $\nu_{\split}$ and the splitting $\split$ that $\nu_{\split}(\tilde X)= \nu(\tilde\tau^{-1}\tilde X)$. Using the multiplicativity law \erf{eq:eps} for the connection $\nu$, we get $\nu(\tilde\tau^{-1}\tilde X)=\nu(\tilde X)- \varepsilon_{\nu}|_{\tau^{-1},\tau}(0,X)$. Looking at the explicit expression \erf{eq:epsilonex}, we see that $- \varepsilon_{\nu}|_{\tau^{-1},\tau}(0,X)=\beta_{\tau}(X)$.
\begin{comment}
The complete calculation is
\begin{eqnarray*}
\nu_{\split}(\hat X) &=& \tilde\theta(\hat X)-\split(p^{*}\theta(\hat X)) \\ &=&  \tilde\tau^{-1}\hat X - \split(\tau^{-1}X)
\\ &=&  \tilde\tau^{-1}\hat X - \tilde\tau^{-1} \hat X+ \nu(\tilde \tau^{-1}\hat X)
\\ &=&  \nu_1(\tilde\tau ^{-1}\hat X)
\\&\stackerf{eq:nugh}{=}& \nu_{\tilde\tau}(\hat X)- \varepsilon_{\nu}|_{\tau^{-1},\tau}(0,X)
\\&=& \nu(\hat X)+\beta(X)
\end{eqnarray*}
\end{comment}
\end{proof}

\label{sec:superficiality}

The connection $\nu$  on  $\lghat$ has an interesting property which  distinguishes it from other connections on $\lghat$, in particular from the  connection $\nu_{\split}$. The property is that $\nu$ is \emph{superficial}. In order to explain this, we  fix the following notation: if $\tau \in LLX$ is a loop in the loop space of a smooth manifold $X$, then by $\tau^{\vee}:S^1 \times S^1 \to X$ we denote the \quot{adjoint} map defined by $\tau^{\vee}(z_1,z_2)  := \tau(z_1)(z_2)$. We  use the following terminology: a map $f:X \to Y$ between smooth manifolds is said to be of rank $k$ if its differential $\mathrm{d}f_x$ has at most rank $k$ for all $x\in X$. The map $f$ is called \emph{thin}, if it is of rank $\dim(X)-1$.

\begin{definition}[{{\cite[Definition 2.2.1]{waldorf10}}}]
\label{def:superficial}
A connection $\nu$ on a Fréchet principal $\ueins$-bundle  over the loop space $LX$ of a smooth manifold $X$ is called \emph{superficial}, if the following two conditions are satisfied:
\begin{enumerate}[(i)]

\item 
The holonomy of a loop $\tau \in LLX$ vanishes if  $\tau^{\vee}$ thin.

\item
Two loops $\tau,\tau'\in LLX$ have the same holonomy, if $\tau^{\vee}$ and $\tau'^{\vee}$ are  thin homotopic. 

\end{enumerate}
\end{definition}

By \cite[Corollary 4.3.3]{waldorf10} all connections in the image of the transgression functor \erf{eq:trcon} are superficial. This comes from the fact that the holonomy of such connections can be expressed in terms of the surface holonomy of the bundle gerbe $\mathcal{G}$ via the formula
\begin{equation*}
\hol{\nu}\tau = \hol{\mathcal{G}}{\tau^{\vee}}\text{.}
\end{equation*}
The surface holonomy of a  connection on a gerbe has the two properties (i) and (ii).

\begin{comment}
Since superficial connections pullback and can be tensorized, it follows that $\varepsilon\in\Omega^1(LG \times LG)$ is a superficial 1-form (in the sense that it is a superficial connection on the trivial $\ueins$-bundle over $LG \times LG$).
We may check this explicitly. Suppose we have a loop $\tau \in L(LG \times LG)$ with $\tau^{\vee}:S^1\times S^1 \to G \times G$ of rank one. Then,
\begin{eqnarray*}
\int \varepsilon_{\nu}|_{\tau(z)}(\partial_z\tau(z))\mathrm{d}z  &=& \int_{S^1\times S^1} \left \langle   \tau_1(z,w)^{-1}\partial_w\tau_1(z,w), \partial_z\tau_2(z,w)\tau_2(z,w)^{-1} \right \rangle \mathrm{d}w\mathrm{d}z \\&&-\int_{S^1\times S^1} \left \langle  \tau_1(z,w)^{-1}\partial_z\tau_1(z,w), \partial_w\tau_2(z,w)\tau_2(z,w)^{-1}  \right \rangle \mathrm{d}w\mathrm{d}z 
\end{eqnarray*}
That $\tau^{\vee}$ is thin means that $\partial_z\tau(z,w)$ and $\partial_w \tau(z,w)$ are linearly dependent. So there is -- for each $z,w$ -- a constant $\alpha_{z,w} \in \R$ such that either $\partial_z\tau_i(z,w)=\alpha\partial_w\tau_i(z,w)$ or $\partial_w\tau_i(z,w)=\alpha\partial_z\tau_i(z,w)$ for $i=1,2$. In both cases, the expression vanishes. The calculation shows, in fact, that the parallel transport of $\varepsilon$ along any thin path (not just around thin loops) vanishes.  
\end{comment}

\subsection{Thin structures and fusion}

\label{sec:thinfusion}

In this article, the most important aspect of superficial connections is that they induce  \emph{thin structures}, a kind of equivariance with respect to thin homotopies. We use \emph{diffeological spaces} as an auxiliary tool. In short, a diffeological space is a set $X$ with specified \emph{plots}: maps $c:U \to X$ defined on open subsets $U \subset  \R^n$, $n\geq 0$. There are full and faithful functors
\begin{equation*}
\man \incl \frech \incl \diff
\end{equation*}
that realize smooth manifolds and Fréchet manifolds as diffeological spaces with plots given by all smooth maps $c:U \to X$ from all open subsets $U\subset \R^n$ for all $n$.
\begin{comment}
That the  functor $\frech \to \diff$ is full and faithful is a claim of \cite{losik1} and also proved in detail in \cite[Theorem 4.10]{Wockel2013}.
\end{comment}
In almost all aspects relevant for this article, diffeological spaces behave exactly as smooth manifolds -- there are just more of them. For example, differential forms, principal bundles, and connections can be defined on diffeological spaces in a manner  consistent with above inclusions, see \cite{waldorf9}.

\begin{comment}
Thin structures can be defined for principal bundles over loop spaces, see \cite[Section 3]{waldorf11}. 
\end{comment}
If $X$ is a smooth manifold, we denote by $\thinpairs {LX} \subset LX \times LX$ the set  consisting of pairs $(\tau_1,\tau_2)$ of thin homotopic loops, i.e. there exists a  homotopy $h: [0,1] \times S^1 \to X$  of rank one. The set $\thinpairs {LX}$ carries a natural diffeology \cite[Section 3.1]{waldorf11}.
\begin{comment}
This diffeology is finer than just regarding it as a submanifold of $LX \times LX$ in the sense that it takes into account that thin homotopies can be chosen in smooth families.  
\end{comment}

\begin{definition}[{{\cite[Definition 3.1.1]{waldorf11}}}]
A \emph{thin homotopy equivariant structure} on   a Fréchet principal $\ueins$-bundle $P$ over $LX$ is a smooth bundle isomorphism
\begin{equation*}
d: \mathrm{pr}_1^{*}P \to \mathrm{pr}_2^{*}P
\end{equation*}
over $\thinpairs {LX}$ that satisfies the  cocycle condition
$d_{\tau_2,\tau_3} \circ d_{\tau_1,\tau_2} = d_{\tau_1,\tau_3}$
for any triple $(\tau_1,\tau_2,\tau_3)$ of thin homotopic loops.
\end{definition}

A bundle
morphism $\varphi:P_1 \to P_2$ between bundles with thin homotopy equivariant structures $d_1$ and $d_2$ is called \emph{thin}, if the diagram
\begin{equation}
\label{eq:thinmorphism}
\alxydim{@=\xypicst}{\mathrm{pr}_1^{*}P_1 \ar[d]_{d_1} \ar[r]^{\mathrm{pr}_1^{*}\varphi} & \mathrm{pr}_1^{*}P_2 \ar[d]^{d_2} \\ \mathrm{pr}_2^{*}P_1 \ar[r]_{\mathrm{pr}_2^{*}\varphi} & \mathrm{pr}_2^{*}P_2}
\end{equation}
of bundle morphisms over $\thinpairs{LX}$ is commutative.

Now suppose $P$ is equipped with a superficial connection $\omega$.
Property (i) implies that the parallel transport of $\omega$ along a path $\gamma:[0,1] \to LX$ between two loops $(\tau_1,\tau_2)\in \thinpairs{LX}$ is independent of the choice of the path (provided it is chosen so that $\gamma^{\vee}$ is thin). Thus, we have a well-defined map
\begin{equation*}
d^{\omega}_{\tau_1,\tau_2}: P_{\tau_1} \to P_{\tau_2}\text{.}
\end{equation*}
\begin{comment}
The smooth dependence of parallel transport  on the path, together with above-mentioned diffeology on $\thinpairs{LX}$ imply
 the following
\end{comment} 
 The maps $d^{\omega}_{\tau_1,\tau_2}$ form a thin homotopy equivariant structure \cite[Lemma 3.1.5]{waldorf11}. A thin homotopy equivariant structure $d$ is called a \emph{thin structure}, if there is a superficial connection $\omega$ with $d=d^{\omega}$.

Summarizing, a thin structure on a bundle $P$ over a loop space $LX$ is a consistent way of identifying its fibres over thin homotopic loops. As orientation-preserving diffeomorphisms of $S^{1}$ induce thin homotopies ($\diff^{+}(S^1)$ is connected), we have the following. 

\begin{proposition}[{{\cite[Proposition 3.1.2]{waldorf11}}}]
\label{prop:equiv}
A thin structure on a Fréchet principal $\ueins$-bundle $P$ over  $LX$ determines a
 $\diff^{+}(S^1)$-equivariant structure on $P$. 
\end{proposition}

We continue to discuss the central extension $\lghat$ obtained by transgression of a multiplicative bundle gerbe over $G$ with connection relative to the forms $H$ and $\rho$ given by \erf{eq:H} and \erf{eq:rho}. Since the connection $\nu$ on $\lghat$ is superficial, the central extension $\lghat$ is equipped with a thin structure $d^{\nu}$. 

\begin{proposition}
\label{prop:thinmult}
The thin  structure $d^{\nu}$ is  multiplicative in the sense that
\begin{equation*}
d^{\nu}_{\tau_0\gamma_0,\tau_1\gamma_1}(\tilde\tau \cdot \tilde\gamma)=d^{\nu}_{\tau_0,\tau_1}(\tilde\tau)\cdot d^{\nu}_{\gamma_0,\gamma_1}(\tilde\gamma)
\end{equation*}
for all $((\tau_0,\gamma_0),(\tau_1,\gamma_1))\in \thinpairs{L(G\times G)}$ and all $\tilde\tau,\tilde\gamma \in \lghat$ projecting to $\tau_0$ and $\gamma_0$, respectively. 
\end{proposition}

\begin{proof}
We consider the connection-preserving bundle morphism
\begin{equation*}
\alxydim{@C=1.6cm}{\pr_1^{*}\tr_{\mathcal{G}} \otimes \pr_2^{*}\tr_{\mathcal{G}} \ar[r]^-{\tr_{\mathcal{M}}} & m^{*}\tr_{\mathcal{G}} \otimes \tr_{\mathcal{I}_{\rho}} = m^{*}\tr_{\mathcal{G}} \otimes \trivlin_{\varepsilon_{\nu}}\text{,}}
\end{equation*}
that describes the relation between the group structure of $\lghat$ and the superficial connection $\nu$. Now we pass from superficial connections to thin structures, and observe that the thin structure on the trivial bundle $\trivlin_{\varepsilon_{\nu}}$ is the trivial one. This comes simply from the fact that the parallel transport of the connection $\varepsilon_{\nu}=\tau_{\Omega}(\rho)$ along a path $\gamma\in P(G \times G)$ can be expressed as an integral of the 2-form $\rho$ over $\gamma^{\vee}$. If $\gamma^{\vee}$ is of rank one, that integral vanishes; see \cite[Proposition 3.1.8]{waldorf11} for the full argument.  Thus, 
\begin{equation*}
\alxydim{@C=1.6cm}{\pr_1^{*}\tr_{\mathcal{G}} \otimes \pr_2^{*}\tr_{\mathcal{G}} \ar[r]^-{\tr_{\mathcal{M}}} & m^{*}\tr_{\mathcal{G}} }
\end{equation*}
is a thin bundle morphism over $\thinpairs{L(G \times G)}$. Now, diagram \erf{eq:thinmorphism} evaluated over the point 
$((\tau_0,\gamma_0),(\tau_1,\gamma_1))\in \thinpairs{L(G \times G)}$
gives the claimed identity. 
\end{proof}

We remark that $((\tau_0,\gamma_0),(\tau_1,\gamma_1))\in \thinpairs{L(G\times G)}$ means that there exists a thin path $(\tau,\gamma)$ in $L(G \times G)$ connecting $(\tau_0,\gamma_0)$ with $(\tau_1,\gamma_1)$. It is necessary, but not sufficient, that the paths $\tau$, $\gamma$, and $\tau\gamma$ in $LG$ are separately thin. 

\label{sec:fusion}

Finally, we come to another additional structure on the central extension
 $\lghat$: a \emph{fusion product}. By $PX$ we denote the set of paths in a smooth manifold $X$ with \quot{sitting instants}, i.e. smooth maps $\gamma\maps [0,1] \to X$ that are locally constant near the endpoints. Due to the sitting instants, $PX$ is not a Fréchet manifold, but still a nice diffeological space, with plots $c:U \to PX$ those maps whose adjoints $c^{\vee}:U \times [0,1] \to X$ are smooth. We denote by $PX^{[k]}$ the $k$-fold fibre product of $PX$ over the evaluation map $\ev\maps PX \to X \times X$, i.e. the diffeological space  of $k$-tuples of paths with a common initial point and a common end point. Due to the sitting instants, we have a well-defined and smooth map
\begin{equation*}
\lop\maps  \p X^{[2]} \to L X\maps  (\gamma_1,\gamma_2) \mapsto \prev{\gamma_2} \pcomp \gamma_1\text{,}
\end{equation*}
where $\pcomp$ denotes the path concatenation, and $\prev{\gamma}$ denotes the reversed path; see \cite[Section 2]{waldorf9} for a more detailed discussion.  For $ij \in \{12,23,13\}$, we denote by $\lop_{ij}$ the composition of $\lop$ with  the  projection $\pr_{ij}: PX^{[3]} \to PX^{[2]}$.

\begin{definition}[{{\cite[Definition 2.1.3]{waldorf10}}}]
\label{def:fusionproduct}
A \emph{fusion product} on a Fréchet principal $\ueins$-bundle $P$ over the loop space $LX$ of a smooth manifold $X$ is a smooth bundle morphism
\begin{equation*}
\lambda: \lop_{23}^{*}P \otimes \lop_{12}^{*}P \to {\lop_{13}^{*}}P
\end{equation*}
over $PX^{[3]}$
that is associative in the sense that 
\begin{equation*}
\lambda(\lambda(p_{34} \otimes p_{23} ) \otimes p_{12} ) = \lambda(p_{34}\otimes \lambda(p_{23} \otimes p_{12}))
\end{equation*}
for all $p_{ij} \in P_{\gamma _i \lop \gamma_j}$ and all $(\gamma_1,\gamma_2,\gamma_3,\gamma_4) \in PX^{[4]}$.
\end{definition}

A morphism $\varphi: P_1 \to P_2$ between principal $\ueins$-bundles over $LX$ equipped with fusion products $\lambda_1$ and $\lambda_2$, respectively, is called \emph{fusion-preserving} if the diagram
\begin{equation*}
\alxydim{@=\xypicst}{\lop_{23}^{*}P_1 \otimes \lop_{12}^{*}P_1 \ar[r]^-{\lambda_1} \ar[d]_{\lop_{23}^{*}\varphi \otimes \lop_{12}^{*}\varphi} &  {\lop_{13}^{*}}P \ar[d]^{{\lop_{13}^{*}}\varphi} \\ \lop_{23}^{*}P_2 \otimes \lop_{12}^{*}P_2 \ar[r]_-{\lambda_2} & {\lop_{13}^{*}}P_2 }
\end{equation*}
of bundle morphisms over $PX^{[3]}$ is commutative.

If $P$ is equipped with a fusion product $\lambda$, then a connection $\nu$ is called \emph{fusive}, if the following  conditions are satisfied:
\begin{enumerate}[(i)]

\item 
\label{def:fusionconnection:1}
The fusion product $\lambda$ is a \emph{connection-preserving} bundle morphism over $PX^{[3]}$.

\item
\label{def:fusionconnection:2}
The rotation by an angle of $\pi$ is an orientation-preserving diffeomorphism of $S^1$ and induces a diffeomorphism $r_{\pi}:LX \to LX$. The $\diff^{+}(S^1)$-equivariant structure of Proposition \ref{prop:equiv} provides a lift $r_{\pi}^{d^{\nu}}:P \to P$. We demand that the condition
\begin{equation*}
\lambda(r^{d^{\nu}}_{\pi}(p_{12}) \otimes r^{d^{\nu}}_{\pi}(p_{23})) = r^{d^{\nu}}_{\pi} (\lambda(p_{23} \otimes p_{12}))
\end{equation*}
is satisfied for all $p_{12} \in P_{\gamma_1 \lop \gamma_2}$, $p_{23} \in P_{\gamma_2 \lop \gamma_3}$, and $(\gamma_1,\gamma_2) \in PX^{[2]}$. 

\end{enumerate}

In \cite{waldorf10} a category $\ufusbunconsf {LX}$
is considered with objects the  principal $\ueins$-bundles over $LX$ equipped with fusion products and superficial fusive connections, and morphisms the fusion-preserving, connection-preserving bundle morphisms. By a construction performed in \cite[Section 4.2]{waldorf10}, the transgression functor \erf{eq:trcon} lifts into this category:
\begin{equation}
\label{eq:trfull}
\tr: \hc 1 \ugrbcon X \to \ufusbunconsf {LX}\text{.}
\end{equation}
Before we return to the central extension $\lghat$, we relate fusion products to thin structures. 

\begin{definition}
\label{def:fusionthin}
Let $P$ be a principal $\ueins$-bundle $P$ over $LX$ with a fusion product $\lambda$. A thin structure $d$ on $P$ is called  \emph{fusive} with respect to $\lambda$, if there exists a superficial fusive connection $\nu$ on $P$ such that $d=d^{\nu}$. 
\end{definition}

In particular, the fusion product $\lambda$ is a thin bundle morphism with respect to a fusive thin structure. In \cite{waldorf11} a category $h\ufusbunth {LX}$ is considered with objects the   principal $\ueins$-bundles over $LX$ equipped with fusion products and fusive thin structures, and morphisms the homotopy classes of  fusion-preserving, thin bundle morphisms.

The two categories $\ufusbunconsf {LX}$ and $h\ufusbunth{LX}$ are loop space analogues of the  categories $\hc 1 \ugrbcon X$ and $\hc 1 \ugrb X$ of bundle gerbes with and without connections over $X$, respectively. The procedure of inducing a thin structure from a superficial connection (and projecting to the homotopy class of a morphism) defines a functor
\begin{equation*}
\ufusbunconsf{LX} \to h\ufusbunth{LX}\text{;}
\end{equation*}
it is the loop space analogue of the 2-functor \erf{eq:forgetconnections} that passes from gerbes with connection to gerbes without connections. These analogies are the content of the following theorem, which is the main result of the series of articles \cite{waldorf9,waldorf10,waldorf11}.

\begin{theorem}
\label{th:equiv}
Let $X$ be a connected smooth manifold. 
There is a strictly commutative diagram
\begin{equation*}
\alxydim{@=\xypicst}{\ufusbunconsf {LX} \ar[d] \ar[r] & \hc 1 \ugrbcon X \ar[d] \\ h\ufusbunth{LX} \ar[r]  &\hc 1 \ugrb X }
\end{equation*}
of monoidal categories and functors, natural in $X$,  whose horizontal functors are monoidal equivalences of categories. 
\end{theorem}

The functor in the first row of the diagram is inverse to the transgression functor $\tr$ of \erf{eq:trfull}, i.e. the two functors form an equivalence of categories. The functor in the second row is  essentially surjective, full and faithful, but has no canonical inverse functor.   The two functors are called \emph{regression} \cite[Section 5]{waldorf10}.

Let us now return to the discussion of the central extension $\lghat$ defined by transgression of a multiplicative bundle gerbe $\mathcal{G}$ over $G$. According to above discussion, $\lghat$ is equipped with a fusion product, which we denote by $\lambda_{\mathcal{G}}$. 
\begin{comment}
\begin{example}
We recall that in case that $G$ is compact, simple and simply-connected, and $\mathcal{G}=\gbas$ the basic bundle gerbe, the central extension $\lghat$ is the universal one. The universal central extension has another  model based on smooth maps $f: D^2 \to G$ defined on the disc, subject to an equivalence relation involving the Wess-Zumino term on $G$, see \cite{mickelsson1}. In that model, one can explicitly write down the fusion product $\lambda_{\mathcal{G}}$, see \cite{Waldorfa}.
\end{example}
\end{comment}
The connection $\nu$ on $\lghat$ and the induced thin structure $d^{\nu}$ are fusive.  Since transgression is a functor, the multiplication $\tr_{\mathcal{M}}$ is fusion-preserving. This can be rephrased as follows.

\begin{lemma}[{{}}]
\label{lem:multfus}
The fusion product on $\lghat$ is multiplicative in the sense that
\begin{equation}
\label{eq:multfus}
\lambda_{}(p_{23} \otimes p_{12}) \cdot \lambda(p_{23}' \otimes p_{12}') = \lambda(p^{}_{23}p_{23}' \otimes p^{}_{12}p_{12}')
\end{equation}
for all elements $p_{ij},p_{ij}'\in \lghat$ projecting to loops ${\gamma_i\lop \gamma_j}$ and ${\gamma_i' \lop \gamma_j'}$, respectively, for all $(\gamma_1,\gamma_2,\gamma_3),(\gamma_1',\gamma_2',\gamma_3') \in PG^{[3]}$.
\end{lemma}
 
\begin{comment}
The transgression of $\mathcal{M}$ is a bundle isomorphism
\begin{equation*}
\phi: \pr_1^{*}\tr_{\mathcal{G}}\otimes \pr_2^{*}\tr_{\mathcal{G}} \to m^{*}\tr_{\mathcal{G}} \otimes \tr_{\mathcal{I}_{\rho}}\cong m^{*}\tr_{\mathcal{G}}\text{.}
\end{equation*}
over $LG \times LG$, in which the isomorphism $\tr_{\mathcal{I}_{\rho}}\cong \trivlin$ is fusion-preserving \cite[Lemma 3.6]{waldorf13}.  
\end{comment}

We have now listed all additional structures and properties of the  central extension $\lghat$ that arise from our approach using the transgression multiplicative bundle gerbes, and that we need in the following. In \cite{Waldorfc} we show how \quot{transgressive} central extensions can be characterized by fusion products and thin structures.

\setsecnumdepth{2}

\section{Thin fusion spin structures}
\label{sec:spinstructures}

In Section \ref{sec:fusionspin}
we first recall the definition of  spin structures on loop spaces following Killingback \cite{killingback1}. Based upon this definition we develop the notion of  thin fusion spin structures, which constitute our  loop space analogue for string structures. In Section \ref{sec:lifting} we prepare one part of the  proof of this analogy: we provide a lifting gerbe formulation for thin fusion spin structures.

\subsection{Versions of spin structures on loop spaces}

\label{sec:fusionspin}

Let $M$ be a spin manifold of dimension $n=3$ or $n>4$, so that $\spin n$ is compact, simple and simply-connected. We denote  by $\pi:FM\to M$ the spin-oriented frame bundle of $M$, which is a $\spin n$-principal bundle over $M$. 
Since $\spin n$ is connected, $LFM$ is a principal $\lspin n$-bundle over $LM$.  

\begin{definition}[{{\cite{killingback1}}}]
\label{def:spinstructure}
A \emph{spin structure on $LM$} is a lift of the structure group of the looped  frame bundle $LFM$ from $\lspin n$ to the universal central extension $\lspinhat n$.
\end{definition}

Thus, a spin structure on $LM$ is a pair $(\inf S,\sigma)$ of a Fréchet principal $\lspinhat n$-bundle $\inf S$ over $LM$ together with a smooth map $\sigma:\inf S \to LFM$ such that the diagram
\begin{equation*}
\alxydim{@C=\xypicst@R=1.3em}{\inf S \times \lspinhat n \ar[dd]_{\sigma \times p} \ar[r] & \inf S \ar[dd]_{\sigma} \ar[dr] \\ &&LM \\ LFM \times \lspin n \ar[r] & LFM \ar[ur] }
\end{equation*}
is commutative.  A morphism between spin structures $(\inf S_1,\sigma_1)$ and $(\inf S_2,\sigma_2)$ is a bundle morphism $\varphi:\inf S_1 \to \inf S_2$ such that $\sigma_1=\sigma_2 \circ \varphi$. Spin structures on $LM$ form a category that we denote by $\spst$. It is a module for the monoidal category $\ubun{LM}$ of Fréchet principal $\ueins$-bundles over $LM$, under an action functor
\begin{equation}
\label{eq:actionspst}
\ubun{LM} \times \spst \to \spst: (K,(\inf S,\sigma))\mapsto K\otimes (\inf S, \sigma)\text{.}
\end{equation}
Here, $K\otimes (\inf S, \sigma)$ is the spin structure with the $\lspinhat n$-bundle  $K\otimes \,\inf S :=(K \times_{LM} \inf S)/\,\ueins$  over $LM$, and the map $(k,s)\mapsto \sigma(s)$ to $LFM$. By Corollary \ref{co:spsttorsor} proved below, the action \erf{eq:actionspst} exhibits $\spst$ as a \emph{torsor} over $\ubun{LM}$ in the sense that the associated functor
\begin{equation*}
\ubun{LM} \times \spst \to \spst \times \spst: (K,(\inf S,\sigma))\mapsto (K\otimes (\inf S, \sigma),(\inf S, \sigma))
\end{equation*}
is an equivalence of categories.

The notion of a spin structure in the sense of Definition \ref{def:spinstructure} suffers from the fact that there are manifolds that are not string manifolds but  whose loop space admits spin structures \cite{Pilch1988}. The plan we follow in this article is to add additional conditions/structure to spin structures on loop spaces, in order to better reflect string structures on the base manifold. 
\begin{comment}
We recall from Section \ref{sec:fusion} that the central extension $\lspinhat n$ is equipped with a fusion product $\lambda_{\gbas}$. 
\end{comment}

If $(\inf S,\sigma)$ is a spin structure, then $\sigma:\inf S \to LFM$ is a principal $\ueins$-bundle under the $\ueins$-action obtained by restriction of the $\lspinhat n$-action. We use the notation $T_{\inf S}$ for explicit reference to this principal $\ueins$-bundle. Any morphism $\varphi:\inf S \to \inf S'$  between spin structures is a morphism $\varphi: T_{\inf S} \to T_{\inf S'}$ between the associated principal $\ueins$-bundles.  Under the action  \erf{eq:actionspst} we have 
\begin{equation}
\label{eq:taction}
T_{K \otimes\, \inf S}=L\pi^{*}K \otimes T_{\inf S}\text{.}
\end{equation}
\begin{comment}
Now we formulate additional structures for spin structures in terms of the associated principal $\ueins$-bundle $T_{\inf S}$.
\end{comment}

\begin{definition}[{{\cite[Definition 3.6]{Waldorfa}}}]
\label{def:fusionspinstructure}
A \emph{fusion product} on a spin structure $(\inf S, \sigma)$ is a fusion product $\lambda_{\inf S}$ on $T_{\inf S}$  such that the $\lspinhat n$-action on $\inf S$ is fusion-preserving:
\begin{equation*}
\lambda_{\inf S}(t_{23}\cdot \tilde\gamma_{23} \otimes t_{12} \cdot \tilde\gamma_{12}) = \lambda_{\inf S}(t_{23} \otimes t_{12}) \cdot \lambda_{\gbas}(\tilde\gamma_{23} \otimes \tilde\gamma_{12})\text{,}
\end{equation*}
for all
 $t_{12},t_{23} \in \inf S$ and $\tilde\gamma_{12},\tilde\gamma_{23} \in \lspinhat n$  such that the fusion products are defined. A morphism $\varphi:  \inf S \to \inf S'$ between spin structures with fusion products is called \emph{fusion-preserving} if the associated morphism $\varphi: T_{\inf S} \to T_{\inf S'}$ is fusion-preserving.
\end{definition}

Spin structures with fusion products form a category that we denote by $\spstfus$. Similar to the action functor \erf{eq:actionspst}, the category $\spstfus$ carries an action of the monoidal category $\ufusbun {LM}$ of Fréchet principal $\ueins$-bundles with fusion products, under which \erf{eq:taction} holds as an equation of bundles with fusion products.

The main result of the paper \cite{Waldorfa} was that a spin manifold $M$ is a string manifold if and only if its loop space $LM$ admits a spin structure with fusion product. Next we explain how to add thin structures into the picture in order to improve that result.
\begin{comment}
We recall from Section \ref{sec:superficiality} that a thin structure on a $\ueins$-bundle $P$ over a loop space $LX$ is a bundle morphism $d:\pr_1^{*}P \to \pr_2^{*}P$ over the diffeological space $\thinpairs{LX}\subset LX \times LX$. 
\end{comment}
We recall that the central extension $\lspinhat n$ is equipped with a thin structure $d^{\nu}$ induced from the superficial connection $\nu$.

\begin{definition}
\label{def:thinspin}
 A \emph{thin  structure} on a spin structure  $(\inf S, \sigma)$ is a thin structure $d$ on $T_{\inf S}$ such that
\begin{equation*}
d_{\tau_1\cdot \gamma_1,\tau_2\cdot \gamma_2}(t\cdot \tilde\gamma ) = d_{\tau_1,\tau_2}(t) \cdot d^{\nu}_{\gamma_1,\gamma_2}(\tilde\gamma)\text{.}
\end{equation*}
for all $((\tau_1,\gamma_1),(\tau_2,\gamma_2)) \in \thinpairs{L(FM \times \spin n)}$, all $t\in T_{\inf S}$ projecting to $\tau_1$, and all $\tilde\gamma\in\lspinhat n$ projecting $\gamma_1$.  A \emph{thin spin structure} is a spin structure together with a thin structure. A \emph{morphism between thin spin structures} is a morphism $\varphi\maps \inf S \to \inf S'$ between spin structures such that the induced morphism $\varphi: T_{\inf S} \to T_{\inf S'}$ is thin. \end{definition}

In this definition it is relevant to observe that $((\tau_1,\gamma_1),(\tau_2,\gamma_2)) \in \thinpairs{L(FM \times \spin n)}$ implies that $(\tau_1,\tau_2)\in\thinpairs{LFM}$, $(\gamma_1,\gamma_2)\in \thinpairs{\lspin n}$, and $(\tau_1\cdot\gamma_1,\tau_2\cdot \gamma_2)\in \thinpairs{LFM}$.

Thin spin structures form a category that we denote by $\spstth$. Based on the action functor \erf{eq:actionspst}, it carries an action of the monoidal category $\ubunth{LM}$ of Fréchet principal $\ueins$-bundles with thin structures, under which \erf{eq:taction} is an equality of bundles with thin structures. 
\begin{comment}
We have the following  consequence of the presence of a thin structure.
\end{comment}

\begin{proposition}
A thin  structure on a spin structure $(\inf S,\sigma)$ on $LM$ determines a  $\diff^{+}(S^1)$-equivariant structure on the principal $\lspinhat n$-bundle $\inf S$ over $LM$, such that the map  $\sigma:\inf S \to LFM$ is $\diff^{+}(S^1)$-equivariant.
\end{proposition}

\begin{proof}
We note that $LFM$ is obviously $\diff(S^1)$-equivariant as a $\lspin n$-bundle over $LM$, since $\diff^{+}(S^1)$ acts on $LFM$. By Proposition \ref{prop:equiv}, the thin structure on $T_{\inf S}$ lifts this action to $\inf S$. 
\end{proof}

Thin structures and fusion products for spin structures combine in the following way.

\begin{definition}
\label{def:thinfusionspin}
A \emph{thin fusion spin structure} on $LM$ is a spin structure $(\inf S,\sigma)$ with a fusion product $\lambda$ in the sense of Definition \ref{def:fusionspinstructure} and a thin structure $d$ in the sense of Definition \ref{def:thinspin}, such that $d$ is fusive with respect to $\lambda$ in the sense of Definition \ref{def:fusionthin}. \end{definition}

Particular care has to be taken with the correct notion of morphisms between thin fusion spin structures. 
\begin{comment}
If $(\inf S,\sigma,\lambda,d)$ and $(\inf S',\sigma',\lambda',d')$ are thin fusion spin structures, then an obvious notion of  morphism
\begin{equation*}
\varphi: (\inf S,\sigma,\lambda,d)
\to
(\inf S',\sigma',\lambda',d')
\end{equation*}
would be a morphism between spin structures such that the induced bundle morphism  $\varphi\maps T_{\inf S} \to T_{\inf S'}$ is thin and fusion-preserving. This leads to a category which carries an action of the monoidal category $\ufusbunth{LM}$ under which \erf{eq:taction} is an equality of thin fusion bundles over $LFM$.

We  now give a definition of morphisms between thin fusion spin structures that leads to a category which is equivalent to the category of string structures. We  prove this in a combination of results from the next subsection and Section \ref{sec:transgression}. 
First we recall the notion of a fusion homotopy \cite[Definitions 2.2.3 and 2.2.5]{waldorf9}. 
\end{comment}
If $X$ is a smooth manifold, a \emph{fusion map} $f: LX \to \ueins$ is a smooth map with the following properties:
\begin{enumerate}[(i)]

\item
If $\tau, \tau'\in LX$ are thin homotopic loops, then  $f(\tau)=f(\tau')$.

\item
If $(\gamma_1,\gamma_2,\gamma_3)\in PX^{[3]}$, then $f(\gamma_1\lop \gamma_2) \cdot f(\gamma_2\lop \gamma_3) = f(\gamma_1 \lop \gamma_3)$.
\end{enumerate}  
A \emph{fusion homotopy} is a smooth map $h: [0,1] \times LX \to \ueins$ such that $h_t: LX \to \ueins$ is a fusion map for all $t\in [0,1]$.

\begin{definition}
\label{def:thinfusionspinmorph}
Let  $(\inf S,\sigma,\lambda,d)$ and $(\inf S',\sigma',\lambda',d')$ be thin fusion spin structures. A morphism is a smooth map $\varphi: \inf S \to \inf S'$ satisfying the following conditions:
\begin{enumerate}[(i)]

\item 
$\sigma' \circ \varphi = \sigma$; in particular, $\varphi$ covers the identity on $LM$. 

\item
$\varphi$ is equivariant with respect to the $\ueins$-actions on $\inf S$ and $\inf S'$, i.e. it induces a morphism  $\varphi\maps T_{\inf S} \to T_{\inf S'}$ between $\ueins$-bundles over $LFM$.

\item
The bundle morphism $\varphi: T_{\inf S} \to T_{\inf S'}$ is fusion-preserving and thin. 

\item
$\varphi$ is fusion-homotopy-equivariant with respect to the $\lspinhat n$-action, i.e.  there exists  a fusion homotopy
$h: [0,1] \times LFM \times \lspin n \to \ueins$ with $h_0=1$ and 
\begin{equation}
\label{eq:proofcond2}
\varphi(t\cdot \tilde\tau)\cdot h_1(\beta,\tau) = \varphi(t)\cdot \tilde\tau
\end{equation}
for all $t\in \inf S$ and $\tilde\tau\in\lspinhat n$ over $\beta\in LFM$ and $\tau\in \lspin n$, respectively. 

\end{enumerate}
\end{definition}

\begin{comment}
We remark that morphisms between thin fusion spin structures only respect the $\lspinhat n$-action up to homotopy. So it might happen that two thin fusion spin structures are isomorphic, although the underlying spin structures are not isomorphic. 
\end{comment}

Definitions \ref{def:thinfusionspin} and \ref{def:thinfusionspinmorph} result in a category of thin fusion spin structures which we denote by $\spstthfus$. It carries an action of the monoidal category $\ufusbunth{LM}$.  

In the end, the category that is equivalent to the category of string structures on $M$ is the \emph{homotopy category} $h\spstthfus$, i.e. two morphisms $\varphi_0,\varphi_1: \inf S \to \inf S'$ become identified if there is a smooth map $h:[0,1] \times \inf S \to \inf S'$ with $h_0=\varphi_0$, $h_1=\varphi_1$, and $h_t$ is a morphism between thin fusion spin structures for all $t\in[0,1]$.
The homotopy category $h\spstthfus$ inherits an action of the homotopy category $h\ufusbunth{LM}$. As a consequence of Theorem \ref{th:trequiv}, this action exhibits $h\spstthfus$ as a torsor over $h\ufusbunth{LM}$.

\subsection{Lifting theory for  spin structures}

\label{sec:lifting}

As any lifting problem, spin structures on loop spaces can be described by a bundle gerbe, the \emph{spin lifting  gerbe} $\mathcal{S}_{LM}$ \cite{Carey1998}. We refer to \cite[Section 4.1]{Waldorfa} for a  detailed treatment. In short, the spin lifting  gerbe $\mathcal{S}_{LM}$ is the following bundle gerbe over $LM$:
\begin{enumerate}[(i)]

\item 
it has the surjective submersion $L\pi:LFM \to LM$.

\item
over the 2-fold fibre product $LFM^{[2]}$ it carries the Fréchet principal $\ueins$-bundle \begin{equation*}
\q:=L\delta^{*}\lspinhat n\text{,}
\end{equation*}
where $\delta: FM^{[2]} \to \spin n$ is the \quot{difference map} defined by $p'\cdot \delta(p,p')=p$.

\item
over the 3-fold fibre product $LFM^{[3]}$ it has the bundle gerbe product
\begin{equation*}
\mu: \pr_{23}^{*}\q \otimes \pr_{12}^{*}\q \to \pr_{13}^{*}\q: ((\beta_2,\beta_3,\tilde \tau_{23}) \otimes(\beta_1,\beta_2,\tilde \tau_{12}))\mapsto (\beta_1,\beta_3,\tilde\tau_{23}\cdot \tilde\tau_{12})
\end{equation*}
defined from the group structure of $\lspinhat n$. 
\begin{comment}
If $\tilde \tau_{12}$ is over $\delta(p_1,p_2)$ and $\tilde\tau_{23}$ is over $\delta(p_2,p_3)$ then $\tilde \tau_{23}\cdot \tilde\tau_{12}$ is over $\delta(p_2,p_3)\cdot \delta(p_1,p_2)=\delta(p_1,p_3)$. 

Alternatively, we recall that $\lspinhat n=\tr_{\gbas}$ is a multiplicative $\ueins$-bundle over $\lspin n$, i.e. we have this isomorphism
\begin{equation*}
\tr_{\mathcal{M}}: \pr_1^{*}\tr_{\gbas} \otimes \pr_2^{*}\tr_{\gbas} \to m^{*}\tr_{\gbas}
\end{equation*}
between bundles over $\lspin n\times \lspin n$. We consider the map 
\begin{equation*}
\delta_2: LFM^{[3]} \to \lspin n\times \lspin n
\end{equation*}
defined by $(\tau_3,\tau_2) \cdot \delta_2(\tau_1,\tau_2,\tau_3)=(\tau_2,\tau_1)$. We have $\pr_1\circ \delta_2=\delta\circ \pr_{23}$, $\pr_2\circ \delta_2=\delta\circ \pr_{12}$ and $m\circ \delta_2 = \delta\circ \pr_{13}$. So we simply define
\begin{equation*}
\mu := \delta_2^{*}\tr_{\mathcal{M}}\text{.}
\end{equation*}

\end{comment}
\end{enumerate}

The purpose of the lifting bundle gerbe is to provide a reformulation of spin structures in terms of trivializations of $\mathcal{S}_{LM}$. A trivialization of $\mathcal{S}_{LM}$ is by definition a pair $\mathcal{T}=(T,\kappa)$ consisting of a principal $\ueins$-bundle $T$ over $LFM$ and  a bundle isomorphism
\begin{equation*}
\kappa\maps \pr_2^{*}T\otimes  \q \to \pr_1^{*}T
\end{equation*} 
over $LFM^{[2]}$ such that the diagram
\begin{equation*}
\alxydim{@C=2cm@R=\xypicst}{\pr_3^{*}T\otimes \pr_{23}^{*}\q \otimes \pr_{12}^{*}\q  \ar[r]^-{\pr_{23}^{*}\kappa \otimes \id} \ar[d]_{\id \otimes \mu} & \pr_2^{*}T \otimes \pr_{12}^{*}\q \ar[d]^{\pr_{12}^{*}\kappa} \\ \pr_3^{*}T \otimes \pr_{13}^{*}\q \ar[r]_-{\pr_{13}^{*}\kappa} & \pr_1^{*}T}
\end{equation*}
of bundle morphisms over $LFM^{[3]}$ is commutative.
A morphism between trivializations $(T,\kappa)$ and $(T',\kappa')$ is a bundle isomorphism $\varphi: T \to T'$ such that the diagram
\begin{equation*}
\alxydim{@=\xypicst}{\pr_2^{*}T\otimes \q  \ar[r]^-{\kappa} \ar[d]_{\pr_2^{*}\varphi} & \pr_1^{*}T \ar[d]^{\pr_1^{*}\varphi} \\ \pr_2^{*}T' \otimes \q \ar[r]_-{\kappa'} & \pr_1^{*}T'}
\end{equation*}
is commutative. Trivializations of $\mathcal{S}_{LM}$ form a category $\triv{\mathcal{S}_{LM}}$, which is a module for the monoidal category $\ubun{LM}$ of Fréchet principal $\ueins$-bundles over $LM$ under the action functor
\begin{equation}
\label{eq:actiontrivspin}
\ubun{LM} \times \triv{\mathcal{S}_{LM}} \to \triv{\mathcal{S}_{LM}}: (K,\mathcal{T})\mapsto K \otimes \mathcal{T}\text{.}
\end{equation}
Here, the trivialization $K\otimes \mathcal{T}$ consists of the principal $\ueins$-bundle $L\pi^{*}K \otimes T$ over $LFM$ and of the bundle isomorphism $\id\otimes\kappa$. The action \erf{eq:actiontrivspin} exhibits $\triv{\mathcal{S}_{LM}}$ as a torsor over $\ubun{LM}$.
\begin{comment}
This can be seen by regarding  trivializations of $\mathcal{S}_{LM}$ as isomorphisms $\mathcal{S}_{LM} \to \mathcal{I}$ \cite[Section 3.1]{waldorf1}; but isomorphisms between two bundle gerbes over any smooth manifold $X$ form a torsor  over $\ubun{X}$.
\end{comment}
\begin{comment}
We recall that associated to a spin structure $(\inf S,\sigma)$ is a principal $\ueins$-bundle $T_{\inf S}$ over $LFM$.
\end{comment}
For a spin structure $(\inf S,\sigma)$ we have  a bundle isomorphism
\begin{equation*}
\kappa_{\inf S}: \pr_2^{*}T_{\inf S} \otimes \q \to \pr_1^{*}T_{\inf S}: (\beta_2,t) \otimes (\beta_1,\beta_2,\tilde\tau) \mapsto (\beta_1,t \cdot \tilde\tau)
\end{equation*}
over $LFM^{[2]}$,
where  $t\cdot \tilde\tau$ is  the $\lspinhat n$-action on $\inf S$. It is easy to see that $(T_{\inf S},\kappa_{\inf S})$ is a trivialization of the spin lifting gerbe $\mathcal{S}_{LM}$.
A morphism $\varphi: \inf S \to \inf S'$  between spin structures induces a morphism $\varphi: T_{\inf S} \to T_{\inf S'}$ between bundles over $LFM$, which is in fact a morphism $(T_{\inf S},\kappa_{\inf S}) \to (T_{\inf S'},\kappa_{\inf S'})$  between trivializations.
\begin{comment}
For the definition of $\kappa$ we have to check: if $t\in T$ sits over $\sigma(t)=\tau\in LFM$ and $q\in \lspinhat n$ sits over $\gamma\in \lspin n$, then $t\cdot q$ sits over $\sigma(t\cdot q)=\tau\cdot \gamma$, and $\delta(\tau\cdot\gamma,\tau)=\gamma$.

A morphism between spin structures is a morphism between trivializations.

The inverse of this construction goes as follows. If $(T,\kappa)$ is a trivialization, then put $\inf S := T$ with the projection $T \to LFM \to LM$, and $\sigma: \inf S \to LFM$ just the bundle projection $T \to LFM$. The $\lspinhat n$-action on $\inf S$ is defined by $t \cdot q := \kappa(t \otimes q)$. 
\end{comment}
As a consequence of a general theorem of Murray about lifting gerbes \cite{murray} we obtain the following result; also see \cite[Theorem 4.1.3]{Waldorfa}.

\begin{proposition}
\label{prop:lifting}
The assignment $(\inf S,\sigma)\mapsto (T_{\inf S},\kappa_{\inf S})$ establishes an equivalence of categories:
\begin{equation*}
\spst \cong \bigset{11em}{Trivializations of the spin lifting gerbe $\mathcal{S}_{LM}$}\text{.}
\end{equation*}
\end{proposition}

Formula \erf{eq:taction} shows that the equivalence of Proposition \ref{prop:lifting} is equivariant under the actions \erf{eq:actionspst} and \erf{eq:actiontrivspin} of $\ubun{LM}$.
In particular, we obtain the following consequence.

\begin{corollary}
\label{co:spsttorsor}
The category $\spst$ of spin structures on $LM$ is a torsor over the monoidal category $\ubun{LM}$.
\end{corollary}

\begin{comment}
Next we want to include fusion products into the  lifting gerbe-theoretic description of  spin structures. 
\end{comment}
The fusion product $\lambda_{\gbas}$ of the central extension $\lspinhat n$ pulls back along
the map $L\delta$ to a fusion product $\lambda_P := L\delta^{*}\lambda_{\gbas}$ on the $\ueins$-bundle $\q$ of the lifting gerbe $\mathcal{S}_{LM}$. The bundle gerbe product $\mu$ of $\mathcal{S}_{LM}$ is fusion-preserving according to Lemma \ref{lem:multfus}. 
\begin{comment}
In the terminology of \cite{Waldorfa}, this equips $\mathcal{S}_{LM}$ with an \emph{internal fusion product}. 
\end{comment}

Suppose $\mathcal{T}=(T,\kappa)$ is a trivialization of $\mathcal{S}_{LM}$. A fusion product $\lambda$ on $T$ is called \emph{compatible} if  the bundle morphism $\kappa$ is fusion-preserving (with respect to the fusion product $\lambda_{\q}$ on $\q$). A morphism $\varphi:\mathcal{T}_1 \to \mathcal{T}_2$ between two trivializations with fusion products is called \emph{fusion-preserving}, if it is fusion-preserving as a bundle morphism $\varphi: T_1 \to T_2$.

\begin{proposition}[{{\cite[Corollary 4.4.8]{Waldorfa}}}]
\label{prop:fusionlifting}
The assignment $(\inf S,\sigma,\lambda)\mapsto (T_{\inf S},\kappa_{\inf S},\lambda)$ establishes an equivalence of categories:
\begin{equation*}
\spstfus 
\cong
\bigset{12em}{Trivializations of the spin lifting gerbe $\mathcal{S}_{LM}$ with compatible fusion products}\text{.}
\end{equation*}
\end{proposition}

\begin{comment}
This equivalence is based on  Proposition \ref{prop:lifting}; additionally it just states that the compatibility condition in Definition \ref{def:fusionspinstructure} between the fusion product $\lambda_{\inf S}$ on a spin structure and the fusion product $\lambda_{\gbas}$ is the same condition as the one that the isomorphism $\kappa$ of the corresponding trivialization is fusion-preserving. \end{comment}

\begin{comment}
In the conventions here the key calculation (that above condition is satisfied if and only if the bundle isomorphism $\kappa$ is fusion-preserving) is a lot easier:
\begin{eqnarray*}
\lambda_{\inf S}(q_{23}\cdot \beta_{23} \otimes q_{12} \cdot \beta_{12}) &=& \lambda_{\inf S}(\kappa(q_{23}\otimes \beta_{23}) \otimes \kappa(q_{12},\beta_{12}))
\\&=&
\kappa(\lambda_{\inf S}(q_{23} \otimes q_{12}) \otimes \lambda_{L\delta^{*}\gbas}(\beta_{23} \otimes \beta_{12}))
\\&=&
\lambda_{\inf S}(q_{23} \otimes q_{12}) \cdot \lambda_{L\delta^{*}\gbas}(\beta_{23} \otimes \beta_{12})
\end{eqnarray*}
\end{comment}

Next we include thin structures into the lifting-gerbe description. 
The thin structure $d^{\nu}$ on the central extension $\lspinhat n$ pulls back along
the map $L\delta$ to a  thin structure $d_{\q}$ on the $\ueins$-bundle $\q$ of the lifting gerbe $\mathcal{S}_{LM}$. 
\begin{comment}
The bundle gerbe product $\mu$ is thin according to Proposition \ref{prop:thinmult}. 

This leads to the following definition, analogous the the one of an internal fusion product.

\begin{definition}
Let $\mathcal{G}$ be a bundle gerbe over $LM$ whose surjective submersion is the looping of a surjective submersion $\pi:Y \to M$. An \emph{internal thin structure} on $\mathcal{G}$ is a thin structure on its principal $\ueins$-bundle $\q$ over $LY^{[2]}$, such that its bundle gerbe product $\mu$ is a thin bundle morphism.
\end{definition}

\begin{remark}
A  \emph{thin structure} on $\mathcal{G}$ would be an isomorphism
\begin{equation*}
\mathcal{D}:\pr_1^{*}\mathcal{G} \to \pr_2^{*}\mathcal{G}
\end{equation*}
over $\thinpairs{LM}$, together with a transformation
\begin{equation*}
\delta: \pr_{23}^{*}\mathcal{D} \circ \pr_{12}^{*}\mathcal{D} \Rightarrow \pr_{13}^{*}\mathcal{D}
\end{equation*}
over $\thinpairs{LM} \times_{LM} \thinpairs{LM}$ satisfying a coherence condition. A thin structure can be considered on arbitrary bundle gerbes over $LM$, while an internal thin structure only makes sense for bundle gerbes whose surjective submersions are loopings. An internal thin structure determines a thin structure on $\mathcal{G}$, and any thin structure on a bundle gerbe $\mathcal{G}$ over $LM$ determines a $\diff^{+}(S^1)$-equivariant structure on $\mathcal{G}$. The problem with this remark is that it actually only defined an almost thin structure, and it is not very clear what the integrability condition has to be.
\end{remark}
\end{comment}
Suppose $\mathcal{T}=(T,\kappa)$ is a trivialization of $\mathcal{S}_{LM}$. A thin structure $d$ on $T$ is called \emph{compatible}, if  $\kappa$ is a thin bundle morphism (with respect to the thin structure  $d_{\q}$ on $\q$). A morphism $\varphi: \mathcal{T}_1 \to \mathcal{T}_2$ between trivializations with thin structures is called \emph{thin}, if it is thin as a morphism $\varphi: T_1 \to T_2$.

\begin{proposition}
\label{prop:thinlifting}
The assignment $(\inf S,\sigma,d)\mapsto (T_{\inf S},\kappa_{\inf S},d)$ establishes an equivalence of categories:

\begin{equation*}
\spstth
\cong
\bigset{12em}{Trivializations of $\mathcal{S}_{LM}$ with compatible thin structures}\text{.}
\end{equation*}
\end{proposition}

\begin{proof}
Based on the equivalence of Proposition \ref{prop:lifting}, we observe that the structure on the objects on both hand sides is the same, namely a thin structure $d$ on the $\ueins$-bundle $T=T_{\inf S}$.
It remains to check that the conditions are the same. For the trivialization, the condition is that the diagram
\begin{equation*}
\alxydim{@=\xypicst}{\pr_1^{*}T \otimes \pr_1^{*}\q \ar[r]^-{\pr_1^{*}\kappa} \ar[d]_{d \otimes d_{\q}}  & \pr_1^{*}T \ar[d]^{d} \\  \pr_2^{*}T \otimes \pr_2^{*}\q \ar[r]_-{\pr_2^{*}\kappa} & \pr_2^{*}T}
\end{equation*}
over $\thinpairs{L(FM^{[2]})}=\thinpairs{L(FM \times \spin n)}$ is commutative. We recall the relation  $\kappa(t \otimes \tilde\tau ) = t\cdot \tilde\tau$
between $\kappa$ and the principal $\lspinhat n$-action on $T$. Under this relation, the commutativity of the diagram is equivalent to the following equation:
\begin{equation*}
d(t \cdot \tilde\tau)=d(\kappa_{}(t \otimes \tilde \tau) = \kappa_{}(d(t) \otimes d_{\q}(\tilde\tau)) = d_{}(t)\cdot d_{\q}(\tilde\tau)\text{.} 
\end{equation*}
This is precisely the  condition of Definition \ref{def:thinspin}. For the morphisms, we have on both  sides the same condition, namely that $\varphi: T_1 \to T_2$ is thin with respect to the thin structures on $T_1$ and $T_2$. \end{proof}

Finally, we  combine fusion products and thin structures on trivializations in the following definition.  \begin{comment}
We have the spin lifting gerbe $\mathcal{S}_{LM}$  equipped with an internal fusion product $\lambda_{\q}$ and with an internal thin structure $d_{\q}$. The strict multiplicativity of the thin structure $d^{\nu}$   (Proposition \ref{prop:thinmult}) shows that $d_{\q}$ is fusion with respect to $\lambda_{\q}$ ... though this doesn't seem to be important!
\end{comment}

\begin{definition}
\label{def:thinfusiontriv}
\begin{enumerate}[(i)]
\item 
A \emph{thin fusion trivialization} of the spin lifting gerbe $\mathcal{S}_{LM}$ is a trivialization $\mathcal{T}=(T,\kappa)$ with a fusion product $\lambda$ on $T$ compatible with $\lambda_{\q}$ and a thin structure $d$ on $T$ that is fusive with respect to $\lambda$ and compatible with $d_{\q}$.

\item
A \emph{morphism between thin fusion trivializations} $(T,\kappa,\lambda,d)$ and $(T',\kappa',\lambda',d')$ is a fusion-preserving, thin bundle morphism $\varphi \maps T \to T'$, such that the diagram 
\begin{equation*}
\alxydim{@=\xypicst}{\pr_2^{*}T  \otimes \q  \ar[r]^-{\kappa} \ar[d]_{\pr_2^{*}\varphi} & \pr_1^{*}T \ar[d]^{\pr_1^{*}\varphi} \\ \pr_2^{*}T' \otimes \q \ar[r]_-{\kappa'} & \pr_1^{*}T'}
\end{equation*}
commutes in the homotopy category $h\ufusbunth {LFM^{[2]}}$. 
\end{enumerate}
\end{definition}

The condition that the diagram in (ii) commutes in $h\ufusbunth{LFM^{[2]}}$, i.e. up to homotopy through thin, fusion-preserving bundle morphisms means, explicitly, that there is a smooth map 
\begin{equation}
\label{eq:Phi}
H: [0,1] \times \pr_2^{*}T_{\inf S}  \otimes \q \to \pr_1^{*}T'_{\inf S}
\end{equation}
such that $H_t: \pr_2^{*}T_{\inf S}  \otimes \q \to \pr_1^{*}T'_{\inf S}$ is a thin, fusion-preserving bundle morphism for all $t\in[0,1]$ and we have $H_0=\pr_1^{*}\varphi \circ \kappa$ and $H_1=\kappa' \circ \pr_2^{*}\varphi$.

The category of thin fusion trivializations is denoted by $\trivthfus{\mathcal{S}_{LM}}$. Based on the action functor \erf{eq:actiontrivspin}, it is straightforward to see that it is a module over the monoidal category $\ufusbunth{LM}$.

\begin{proposition}
\label{prop:equivthinfusion}
The assignment $(\inf S,\sigma,\lambda,d)\mapsto (T_{\inf S},\kappa_{\inf S},\lambda,d)$ establishes an equivalence of categories:
\begin{equation*}
\spstthfus
\cong
\trivthfus{\mathcal{S}_{LM}}\text{.}
\end{equation*}
Moreover, it is equivariant with respect to the $\ufusbunth{LM}$-actions on both categories.
\end{proposition}

\begin{proof}
Based on the equivalence of Proposition \ref{prop:lifting} and its extension to fusion products (Proposition \ref{prop:fusionlifting}) and thin structures (Proposition \ref{prop:thinlifting}) it remains to notice, on the level of objects, that  the  compatibility condition between fusion product and thin structure is the same on both sides. Concerning the morphisms, we first observe that we have, in both categories, thin, fusion-preserving morphisms $\varphi: T_{\inf S} \to T_{\inf S'}$ between $\ueins$-bundles over $LFM$. It remains to check that the commutativity in $h\ufusbunth{LFM^{[2]}}$ of  (ii) of Definition \ref{def:thinfusiontriv} is equivalent to (iv) of Definition \ref{def:thinfusionspinmorph}. In order to see this, we notice that the existence of the map $H$ in \erf{eq:Phi} is equivalent to the existence of a fusion homotopy
\begin{equation*}
h: [0,1] \times LFM^{[2]} \to \ueins
\end{equation*}
with $h_0=1$ and 
\begin{equation}
\label{eq:proofcond}
h_1\cdot (\pr_1^{*}\varphi \circ \kappa)= \kappa' \circ \pr_2^{*}\varphi
\end{equation}
\begin{comment}
Indeed, the relation between $H$ and $h$ is
$H_t=h_t\cdot H_0$,
over which one determines the other. The condition that $h_t$ is a fusion map for all $t$ is equivalent to $H_t$ being a thin, fusion-preserving bundle morphism \cite[Section 3.3]{waldorf11}. 
\end{comment}
Now, under the correspondence ($\kappa(t \otimes \tilde\tau ) = t\cdot \tilde\tau$) between the bundle morphisms $\kappa$ and $\kappa'$ with the $\lspinhat n$-action on $\inf S$ and $\inf S'$, respectively, \erf{eq:proofcond} is precisely Equation \erf{eq:proofcond2} in Definition \ref{def:thinfusionspinmorph}. \end{proof}

\begin{comment}
In Theorem \ref{th:trequiv} we will consider the homotopy category of the category $\trivthfus{\mathcal{S}_{LM}}$ of thin fusion trivializations. In this homotopy category,  two morphisms $\varphi_0,\varphi_1: T \to T'$ are identified if there is a smooth map $h:[0,1] \times T \to T'$ with $h_0=\varphi_0$, $h_1=\varphi_1$, and $h_t$ is a morphism in $\trivthfus{\mathcal{S}_{LM}}$ for all $t\in[0,1]$.
The equivalence of Proposition \ref{prop:equivthinfusion} induces an equivalence between the homotopy categories, equivariant under the actions of the homotopy category $h\ufusbunth{LM}$.
\end{comment}

\setsecnumdepth{2}

\section{Superficial spin connections}

\label{sec:geometricspin}

In Section \ref{sec:spinconnections} we study the notion of a spin connection introduced by Coquereaux and Pilch \cite{Coquereaux1998}, and  circumstances under which they induce thin spin structures. We couple spin connections to fusion products and introduce the notion of a superficial geometric fusion spin  structure. In Section \ref{sec:liftingtheoryspinconnections} we develop the corresponding lifting gerbe theory.

\subsection{Spin connections on loop spaces}

\label{sec:spinconnections}

In the following we denote  by $\mathfrak{g}$ the Lie algebra of $\spin n$. 
The Levi-Cevita connection on $M$ induces a  connection $A \in \Omega^1(FM,\mathfrak{g})$  on the spin-oriented frame bundle $FM$. 
\begin{comment}
Recall that a connection satisfies
\begin{equation*}
A_{pg}(\partial_t(\rho(t)\gamma(t))) = \mathrm{Ad}^{-1}_{g}(A_{p}(\partial_t\rho(t)))+g^{-1}\partial_t\gamma(t)\text{.}
\end{equation*}
Its curvature is defined by $F := \mathrm{d}A +\frac{1}{2} [A \wedge A]$. 
\end{comment}
One can define a 1-form
 $\bar A \in \Omega^1(LFM,L\mathfrak{g})$  by
\begin{equation*}
\bar A|_{\tau}(X)(z) := A|_{\tau(z)}(X(z))\text{,}
\end{equation*}
where $\tau\in LFM$ and $X\in T_{\tau}LFM$. 
It is straightforward to  check that $\bar A$ is a connection on $LFM$.

\begin{comment}
\begin{remark}
The assignment $A \mapsto \bar A$ can be generalized to a map
\begin{equation*}
\Omega^n(X,V) \to \Omega^{n}(LX,LV)
\end{equation*}
for differential forms with values in a vector space $V$. In \cite{Coquereaux1998} this is called the \emph{looping of forms}, and $\bar A$ is called the \emph{loop form} of $A$. 
\end{remark} 
\end{comment}

\begin{definition}[{{\cite{Coquereaux1998}}}]
\label{def:spinconnection}
Let $(\inf S,\sigma)$ be a spin structure on $LM$. A \emph{spin connection} on $(\inf S,\sigma)$ is a connection $\Omega \in \Omega^1(\inf S,\widetilde{L\mathfrak{g}})$ on $\inf S$ such that $p_{*}(\Omega) = \sigma^{*}\bar A$, where $p_{*}: \widetilde{L\mathfrak{g}} \to L\mathfrak{g}$ is the projection in the Lie algebra extension.
\end{definition}

A triple $(\inf S,\sigma,\Omega)$ consisting of a spin structure and a spin connection is called a \emph{geometric spin structure on $LM$}.    Geometric spin structures form a category $\spstcon$ whose morphisms are connection-preserving morphisms between spin structures. This category is a module for the monoidal category $\ubuncon{LM}$ of Fréchet principal $\ueins$-bundles over $LM$ with connection, in terms of an action functor
\begin{equation}
\label{eq:actionspstcon}
\ubuncon{LM} \times \spstcon \to \spstcon
\end{equation}
lifting the action \erf{eq:actionspst} of $\ubun{LM}$ on $\spst$ to a setting with connections. If $K$ is a principal $\ueins$-bundle over $LM$ with connection $\eta \in \Omega^1(K)$, and $(\inf S,\sigma,\Omega)$ is a geometric spin structure, then a spin connection on the spin structure $K \otimes \inf S = (K \times_{LM} \inf S)/\,\ueins$ is defined by the 1-form 
\begin{equation*}
\eta \otimes \Omega := \pr_1^{*}\eta + \pr_2^{*}\Omega \in \Omega^1(K \times_{LM} \inf S)
\end{equation*}
that descends to a connection on $K \otimes \inf S$.

We introduce a notion of  \emph{scalar curvature} of a spin connection. For this purpose, we need  the splitting $\split \maps  L\mathfrak{g} \to \widetilde {L\mathfrak{g}}$  described in Section \ref{sec:connsplits} as the horizontal lift with respect to the connection $\nu$, as well as the associated map $Z$ of Lemma \ref{lem:Z}.
Further, we need a \emph{reduction of $LFM$ adapted to $\split$}, i.e. a map
\begin{equation*}
r: LFM \times L\mathfrak{g} \to \R
\end{equation*}
that is linear in the second argument and satisfies
\begin{equation}
\label{eq:redsplit}
r(\tau \cdot \gamma, \mathrm{Ad}^{-1}_{\gamma}(X)) =r(\tau,X)- Z(\gamma,X)
\end{equation}
for all $\tau\in LFM$, $X \in L\mathfrak{g}$ and $\gamma \in L\spin n$. 
Such a reduction can be defined using the connection $A$ on $FM$, by setting \cite[Proposition 6.2]{gomi3}
\begin{equation}
\label{eq:reduction}
r(\tau,X) :=-2 \int_{S^1} \left \langle  \tau^{*} A,X  \right \rangle\text{.}
\end{equation}
\begin{comment}
Explicitly, this means
\begin{equation*}
r(\tau,X)=-2\int_{S^1}\left \langle  A_{\tau(z)}(\partial_z\tau(z)),X(z)  \right \rangle\mathrm{d}z\text{.}
\end{equation*}
We check the condition:
\begin{eqnarray*}
r(\tau\gamma, \mathrm{Ad}^{-1}_{\gamma}(X)) &=&  -2\int_{S^1} \left \langle \tau\gamma^{*}A,\gamma^{-1}X\gamma  \right \rangle \\&=&  -2\int_{S^1} \left \langle  \gamma^{-1}(\tau^{*}A)\gamma + \gamma^{*}\theta,\gamma^{-1}X\gamma  \right \rangle 
 \\&=& -2\int_{S^1} \left \langle  \tau^{*}A,X \right \rangle -2 \int_{S^1} \left \langle \gamma^{-1}\mathrm{d}\gamma,\gamma^{-1}X\gamma  \right \rangle
 \\&=& r(\tau,X) - Z(\gamma,X)\text{.}
\end{eqnarray*}
\end{comment}

In order to define the announced scalar curvature we produce the auxiliary  map
\begin{equation*}
R: \inf S \times \widetilde {L\mathfrak{g}} \to \R : (\beta,\hat X)  \mapsto \hat X - \split(p_{*}(\hat X))+r(\sigma(\beta),p_{*}(\hat X))\text{.}
\end{equation*}
Then we define a 2-form $\psi\in \Omega^2(\inf S)$ by the formula
$\psi_t(X,Y) :=R(t, \mathrm{curv}(\Omega)_{t}(X,Y))$
where $t\in \inf S$ and $X,Y\in T_t\inf S$. The scalar curvature is now defined as follows.

\begin{lemma}
There is a unique 2-form $\mathrm{scurv}(\Omega) \in \Omega^2(LM)$ such that $L\pi^{*}\mathrm{scurv}(\Omega)=\psi$.
\end{lemma}

\begin{proof}
We show that $\psi$ descends.
Using \erf{eq:redsplit} it is straightforward to show that
\begin{equation*}
R(t\gamma,\mathrm{Ad}_{\gamma}^{-1}(\hat X)) = R(t,\hat X)
\end{equation*}
for all $t\in \inf S$, $\gamma\in \lspinhat n$, and $\hat X\in \widetilde{L\mathfrak{g}}$.  
\begin{comment}
Indeed,
\begin{eqnarray*}
R(\beta\gamma,\mathrm{Ad}_{g}^{-1}(\hat X))&=&\mathrm{Ad}_{g}^{-1}(\hat X) - \split(p_{*}(\mathrm{Ad}_{g}^{-1}(\hat X)))+r(\sigma(\beta\gamma),p_{*}(\mathrm{Ad}_{g}^{-1}(\hat X)))
\\&=&\mathrm{Ad}_{g}^{-1}(\hat X) - \split(\mathrm{Ad}_{g}^{-1}(p_{*}(\hat X)))+r(\sigma(\beta)g,\mathrm{Ad}_{g}^{-1}(p_{*}(\hat X)))
\\&=&\mathrm{Ad}_{g}^{-1}(\hat X) - \mathrm{Ad}_{g}^{-1}(\split(p_{*}(\hat X)))+Z(g,p_{*}(\hat X))+r(\sigma(\beta),p_{*}(\hat X))-Z(g,p_{*}(\hat X))
\\&=& R(\beta,\hat X)\text{.}
\end{eqnarray*}
\end{comment}
On the other hand, the curvature satisfies
\begin{equation*}
\pr_2^{*}\mathrm{curv}(\Omega)=\mathrm{Ad}^{-1}_{\delta}(\pr_1^{*}\mathrm{curv}(\Omega))
\end{equation*}
over $\inf S^{[2]}$, where $\delta:\inf S^{[2]} \to \lspinhat n$ is the difference map of the principal $\lspinhat n$-bundle $\inf S$. This shows that $\pr_2^{*}\psi = \pr_1^{*}\psi$ over $\inf S^{[2]}$.  
\end{proof}

We will  see  (Theorems \ref{th:spinconlift} and  \ref{th:trequivcon}) that under the correspondence between geometric string structures on $M$ and geometric spin structures on $LM$,  the scalar curvature is (minus) the transgression of the 3-form $K\in\Omega^3(M)$  associated to a geometric string structure (see Theorem \erf{th:stringconnections} \erf{th:stringconnections:3form}).

\label{sec:superficial}

\begin{comment}
We recall that we have formulated thin structures and fusion products for spin structures in terms of the $\ueins$-bundle $T_\inf S$ associated to a spin structure $(\inf S,\sigma)$. We would like to treat spin connections in a similar manner in terms of  connections on $T_\inf S$. The main problem  we encounter is that there are many connections on $T_{\inf S}$ corresponding to a single spin connection $\Omega$, and  a priori it is not clear which to choose.
\end{comment}

Let $(\inf S,\sigma,\Omega)$ be a geometric spin structure on $LM$, and let $T_{\inf S}$ be the associated principal $\ueins$-bundle over $LFM$. We consider the 1-form $\zeta\in\Omega^1(LFM)$ defined by
\begin{equation}
\label{eq:zeta}
\zeta_\tau(X) := r(\tau,\bar A_{\tau}(X))
\end{equation}
for $\tau\in LFM$ and $X\in T_{\tau}LFM$.
\begin{comment}
Explicitly, this is $\zeta_{\tau}(X)=-2\int_0^1 \left \langle  A_{\tau(z)}(\partial_z\tau(z)),A_{\tau(z)}(X(z))  \right \rangle \mathrm{d}z$. In case of the trivial bundle with $A=\theta$ we have
\begin{equation*}
\zeta_{\tau}(X)=-2\int_0^1 \left \langle  \theta_{\tau(z)}(\partial_z\tau(z)),\theta_{\tau(z)}(X(z))  \right \rangle \mathrm{d}z=\beta_{\tau}(X)\text{.}
\end{equation*}
\end{comment}

\begin{lemma}
\label{lem:omegafamily}
For every $x\in \R$, the formula
\begin{equation*}
\omega_{\Omega,x} :=\Omega - \split(\sigma^{*}\bar A) + \frac{x}{2}\sigma^{*}\zeta \in \Omega^1(T_\inf S)
\end{equation*}
defines  a  connection on $T_\inf S$,  of curvature
\begin{equation*}
\mathrm{curv}(\omega_{\Omega,x})=L\pi^{*}\mathrm{scurv}(\Omega)-\frac{1}{2}\omega(\bar A \wedge \bar A)-r(\mathrm{curv}(\bar A))+\frac{x}{2}\mathrm{d}\zeta\text{.}
\end{equation*} 
\end{lemma}

\begin{proof}
 It suffices to prove that $\omega_{\Omega,0}$ is a connection, which is a standard calculation. 
\begin{comment}
For $\tilde\tau\in P\inf S$ with underlying $\sigma(\tilde\tau)=:\tau\in PLFM$ and $\delta\in P\ueins$ with associated $i(\delta)=:\tilde\delta\in P\lspinhat n$ we calculate
\begin{eqnarray*}
\omega_{\split}(\partial_t(\tilde\tau(t)\delta(t)))&=&\Omega(\partial_t(\tilde\tau(t)\tilde\delta(t)))-s(\bar A(\partial_t\tau(t))) \\
&=&
 \Omega(\partial_t\tilde\tau(t))+\tilde\delta(t)^{-1}\partial_t\tilde\delta(t)-s(\bar A(\partial_t\tau(t)))
\\
&=& \omega_{\split}(\partial_t\tilde\tau(t))+\tilde\delta(t)^{-1}\partial_t\delta(t)\text{.}
\end{eqnarray*}
The calculation uses that $\ueins$ acts as the central subgroup of $\lspinhat n$.
\end{comment}
The connection $\omega_{\Omega,x}$ is then obtained by shifting the connection $\omega_{\Omega,0}$ by the 1-form $\frac{x}{2}\zeta$.
In order to compute the curvature of $\omega_{\Omega,0}$ we use the definition of the scalar curvature and obtain:
\begin{multline*}
\mathrm{d}\omega_{\Omega,0} = \mathrm{d}\Omega - \split(\sigma^{*}\mathrm{d}\bar A)  
\possiblelinebreak
= L\pi^{*}\mathrm{scurv}(\Omega)-\frac{1}{2}[\Omega\wedge\Omega]+\split(p_{*}(\mathrm{curv}(\Omega)))- \split(\sigma^{*}\mathrm{d}\bar A)-r(\sigma,p_{*}(\mathrm{curv}(\Omega)))
\end{multline*}
Then we  use that $\Omega - \split(\sigma^{*}\bar A)=\Omega - \split(p_{*}\Omega)\in \R$, so that
\begin{multline*}
0=[\Omega - \split(\sigma^{*}\bar A) \wedge \Omega - \split(\sigma^{*}\bar A)]=[\Omega\wedge\Omega]+[\split(\sigma^{*}\bar A) \wedge \split(\sigma^{*}\bar A)]-2[\Omega \wedge \split(\sigma^{*}\bar A)]\\=[\Omega\wedge\Omega]-[\split(\sigma^{*}\bar A) \wedge \split(\sigma^{*}\bar A)]\text{.}
\end{multline*}
This yields the claimed result.
\begin{comment}
Indeed, the calculation goes so:
\begin{eqnarray*}
 \mathrm{d}\omega_{\Omega,0}&=& \mathrm{scurv}(\Omega)+\frac{1}{2}[\split(\sigma^{*}\bar A)\wedge \split(\sigma^{*}\bar A)]+\split(\mathrm{curv}(\sigma^{*}\bar A))- \split(\sigma^{*}\mathrm{d}\bar A)-r(\sigma,\mathrm{curv}(\sigma^{*}\bar A))
\\
&=& \mathrm{scurv}(\Omega)+\frac{1}{2}\split([\sigma^{*}\bar A\wedge \sigma^{*}\bar A])+\frac{1}{2}\omega(\sigma^{*}\bar A \wedge \sigma^{*}\bar A)+\frac{1}{2}\split([\sigma^{*}\bar A\wedge \sigma^{*}\bar A])-r(\sigma,\mathrm{curv}(\sigma^{*}\bar A))
\\
&=& \mathrm{scurv}(\Omega)-\frac{1}{2}\sigma^{*}\omega(\bar A \wedge \bar A)-\sigma^{*}r(\mathrm{curv}(\bar A))
\end{eqnarray*}
\end{comment}
\end{proof}

We recall that a morphism between geometric spin structures $(\inf S, \sigma, \Omega)$ and $(\inf S', \sigma', \Omega')$ is an isomorphism $f: \inf S \to \inf S'$ of $\lspinhat n$-bundles over $LM$ such that $\Omega = f^{*}\Omega'$ and $\sigma' \circ f=\sigma$. It follows that $\omega_{\Omega,x}= f^{*}\omega_{\Omega',x}$ for all $x\in \R$, i.e. the induced isomorphism $f: T_\inf S \to T_{\inf S'}$ of $\ueins$-bundles is  connection-preserving for all connections  $\omega_{\Omega,x}$. Further, we find  under the action \erf{eq:actionspstcon} of the monoidal category $\ubuncon{LM}$ on the category $\spstcon$ of geometric spin structures the formula  $\omega_{\eta \otimes \Omega,x}=L\pi^{*}\eta + \omega_{\Omega,x}$. 

The following result explains which connection of the one-parameter family $\omega_{\Omega,x}$ should be used.

\begin{proposition}
\label{prop:pass}
Let $(\inf S,\sigma,\Omega)$ be a geometric spin structure.
Suppose the connection $\omega_{\Omega,x}$ is superficial for a parameter $x\in \R$, and let $d^{\omega_{\Omega,x}}$ be the associated thin structure on $T_{\inf S}$. Then, the pair $(\inf S,\sigma,d^{\omega_x})$ is a thin spin structure if and only if $x=1$. \end{proposition}
 
 We prepare the proof of this proposition with three lemmata. 
 \begin{comment}
 The first expresses the condition that $(\inf S,\sigma,d^{\omega_x})$ is a thin spin structure directly in terms of the connection $\omega_{\Omega,x}$.
 \end{comment}
 
\begin{lemma}
\label{lem:thinspinhor}
Suppose $\omega_{\Omega,x}$ is superficial. Then,
$(\inf S,\sigma,d^{\omega_{\Omega,x}})$ is a thin spin structure if and only if the following holds: for $(\gamma_1,\gamma_2):[0,1] \to\ LFM^{[2]}$ a thin path,  $\tilde\gamma_2$ a $\omega_{\Omega,x}$-horizontal lift of $\gamma_2$, $\tilde\delta$ a $\nu$-horizontal lift of the path $\delta$ defined by $\delta(t):= \delta(\gamma_1(t),\gamma_2(t))$,  we have for $\tilde\gamma_1 := \tilde\gamma_2 \cdot \tilde\delta$
\begin{equation*}
\int_{\tilde\gamma_1} \omega_{\Omega,x}\in \Z\text{.}
\end{equation*}
\end{lemma}

\begin{proof}
$(\inf S,\sigma,d^{\omega_{\Omega,x}})$ is a thin spin structure if and only condition of Definition \ref{def:thinspin} is satisfied:
\begin{equation*}
d^{\omega_{\Omega,x}}_{\tau_1\cdot \beta_1,\tau_2\cdot \beta_2}(t \cdot \tilde\tau) = d^{\omega_{\Omega_x}}_{\tau_1,\tau_2}(t) \cdot d^{\nu}_{\beta_1,\beta_2}(\tilde\tau)
\end{equation*}
for all $((\tau_1,\beta_1),(\tau_2,\beta_2)) \in \thinpairs{L(FM \times \spin n)}$, all $t\in T_{\inf S}$ projecting to $\tau_1$, and all $\tilde\tau\in\lspinhat n$ projecting $\beta_1$.  
\begin{comment}
We recall from Section \ref{sec:thinfusion}  that the thin structure induced from a superficial connection is defined by parallel transport along any thin path $\gamma$, i.e. a path $\gamma$ whose adjoint $\gamma^{\vee}$ is of  rank one. 
\end{comment}
We  denote by $pt_{\delta}^{\nu}$ and $pt_{\gamma_k}^{\omega_{\Omega,x}}$ the parallel transport maps associated to the connections $\nu$ on $\lspinhat n$ and $\omega_{\Omega,x}$ on $T_{\inf S}$, along the paths $\delta$, $\gamma_1$, and $\gamma_2$, respectively.  Then,  above condition is equivalent to the assertion that 
\begin{equation}
\label{eq:partranseq}
pt^{\omega_{\Omega,x}}_{\gamma_1}(t\cdot \tilde\tau) = pt_{\gamma_2}^{\omega_{\Omega,x}}(t)\cdot pt^{\nu}_{\delta}(\tilde\tau)
\end{equation}
holds for all thin paths $(\gamma_1,\gamma_2):[0,1] \to LFM^{[2]}$ with difference path $\delta:=L\delta(\gamma_1,\gamma_2)$, elements $t\in \inf S$ projecting to $\gamma_1(0)$ and $\tilde \tau\in \lspinhat n$ projecting to $\delta(0)$. 
We note that
\begin{equation*}
pt^{\omega_{\Omega,x}}_{\gamma_1}(t\cdot \tilde\tau) = \tilde\gamma_1(1)\cdot \exp \left ( 2\pi\im \int_{\tilde\gamma_1} \omega_{\Omega,x} \right )
\end{equation*} 
for \emph{all} lifts $\tilde\gamma_1$ of $\gamma_1$ with $\tilde\gamma_1(0)=t\cdot \tilde\tau$.  For horizontal lifts $\tilde\gamma_2$ with $\tilde\gamma_2(0)=t$ and $\tilde \delta$ with $\tilde\delta(0)=\tilde\tau$ we have $pt_{\gamma_2}^{\omega_{\Omega,x}}(t)=\tilde\gamma_2(1)$ and $pt^{\nu}_{\delta}(\tilde\tau)=\tilde\delta(1)$. With these formulas, \erf{eq:partranseq} is equivalent to the assertion that 
\begin{equation*}
\exp \left ( 2\pi\im\int_{\tilde\gamma_1} \omega_{\Omega,x} \right ) =1
\end{equation*}
for every thin path $(\gamma_1,\gamma_2)$, difference $\delta$, horizontal lifts $\tilde\gamma_2$ and $\tilde\delta$, and $\tilde\gamma_1:=\tilde\gamma_2\cdot \tilde\delta$. 
\end{proof}

\begin{comment} 
Motivated by the previous lemma, we now compute the integral of $\omega_{\Omega,x}$. \end{comment}
The following straightforward calculation only uses  property \erf{eq:redsplit} of the reduction $r$ and the defining property of the connection $A$.

\begin{lemma}
\label{lem:deltazeta}
The 1-form $\zeta$ satisfies the identity \begin{equation*}
(\Delta\zeta)_{\tau_1,\tau_2}(X_1,X_2)=-r(\tau_2,Y\delta^{-1})+Z(\delta,\bar A_{\tau_2}(X_2))+Z(\delta,Y\delta^{-1})\text{,}
\end{equation*}
where $(\tau_1,\tau_2)\in LFM^{[2]}$, $X_1\in T_{\tau_1}LFM$, $X_2\in T_{\tau_2}LFM$,   $\delta\in \lspin n$ is defined by the formula $\delta(z)\df\delta(\tau_1(z),\tau_2(z))$, and $Y \in T_{\delta}\lspin n$ is defined by $Y(z) := \mathrm{d}\delta(X_1(z),X_2(z))$. \end{lemma}

Now we are prepared for the following key calculation.

\begin{lemma}
\label{lem:partranscalc}
Let $(\gamma_1,\gamma_2):[0,1] \to LFM^{[2]}$ be a path with $\delta:[0,1] \to \lspin n$  defined by $\delta(t):=\delta(\gamma_1(t),\gamma_2(t))$. Assume that $\tilde\gamma_2$ is a $\omega_{\Omega,x}$-horizontal lift of $\gamma_2$, and that $\tilde\delta$ is a $\nu$-horizontal lift of $\delta$. Then, $\tilde\gamma_1 := \tilde\gamma_2\cdot\tilde\delta$ is a lift of $\gamma_1$, and
\begin{multline*}
\omega_{\Omega,x}(\partial_t\tilde\gamma_1(t))= \frac{1-x}{2} Z(\delta(t),\partial_t\delta(t)\delta(t)^{-1})\possiblelinebreak+\frac{2-x}{2}Z(\delta(t),\bar A(\partial_t\gamma_2(t)))+\frac{x}{2}r(\gamma_2(t),\partial_t\delta(t)\delta(t)^{-1})\text{.}
\end{multline*}
\end{lemma}

\begin{proof}
The assumptions of horizontality mean that
\begin{equation}
\label{eq:horlifts}
\omega_{\Omega,t}(\partial_t\tilde\gamma_2(t))=0
\quand
\nu(\partial_t\tilde\delta(t))=0
\end{equation} 
for all $t\in [0,1]$.
The relation $\omega_{\Omega,x} = \Omega-\split(\sigma^{*}\bar A) + \frac{x}{2}\sigma^{*}\zeta$ from the definition of $\omega_{\Omega,x}$  shows:
\begin{eqnarray}
\label{eq:help1}
\mathrm{Ad}^{-1}_{\tilde\delta(t)}(\Omega(\partial_t\tilde\gamma_2(t)))) &=& \mathrm{Ad}^{-1}_{\tilde\delta(t)}(\split(\sigma^{*}\bar A(\partial_t\tilde\gamma_2(t)))+\omega_{\Omega,x}(\partial_t\tilde\gamma_2(t))) -\frac{x}{2} \zeta(\partial_t\gamma_2(t))
\\
&\hspace{-0.2cm}\stackerf{eq:horlifts}{=}\hspace{-0.2cm}& \mathrm{Ad}^{-1}_{\tilde\delta(t)}(\split(\bar A(\partial_t\gamma_2(t)))-\frac{x}{2} \zeta(\partial_t\gamma_2(t))
\nonumber\\
&=& Z(\tilde\delta(t),\bar A(\partial_t\gamma_2(t))) + s(\mathrm{Ad}_{\delta(t)}^{-1}(\bar A(\partial_t\gamma_2(t)))) -\frac{x}{2} \zeta(\partial_t\gamma_2(t))\text{.}\nonumber
\end{eqnarray}
That $\Omega$ is a $\lspinhat n$-connection implies then the following:
\begin{eqnarray}
\label{eq:help2}
\Omega(\partial_t\tilde\gamma_1(t))\hspace{-0.7em} 
&=& \mathrm{Ad}^{-1}_{\tilde\delta(t)}(\Omega(\partial_t\tilde\gamma_2(t)))+\tilde\delta(t)^{-1}\partial_t\tilde\delta(t)
\\\nonumber&\stackerf{eq:help1}{=}& Z(\delta(t),\bar A(\partial_t\gamma_2(t))) + s(\mathrm{Ad}_{\delta(t)}^{-1}(\bar A(\partial_t\gamma_2(t))))-\frac{x}{2} \zeta(\partial_t\gamma_2(t))+\tilde\delta(t)^{-1}\partial_t\tilde\delta(t)\text{.}\nonumber 
\end{eqnarray}
The definition of the splitting $\split$ from the connection $\nu$ implies:
\begin{equation}
\label{eq:help3}
X-s(p_{*}(X))=\nu_{1}(g^{-1}X)\stackerf{eq:eps}{=}\nu_{g}(X)-\varepsilon_{\nu}|_{g^{-1},g}(0,X)\text{.}
\end{equation}
Now we put the pieces together and obtain:
\begin{eqnarray*}
&&\hspace{-3.2em}\omega_{\Omega,x}(\partial_t\tilde\gamma_1(t)) \\&=& \Omega(\partial_t\tilde\gamma_1(t)) -\split(\sigma^{*}\bar A(\partial_t\tilde\gamma_1(t)))+\frac{x}{2} \sigma^{*}\zeta(\partial_t\tilde\gamma_1(t))
\\
&\hspace{-0.6em}\stackerf{eq:help2}{=}\hspace{-0.6em}&  Z(\delta(t),\bar A(\partial_t\gamma_2(t))) + s(\mathrm{Ad}_{\delta(t)}^{-1}(\bar A(\partial_t\gamma_2(t))))\\&&\hspace{9em}-\frac{x}{2} \zeta(\partial_t\gamma_2(t))+\tilde\delta(t)^{-1}\partial_t\tilde\delta(t)-\split(\bar A(\partial_t\gamma_1(t)))+\frac{x}{2} \zeta(\partial_t\gamma_1(t)) \\
&=& -\split(\delta(t)^{-1}\partial_t\delta(t))+\tilde\delta(t)^{-1}\partial_t\tilde\delta(t) +Z(\delta(t),\bar A(\partial_t\gamma_2(t)))+\frac{x}{2}\Delta \zeta(\partial_t\gamma_1(t),\partial_t\gamma_2(t)) \\
&\hspace{-0.6em}\stackerf{eq:help3}{=}\hspace{-0.6em}& \nu_{\tilde\delta(t)}(\partial_t\tilde\delta(t))-\varepsilon_{\nu}|_{\tilde\delta(t)^{-1},\tilde\delta(t)}(0,\partial_t\tilde\delta(t)) +Z(\delta(t),\bar A(\partial_t\gamma_2(t)))+\frac{x}{2}\Delta \zeta(\partial_t\gamma_1(t),\partial_t\gamma_2(t)) \\
&\hspace{-0.6em}\stackerf{eq:horlifts}{=}\hspace{-0.6em}& -\varepsilon_{\nu}|_{\delta(t)^{-1},\delta(t)}(0,\partial_t\delta(t)) +Z(\delta(t),\bar A(\partial_t\gamma_2(t)))+\frac{x}{2}\Delta \zeta(\partial_t\gamma_1(t),\partial_t\gamma_2(t))\text{.} \end{eqnarray*}
With Lemma \ref{lem:deltazeta} this simplifies to
\begin{multline*}
\omega_{\Omega,x}(\partial_t\tilde\gamma_1(t)) =-\varepsilon_{\nu}|_{\delta(t)^{-1},\delta(t)}(0,\partial_t\delta(t)) +\frac{2-x}{2}Z(\delta(t),\bar A(\partial_t\gamma_2(t)))\\ -\frac{x}{2}Z(\delta(t),\partial_t\delta(t)\delta(t)^{-1})+\frac{x}{2}r(\gamma_2(t),\partial_t\delta(t)\delta(t)^{-1})\text{.}
\end{multline*}
Substituting explicit expressions, we get  $\frac{1}{2}Z(\delta(t),\partial_t\delta(t)\delta(t)^{-1})=-\varepsilon_{\nu}|_{\delta(t)^{-1},\delta(t)}(0,\partial_t\delta(t))$; this yields the claimed formula.
\begin{comment}
We have
\begin{eqnarray*}
-\varepsilon_{\nu}|_{\delta(t)^{-1},\delta(t)}(0,\partial_t\delta(t))&\stackerf{eq:epsilonex}{=}&\int_0^1 \left \langle \partial_z\delta^{\vee}(t,z)\delta^{\vee}(t,z)^{-1}, \partial_t\delta^{\vee}(t,z)\delta^{\vee}(t,z)^{-1}   \right \rangle\;\mathrm{d}z
  \\
  -\frac{1}{2}Z(\delta(t),\partial_t\delta(t)\delta(t)^{-1})&=&-   \int_{0}^1 \left \langle  \partial_z\delta^{\vee}(t,z)\delta^{\vee}(t,z)^{-1},\partial_t\delta(t,z)\delta(t,z)^{-1}  \right \rangle \mathrm{d}z
\end{eqnarray*}
\end{comment} 
\end{proof}

\begin{proof}[Now we are in position to prove Proposition \ref{prop:pass}, and start with the \quot{if}-part] 
Suppose  $\omega_{\Omega,1}$ is superficial. According to Lemma \ref{lem:thinspinhor}, it suffices to prove that for all thin paths $(\gamma_1,\gamma_2):[0,1] \to\ LFM^{[2]}$, all horizontal lifts $\tilde\gamma_2$ of $\gamma_2$ and $\tilde\delta$ of $\delta$ we get $\omega_{\Omega,1}(\partial_t\tilde\gamma_1(t))=0$ for $\tilde\gamma_1 := \tilde\gamma_2 \cdot \tilde\delta$ and all $t\in[0,1]$. By Lemma \ref{lem:partranscalc} this is given by
\begin{equation*}
\omega_{\Omega,1}(\partial_t\tilde\gamma_1(t))= \frac{1}{2}Z(\delta(t),\bar A(\partial_t\gamma_2(t)))+\frac{1}{2}r(\gamma_2(t),\partial_t\delta(t)\delta(t)^{-1})\text{.}
\end{equation*}
Explicitly, this is
\begin{multline*}
\omega_{\Omega,1}(\partial_t\tilde\gamma_1(t)) =   \int_{0}^1 \bigg\lbrace  \left \langle  \partial_z\delta^{\vee}(t,z)\delta^{\vee}(t,z)^{-1},A(\partial_t\gamma_2^{\vee}(t,z))  \right \rangle
 \possiblelinebreak
- \left \langle  A(\partial_z\gamma_2^{\vee}(t,z)),\partial_t\delta^{\vee}(t,z)\delta^{\vee}(t,z)^{-1} \right \rangle \bigg\rbrace\, \mathrm{d}z\text{.} \end{multline*}
The assumption that  $(\gamma_1,\gamma_2)^{\vee}:[0,1] \times S^1 \to FM^{[2]}$ is a rank one map implies that for every $(t,z)$ there exist $\alpha,\beta\in \R$, not both equal to zero, such that 
\begin{equation*}
\alpha\partial_t(\gamma_1,\gamma_2)^{\vee}(t,z)=\beta\partial_z(\gamma_1,\gamma_2)^{\vee}(t,z)
\end{equation*}
in $TFM \oplus TFM$, i.e. $\alpha\partial_t\gamma_1^{\vee}(t,z)=\beta\partial_z\gamma_1^{\vee}(t,z)$ and $\alpha\partial_t\gamma_2^{\vee}(t,z)=\beta\partial_z\gamma_2^{\vee}(t,z)$. These imply $\alpha\partial_t\delta^{\vee}(t,z)=\beta\partial_z\delta^{\vee}(t,z)$. Assuming either $\alpha\neq 0$ or $\beta \neq 0$, one can see by inspection that the integrand in above formula vanishes identically for every $(t,z)$. 
\end{proof}

\begin{proof}[We are left with the proof of the \quot{only if}-part of Proposition \ref{prop:pass}]
We assume $x\neq 1$ and produce a counterexample, i.e. appropriate paths for which the integral in Lemma \ref{lem:thinspinhor} does not vanish.  
Let $\gamma_2$ be the constant path at a constant loop  at a point $p\in FM$, i.e. $\gamma_2(t)(z):= p$. Let $\delta$ be a thin path in $\lspin n$, to be specified later. Then, with $\gamma_1:=\gamma_2\cdot\delta$ we have a thin path $(\gamma_1,\gamma_2)$ in $LFM^{[2]}$. 
\begin{comment}
Indeed, $(\gamma_1,\gamma_2)^{\vee}$ is given by $(t,z)\mapsto (p,p\cdot \delta^{\vee}(t,z))$. As $\delta$ is thin, there exist $\alpha,\beta\in \R$ such that $\alpha\partial_t\delta^{\vee}=\beta\partial_z\delta^{\vee}$. Thus \begin{equation*}
\alpha\partial_t(\gamma_1,\gamma_2)^{\vee}=(0,p\cdot \alpha\partial_t\delta^{\vee})=(0,p\cdot \beta\partial_z\delta^{\vee})=\beta\partial_z(\gamma_1,\gamma_2)^{\vee}\text{.}
\end{equation*}
Therefore, $(\gamma_1,\gamma_2)$ is thin. 
\end{comment}
We compute the quantity $\omega_{\Omega,x}(\partial_t\tilde\gamma_1(t))$ of Lemma \ref{lem:partranscalc}. The second term in the formula of Lemma \ref{lem:partranscalc}, $Z(\delta(t),\bar A(\partial_t\gamma_2(t)))$, vanishes 
since $\gamma_2$ is constant and $Z$ is linear in the second argument. Likewise, the third term vanishes, using the definition \erf{eq:reduction} of the reduction $r$ and again that $\gamma_2$ is constant:
\begin{comment}
\begin{equation*}
r(\gamma_2(t),\partial_t\delta(t)\delta(t)^{-1}) =-2\int_{S^1}\left \langle  A_{\tau(z)}(\partial_z\gamma_2^{\vee}(t,z)),\partial_t\delta(t,z)\delta(t,z)^{-1}  \right \rangle\mathrm{d}z = 0\text{.} 
\end{equation*}
\end{comment}
For the first term, however, we have
\begin{equation}
\label{eq:counter}
Z(\delta(t),\partial_t\delta(t)\delta(t)^{-1}) = 2 \int_{0}^1 \left \langle  \partial_z\delta^{\vee}(t,z)\delta^{\vee}(t,z)^{-1},\partial_t\delta^{\vee}(t,z)\delta^{\vee}(t,z)^{-1}  \right \rangle \mathrm{d}z\text{.} 
\end{equation}
Now we construct a specific thin path $\delta$.
Let $\tau \in \lspin n$ be a non-constant loop,  and let $\delta(t)(z)\df \tau(z e^{2\pi\im t})$, the full rotation of  the loop $\tau$. Note that 
\begin{equation*}
\partial_z\delta^{\vee}(t,z)= \partial_z\tau(ze^{2\pi\im t}) =\partial_t\delta^{\vee}(t,z)\text{,}
\end{equation*}
i.e. $\delta^{\vee}$ is thin and the linear dependence is expressed by constant coefficients. 
Thus, the integrand
in \erf{eq:counter} is quadratic, hence  non-negative, and even positive at at least one $z\in S^1$ point as $\tau$ is non-constant. Thus therefore,
\begin{equation*}
y_t :=Z(\delta(t),\partial_t\delta(t)\delta(t)^{-1}) >0
\end{equation*}  
for all $t\in [0,1]$.
It follows that
\begin{equation*}
\int_{\tilde\gamma_1}\omega_{\Omega,x} =\int_0^1 \omega_{\Omega,x}(\partial_t\tilde\gamma_1(t)) \mathrm{d}t = \frac{1-x}{2} \int_0^1 y_t \;\mathrm{d}t  \text{,}
\end{equation*}
which is non-zero as $x\neq 1$. Note that one can scale this quantity continuously down to zero with a parameter $0\leq \varepsilon\leq 1$, by simply letting all paths end at $\varepsilon$ instead at $1$. In particular, it can be arranged to be not an integer, hence
\begin{equation*}
\exp \left (2\pi\im \int_{\tilde\gamma_1} \omega_{\Omega,x} \right ) \neq 1\text{.}
\end{equation*}  
\end{proof}

Due to Proposition \ref{prop:pass} we promptly set $\omega_{\Omega} := \omega_{\Omega,1}$ as the connection of our choice on $T_\inf S$. Note that this is a non-standard choice, other treatments of geometric lifting problems  choose $\omega_{\Omega,0}$ -- e.g. \cite{gomi3,waldorf13}.

\begin{definition}
\label{def:superficialspinconnection}
A spin connection $\Omega$ on $\inf S$ is called \emph{superficial}, if the connection $\omega_{\Omega}$ on $T_{\inf S}$ is superficial. A geometric spin structure with superficial spin connection is called a \emph{superficial geometric spin structure}.
\end{definition}

Together with the connection-preserving isomorphisms between spin structures, superficial geometric spin structures form a category that we denote by $\spstconsf$. Due to Proposition \ref{prop:pass}, we obtain a functor
\begin{equation*}
\spstconsf \to \spstth: (\inf S,\sigma,\Omega)\mapsto (\inf S,\sigma,d^{\omega_{\Omega}})\text{.}
\end{equation*}
This functor guarantees the consistency of the various versions of spin structures upon passing from the setting with connections to the setting without connections.

\begin{comment}
It remains to add fusion products to the picture, which is straightforward to do. 
\end{comment}

\begin{definition}
\label{def:geometricfusion}
A \emph{superficial geometric fusion spin structure} on $LM$ is a spin structure $(\inf S,\sigma)$ together with a fusion product $\lambda$ and a superficial  spin connection $\Omega$, such that $\omega_{\Omega}$ is  fusive with respect to $\lambda$. 
\end{definition}

Morphisms between superficial geometric fusion spin structures are connection-preserving, fusion-preserving  morphisms of spin structures. Superficial geometric fusion spin structures
form a category that we denote by $\spstconsffus$. This category  is our loop space formulation of the category of geometric string structures. The action \erf{eq:actionspstcon} of the monoidal category $\ubuncon {LM}$ on geometric spin structures extends to an action
\begin{equation*}
\ufusbunconsf {LM} \times \spstconsffus \to \spstconsffus\text{.}
\end{equation*}
We will see (Corollary \ref{co:equivsupfusion} and Theorem \ref{th:trequivcon}) that this action exhibits $\spstconsffus$ as a torsor over $\ufusbunconsf{LM}$.

It is clear from the construction that the passage from the setting with connections to the setting without connections also works in the presence of fusion products, i.e. we have a functor 
\begin{equation}
\label{eq:forgetfussf}
\spstconsffus \to \spstthfus: (\inf S,\sigma,\lambda,\Omega)\mapsto (\inf S,\sigma,\lambda,d^{\omega_{\Omega}})\text{.}
\end{equation}
On the level of morphisms, this functor produces honest morphisms $f:\inf S \to \inf S'$ between $\lspinhat n$-bundles; these form a subset of the morphisms of $\spstthfus$ defined in Definition \ref{def:thinfusionspinmorph} that is characterized by the condition that the fusion homotopy in (iv) of that definition is constant. In particular, the functor \erf{eq:forgetfussf} is not full -- just as one would expect
it from the passage from a setting with connections to a setting without connections. 

\subsection{Lifting theory for spin connections}

\label{sec:liftingtheoryspinconnections}

We equip the spin lifting gerbe $\mathcal{S}_{LM}$ with a connection.
We recall that $\q \df L\delta^{*}\lspinhat n$ is the principal $\ueins$-bundle of $\mathcal{S}_{LM}$ over $LFM^{[2]}$. 
It is equipped with the pullback connection $L\delta^{*}\nu$, but the bundle gerbe product $\mu$ is not connection-preserving for the connection  $\nu$. Indeed, we have seen in Section \ref{sec:connsplits} that
\begin{comment}
there is an error 1-form $\varepsilon_{\nu}\in\Omega^1(\lspin n^2)$ satisfying
\begin{equation*}
\pr_1^{*}\nu + \pr_2^{*}\nu = m^{*}\nu + \varepsilon_{\nu}
\end{equation*}
over $\lspinhat n \times \lspinhat n$. Pullback along the map $L\delta_2:LFM^{[3]} \to L\spin n^2$ with $(\tau_3,\tau_2) \cdot \delta_2(\tau_1,\tau_2,\tau_3)=(\tau_2,\tau_1)$ gives
\end{comment}
\begin{equation}
\label{eq:nuerror}
\pr_{23}^{*}L\delta^{*}\nu + \pr_{12}^{*}L\delta^{*}\nu=\pr_{13}^{*}L\delta^{*}\nu+L\delta_2^{*}\varepsilon_{\nu}\text{.}
\end{equation}
\begin{comment}
This means that the bundle gerbe product $\mu$ is not connection-preserving, and the error is the 1-form $L\delta_2^{*}\varepsilon_{\nu}$. 
\end{comment}

We  now modify the connection $L\delta^{*}\nu$  such that $\mu$ becomes connection-preserving. For this purpose, we consider the 1-form
 $Z(L\delta,\pr_2^{*}\bar A)\in \Omega^1(LFM^{[2]})$
\begin{comment} 
 , which is, more explicitly,
\begin{equation*}
Z(L\delta,\pr_2^{*}\bar A)|_{\tau_1,\tau_2}(X_1,X_2) = 2\int_0^1 \left \langle \partial_z\delta(z)\delta(z)^{-1},A_{\tau_2(z)}(X_2(z))  \right \rangle \mathrm{d}z \text{,}
\end{equation*}
where $\delta:=L\delta(\tau_1,\tau_2)$, 
\end{comment}
and  the sum
\begin{equation}
\label{eq:xi}
\xi := L\delta^{*}\beta  +Z(L\delta,\pr_2^{*}\bar A)\in\Omega^1(LFM^{[2]})\text{,}
\end{equation}
where $\beta$ is the 1-form defined in Lemma \ref{lem:beta}. 
\begin{comment}
Explicitly, we have
\begin{multline*}
\xi_{\tau_1,\tau_2}(X_1,X_2)=\int_0^1 \left \langle \delta(z)^{-1}\partial_z\delta(z), \delta(z)^{-1}Y(z)   \right \rangle\;\mathrm{d}z \\+2\int_0^1 \left \langle \partial_z\delta(z)\delta(z)^{-1},A_{\tau_2(z)}(X_2(z))  \right \rangle \mathrm{d}z\text{.} 
\end{multline*}
\end{comment}
Under the simplicial operator 
\begin{equation*}
\Delta:=\pr_{23}^{*} + \pr_{12}^{*} - \pr_{13}^{*}:\Omega^k(FM^{[2]}) \to \Omega^{k}(FM^{[3]})
\end{equation*}
we obtain the following result. 

\begin{lemma}
\label{lem:epsiloncancel}
$\Delta\xi = -L\delta_2^{*}\varepsilon_{\nu}$. \end{lemma}

\begin{proof}
We make two calculations. First we calculate $\Delta\beta := \pr_1^{*}\beta + \pr_2^{*}\beta - m^{*}\beta\in \Omega^1(LG^2)$, using that
$ \mathrm{d}m_{\tau_1,\tau_2}(X_1,X_2)=X_1\tau_2 + \tau_1X_2$.
The result is
\begin{multline*}
(\Delta\beta)_{\tau_1,\tau_2}(X_1,X_2)=-\int_0^1 \left\lbrace \left \langle \tau_1(z)^{-1}\partial_z\tau_1(z), X_2(z)\tau_2^{-1}(z)   \right \rangle \right . \possiblelinebreak \left .- \left \langle \partial_z\tau_2(z)\tau_2(z)^{-1}, \tau_1(z)^{-1}X_1(z)   \right \rangle \right\rbrace\mathrm{d}z\text{.} 
\end{multline*}
\begin{comment}
The calculation is here:
\begin{eqnarray*}
&&\hspace{-1cm}(\Delta\beta)_{\tau_1,\tau_2}(X_1,X_2)
\\
&=& \beta_{\tau_1}(X_1) + \beta_{\tau_2}(X_2) - \beta_{\tau_1\tau_2}(X_1\tau_2 + \tau_1X_2)
\\
&=& \int_0^1 \left \langle \partial_z\tau_1(z)\tau_1(z)^{-1}, X_1(z)\tau_1^{-1}(z)   \right \rangle\;\mathrm{d}z
\\&&
+\int_0^1 \left \langle \partial_z\tau_2(z)\tau_2(z)^{-1}, X_2(z)\tau_2^{-1}(z)   \right \rangle\;\mathrm{d}z \\&&
-\int_0^1 \left \langle \partial_z\tau_1(z)\tau_1(z)^{-1}, (X_1(z)\tau_2(z) + \tau_1(z)X_2(z))\tau_2^{-1}(z)\tau_1(z)^{-1}   \right \rangle\;\mathrm{d}z 
\\&&
-\int_0^1 \left \langle \tau_1(z)\partial_z\tau_2(z)\tau_2(z)^{-1}\tau_1(z)^{-1}, (X_1(z)\tau_2(z) + \tau_1(z)X_2(z))\tau_2^{-1}(z)\tau_1(z)^{-1}   \right \rangle\;\mathrm{d}z 
\\
&=&
-\int_0^1 \left \langle \tau_1(z)^{-1}\partial_z\tau_1(z), X_2(z)\tau_2^{-1}(z)   \right \rangle\;\mathrm{d}z -\int_0^1 \left \langle \partial_z\tau_2(z)\tau_2(z)^{-1}, \tau_1(z)^{-1}X_1(z)   \right \rangle\;\mathrm{d}z \end{eqnarray*}
\end{multline*}
\end{comment}
For the second calculation we use
the notation $\delta_{ij}:=L\delta(\tau_i,\tau_j)$, in which $\delta_{13}=\delta_{23}\delta_{12}$ holds, and obtain
\begin{equation*}
(\Delta Z(L\delta,\pr_2^{*}\bar A))_{\tau_1,\tau_2,\tau_3}(X_1,X_2,X_3)=2 \int_0^1 \left \langle \partial_z\delta_{12}(z)\delta_{12}(z)^{-1},\delta_{23}(z)^{-1}Y_{23}(z)  \right \rangle \mathrm{d}z   \text{,}
\end{equation*}
where $Y_{23} := \mathrm{d}L\delta(X_{2},X_{3})$.
\begin{comment}
This calculation goes here:
\begin{eqnarray*}
&&\hspace{-2cm}(\Delta Z(L\delta,\pr_2^{*}\bar A))_{\tau_1,\tau_2,\tau_3}(X_1,X_2,X_3) 
\\&=&  2\int_0^1 \left \langle \partial_z\delta_{23}(z)\delta_{23}(z)^{-1},A_{\tau_3(z)}(X_3(z))  \right \rangle \mathrm{d}z  
 \\ && +2\int_0^1 \left \langle \partial_z\delta_{12}(z)\delta_{12}(z)^{-1},A_{\tau_2(z)}(X_2(z))  \right \rangle \mathrm{d}z  
\\&& -2\int_0^1 \left \langle \partial_z\delta_{13}(z)\delta_{13}(z)^{-1},A_{\tau_3(z)}(X_3(z))  \right \rangle \mathrm{d}z  
\\&=&   2\int_0^1 \left \langle \partial_z\delta_{12}(z)\delta_{12}(z)^{-1},A_{\tau_2(z)}(X_2(z))  \right \rangle \mathrm{d}z  
\\&& -2\int_0^1 \left \langle \delta_{23}(z)\partial_z\delta_{12}(z)\delta_{12}(z)^{-1}\delta_{23}(z)^{-1} ,A_{\tau_3(z)}(X_3(z))  \right \rangle \mathrm{d}z \\&=&  2\int_0^1 \left \langle \partial_z\delta_{12}(z)\delta_{12}(z)^{-1},A_{\tau_2(z)}(X_2(z))+\delta_{23}(z)^{-1}A_{\tau_3(z)}(X_3(z))\delta_{23}(z)  \right \rangle \mathrm{d}z  
\\&=& 2 \int_0^1 \left \langle \partial_z\delta_{12}(z)\delta_{12}(z)^{-1},\delta_{23}(z)^{-1}Y_{23}(z)  \right \rangle \mathrm{d}z  \text{.}
\end{eqnarray*}
This used the transformation formula
\begin{equation*}
A_{\tau_2(z)}(X_2(z))=A_{\tau_3(z)\delta_{23}(z)}(X_3(z)+Y_{23}(z))=\mathrm{Ad}_{\delta_{23}(z)}^{-1}(A_{\tau_3}(z)(X_3(z)))+\delta_{23}(z)^{-1}Y_{23}(z)
\end{equation*}
\end{comment}
We have the relation
$\Delta \circ L\delta^{*} = L\delta_2^{*} \circ \Delta$ which implements the fact that $\delta$ is a chain map between simplicial manifolds.
\begin{comment}
Indeed, for $f:G \to \R$
\begin{multline*}
(\Delta\circ \delta^{*})(f)(p_1,p_2,p_3) =  (\delta^{*}f)(p_2,p_3)+(\delta^{*}f)(p_1,p_2) - (\delta^{*}f)(p_1,p_3) \\= f(\delta_{23})+f(\delta_{12})-f(\delta_{23}\delta_{12}) =(\Delta f)(\delta_{23},\delta_{12})=(\delta_2^{*}\circ \Delta)(f)(p_1,p_2,p_3)
\end{multline*}
\end{comment}
Putting the two calculations together and identifying the result under  \erf{eq:eps} we obtain the claimed result.
\begin{comment}
This is done here:
\begin{eqnarray*}
\Delta\xi_{\tau_1,\tau_2,\tau_3}(X_1,X_2,X_3) &=& (\Delta\beta)_{\delta_{23},\delta_{12}}(Y_{23},Y_{12}) +\Delta Z(L\delta,\pr_2^{*}\bar A)_{\tau_1,\tau_2,\tau_3}(X_1,X_2,X_3) \\&=& -\int_0^1 \left \langle \delta_{23}(z)^{-1}\partial_z\delta_{23}(z), Y_{12}(z)\delta_{12}^{-1}(z)   \right \rangle\;\mathrm{d}z \\&&-\int_0^1 \left \langle \partial_z\delta_{12}(z)\delta_{12}(z)^{-1}, \delta_{23}(z)^{-1}Y_{23}(z)   \right \rangle\;\mathrm{d}z
\\&&+2 \int_0^1 \left \langle \partial_z\delta_{12}(z)\delta_{12}(z)^{-1},\delta_{23}(z)^{-1}Y_{23}(z)  \right \rangle \mathrm{d}z  
\\&=& - \int_0^1 \left \langle \delta_{23}(z)^{-1}\partial_z\delta_{23}(z),Y_{12}(z)\delta_{12}(z)^{-1}   \right \rangle \mathrm{d}z \\&&+ \int_0^1 \left \langle \delta_{23}(z)^{-1}Y_{23}(z),\partial_z\delta_{12}(z)\delta_{12}(z)^{-1}   \right \rangle \mathrm{d}z
\\&=& -\varepsilon_{\delta_{23},\delta_{12}}(Y_{23},Y_{12})
\\&=& -(L\delta_2^{*}\varepsilon)_{\tau_1,\tau_2,\tau_3}(X_1,X_2,X_3)\text{.}
\end{eqnarray*}
The formula for $\varepsilon$ was
\begin{eqnarray*}
&&\hspace{-0.8cm}\varepsilon_{\nu}|_{\tau_1,\tau_2}(X_1,X_2) \\ &=& \int_{S^1} \left \langle   \tau_1(z)^{-1}\partial_z\tau_1(z), X_2(z)\tau_2(z)^{-1} \right \rangle \mathrm{d}z -\int_{S^1} \left \langle  \tau_1(z)^{-1}X_1(z), \partial_z\tau_2(z)\tau_2(z)^{-1}  \right \rangle \mathrm{d}z\text{.}
\end{eqnarray*}
\end{comment}
\end{proof}

We  also need to calculate the derivative of the 1-form $\xi$.

\begin{lemma}
\label{lem:dxi}
$\mathrm{d}\xi= -\frac{1}{2}L\delta^{*}\omega(\theta\wedge\theta)-L\delta^{*}\mathrm{curv}(\nu)+Z(L\delta,\pr_2^{*}\mathrm{d}\bar A)+\omega(L\delta^{*}\theta\wedge \mathrm{Ad}^{-1}_{L\delta}(\pr_2^{*}\bar A))$.
\end{lemma}

\begin{proof}
\begin{comment}
We had defined $\xi := L\delta^{*}\beta  + Z(L\delta,\pr_2^{*}\bar A)$. 
\end{comment}
We have \cite[Lemma 5.8 (a)]{gomi3}
\begin{comment}
\begin{equation*}
\mathrm{d}Z(Lg,\pr_{1}^{*}\bar A)= Z(Lg,\pr_1^{*}\mathrm{d}\bar A)-\omega(Lg^{*}\theta\wedge \mathrm{Ad}^{-1}_{Lg}(\pr_1^{*}\bar A))\text{,}
\end{equation*}
and the equivalent formula
\end{comment}
\begin{equation*}
\mathrm{d}Z(L\delta,\pr_{2}^{*}\bar A)= Z(L\delta,\pr_2^{*}\mathrm{d}\bar A)-\omega(L\delta^{*}\theta\wedge \mathrm{Ad}^{-1}_{L\delta}(\pr_2^{*}\bar A))\text{.}
\end{equation*}
With Lemma \ref{lem:beta}  we obtain the claimed formula.
\begin{comment}
\begin{eqnarray*}
\mathrm{d}\xi &=& L\delta^{*}\mathrm{d}\beta  +\mathrm{d} Z(L\delta,\pr_2^{*}\bar A)
\\&=& -\frac{1}{2}L\delta^{*}\omega(\theta\wedge\theta)-L\delta^{*}\mathrm{curv}(\nu)+Z(L\delta,\pr_2^{*}\mathrm{d}\bar A)+\omega(L\delta^{*}\theta\wedge \mathrm{Ad}^{-1}_{L\delta}(\pr_2^{*}\bar A))
\text{.}
\end{eqnarray*} 
\end{comment}
\end{proof}

In the following we  consider the 1-form $\xi-\frac{1}{2}\Delta\zeta\in \Omega^1(FM^{[2]})$, with $\zeta$ the 1-form defined at the beginning of Section \ref{sec:superficial}. 
\begin{comment}
An explicit expression is obtained from
\begin{multline*}
\xi_{\tau_1,\tau_2}(X_1,X_2)=\int_0^1 \left \langle \delta(z)^{-1}\partial_z\delta(z), \delta(z)^{-1}Y(z)   \right \rangle\;\mathrm{d}z \\+2\int_0^1 \left \langle \partial_z\delta(z)\delta(z)^{-1},A_{\tau_2(z)}(X_2(z))  \right \rangle \mathrm{d}z
\end{multline*}
and from Lemma \ref{lem:deltazeta},
\begin{multline*}
(\Delta\zeta)_{\tau_1,\tau_2}(X_1,X_2) =2 \int_0^1 \left  \langle A_{\tau_2(z)}(\partial_z\tau_2(z)), Y(z)\delta(z)^{-1}    \right \rangle \mathrm{d}z\\+2 \int_0^1 \left \langle \partial_z\delta(z)\delta(z)^{-1},A_{\tau_2(z)}(X_2(z))  \right \rangle\mathrm{d}z  +2 \int_0^1 \left \langle  \partial_z\delta(z)\delta(z)^{-1}, Y(z)\delta(z)^{-1} \right \rangle\mathrm{d}z\text{.}
\end{multline*}
Together they give
\begin{multline*}
(\xi-\frac{1}{2}\Delta\zeta)_{\tau_1,\tau_2}(X_1,X_2) =\int_0^1 \left \langle \partial_z\delta(z)\delta(z)^{-1},A_{\tau_2(z)}(X_2(z))  \right \rangle \mathrm{d}z\\
 - \int_0^1 \left  \langle A_{\tau_2(z)}(\partial_z\tau_2(z)), Y(z)\delta(z)^{-1}    \right \rangle \mathrm{d}z\text{.} \end{multline*}
Thus it can also be expressed as
\begin{equation*}
\xi-\frac{1}{2}\Delta\zeta = \frac{1}{2}Z(L\delta,\pr_2^{*}\bar A)+\frac{1}{2}r(\pr_2,L\delta^{*}\bar\theta)\text{.}
\end{equation*}
\end{comment}

\begin{lemma}
\label{lem:fusionform}
The 1-form $\xi-\frac{1}{2}\Delta\zeta\in\Omega^1(LFM^{[2]})$  is a  superficial fusion form, i.e. a su\-perficial fusive connection on the trivial bundle with respect to the trivial fusion product.
\end{lemma}

\begin{proof}
This can be verified directly; however, we  show in Lemma \ref{lem:tromega} that $\xi-\frac{1}{2}\Delta\zeta$ is in the image of the transgression homomorphism \ref{eq:difftrans}; such forms are automatically superficial \cite[Lemma 3.1.7]{waldorf11} and  fusion \cite[Proposition 3.2.3]{waldorf11}.
\begin{comment}
That the form is thin can be seen explicitly:
With the explicit expression, for any paths $\gamma \in PLFM$ and $\delta\in P\lspin n$:
\begin{eqnarray*}
(\xi-\frac{1}{2}\Delta\zeta)_{\gamma(t)\delta(t),\gamma(t)}(\partial_t(\gamma(t)\delta(t)),\partial_t\gamma(t))
&=& - \int_0^1 \left  \langle A_{\gamma(t,z)}(\partial_z\gamma(t,z)), \partial_t\delta(t,z)\delta(t,z)^{-1}    \right \rangle \mathrm{d}z 
\\&& 
+\int_0^1 \left \langle \partial_z\delta(t,z)\delta(t,z)^{-1},A_{\gamma(t,z)}(\partial_t\gamma(t,z))  \right \rangle \mathrm{d}z
\end{eqnarray*}
Suppose $t\mapsto (\gamma(t)\delta(t),\gamma(t))$ is thin, i.e. the differential of 
\begin{equation*}
(t,z)\mapsto (\gamma(t,z)\delta(t,z),\gamma(t,z))
\end{equation*}
has at most rank 1. This means that the vectors $(\partial_t\gamma(t,z)\delta(t,z)+\gamma(t,z)\partial_t\delta(t,z),\partial_t\gamma(t,z))$
 and $(\partial_z\gamma(t,z)\delta(t,z)+\gamma(t,z)\partial_z\delta(t,z),\partial_z\gamma(t,z))$ are linearly dependent. So there are constants $\alpha,\beta$ such that
\begin{eqnarray*}
\alpha \partial_t\gamma(t,z) &=&\beta \partial_z\gamma(t,z)
\\\alpha\partial_t\gamma(t,z)\delta(t,z)+\alpha\gamma(t,z)\partial_t\delta(t,z)&=&\beta \partial_z\gamma(t,z)\delta(t,z)+\beta\gamma(t,z)\partial_z\delta(t,z)\text{.}
\end{eqnarray*}
Plugging the first equation into the second, it gives $\alpha\partial_t\delta(t,z)=\beta\partial_z\delta(t,z)$.
We may assume that $\alpha\neq0$ and $\beta\neq 0$. 
\end{comment}
\end{proof}

We consider on $\q = L\delta^{*}\lspinhat n$ the connection
\begin{equation}
\label{eq:nukorr}
\chi_{\corr} := L\delta^{*}\nu + (\xi-\frac{1}{2}\Delta \zeta)\text{.}
\end{equation}

\begin{proposition}
\label{prop:chi}
The connection $\chi_{\corr}$ has the following properties:
\begin{enumerate}[(i)]

\item 
It makes the bundle gerbe product $\mu$  connection-preserving.  

\item
It is superficial, and the induced thin structure on $\q$ coincides with the one induced by the original connection: $d^{\chi_{\corr}}=d_P=d^{L\delta^{*}\nu}$. 
\item
It is a fusive connection with respect to the fusion product $\lambda_{\q}=L\delta^{*}\lambda$ on $\q$. 
\end{enumerate}
\end{proposition}

\begin{proof}
(i) holds because the correction term satisfies
\begin{equation*}
\Delta(\xi-\frac{1}{2}\Delta \zeta)=\Delta\xi =-L\delta_2^{*}\varepsilon_{\nu}\text{,}
\end{equation*} 
by Lemma \ref{lem:epsiloncancel} and so cancels the error in the multiplicativity of $\nu$.
\begin{comment}
In more detail, if we compute the error form of $\chi_{\corr}$, we get with \erf{eq:nuerror} and Lemma \ref{lem:epsiloncancel}:
\begin{equation*}
\Delta\chi_{\corr} = \Delta L\delta^{*}\nu+\Delta\xi=L\delta_2^{*}\varepsilon_{\nu}-L\delta_2^{*}\varepsilon_{\nu}=0\text{.}
\end{equation*}
\end{comment}
(ii)  and (iii) hold because of Lemma \ref{lem:fusionform}. We regard $\q$ (equipped with the fusion product $\lambda_{\q}$ and  connection $\chi_{\corr}$) as the tensor product of $\q$ (equipped with $\lambda_{\q}$ and the superficial fusive connection $L\delta^{*}\nu$) and the trivial bundle (equipped with the trivial fusion product and the superficial fusive connection $\xi-\frac{1}{2}\Delta\zeta$). Since the conditions of being superficial and fusion are preserved under the tensor product, $\chi_{\corr}$ is superficial and fusion. The same argument works for thin structures instead of connections. Here, the thin structure $d^{\;\xi-\frac{1}{2}\zeta}$ is the trivial one \cite[Proposition 3.1.8]{waldorf11}, so that $d^{\chi_{\corr}}=d^{L\delta^{*}\nu}$.
\end{proof}

It remains to find a \emph{curving} adapted to the connection $\chi_{\corr}$, i.e. a 2-form $B_{\corr}$ on  $LFM$ such that
$\Delta B_{\corr}= \mathrm{curv}(\chi_{\corr})$. 

\begin{proposition}
\label{prop:DeltaB}
The 2-form
\begin{equation*}
B_{\corr} := \frac{1}{2}\omega(\bar A \wedge \bar A) +r(\mathrm{curv}(\bar A))-\frac{1}{2}\mathrm{d}\zeta\in\Omega^2(LFM^{[2]}) \end{equation*}
is a curving for the connection $\chi_{\corr}$. \end{proposition}

\begin{proof}
With Lemmata \ref{lem:beta} and  \ref{lem:dxi} the curvature of $\chi_{\corr}$ is:
\begin{eqnarray*}
\mathrm{curv}(\chi_{\corr}) = -\frac{1}{2}L\delta^{*}\omega(\theta\wedge\theta)+Z(L\delta,\pr_2^{*}\mathrm{d}\bar A)+\omega(L\delta^{*}\theta\wedge \mathrm{Ad}^{-1}_{L\delta}(\pr_2^{*}\bar A))-\frac{1}{2}\Delta\mathrm{d}\zeta\text{.}
\end{eqnarray*}
\begin{comment}
More explicitly,
\begin{eqnarray*}
\mathrm{curv}(\chi_{\corr}) &=& L\delta^{*}\mathrm{curv}(\nu)+ \mathrm{d}\xi-\frac{1}{2}\Delta\mathrm{d}\zeta
\\&=&L\delta^{*}\mathrm{curv}(\nu)-\frac{1}{2}L\delta^{*}\omega(\theta\wedge\theta)-L\delta^{*}\mathrm{curv}(\nu)+Z(L\delta,\pr_2^{*}\mathrm{d}\bar A)\\&&\hspace{14em}+\omega(L\delta^{*}\theta\wedge \mathrm{Ad}^{-1}_{L\delta}(\pr_2^{*}\bar A))-\frac{1}{2}\Delta\mathrm{d}\zeta
\\&=&-\frac{1}{2}L\delta^{*}\omega(\theta\wedge\theta)+Z(L\delta,\pr_2^{*}\mathrm{d}\bar A)+\omega(L\delta^{*}\theta\wedge \mathrm{Ad}^{-1}_{L\delta}(\pr_2^{*}\bar A))-\frac{1}{2}\Delta\mathrm{d}\zeta\text{.}
\end{eqnarray*}
\end{comment}
In order to calculate $\Delta B_{\corr}$, we compute
with \erf{eq:omegainv} and \erf{eq:redsplit} the formulas
\begin{eqnarray*}
\Delta\omega(\bar A\wedge \bar A)&=&-Z(L\delta,[\pr_2^{*}\bar A\wedge\pr_2^{*}\bar A])-2\omega(\mathrm{Ad}_{L\delta}^{-1}(\pr_2^{*}\bar A) \wedge  L\delta^{*}\theta)-\omega( L\delta^{*}\theta\wedge L\delta^{*}\theta)
\\
\Delta r(\mathrm{curv}(\bar A))&=& Z(L\delta,\mathrm{curv}(\pr_2^{*}\bar A)) 
\end{eqnarray*}
\begin{comment}
Actually we rather compute the equivalent formulas
\begin{eqnarray*}
\Delta\omega(\bar A\wedge \bar A)&=&Z(Lg,[\pr_1^{*}\bar A\wedge\pr_1^{*}\bar A])+2\omega(\mathrm{Ad}_{Lg}^{-1}(\pr_1^{*}\bar A) \wedge  Lg^{*}\theta)+\omega( Lg^{*}\theta\wedge Lg^{*}\theta)
\\
\Delta r(\mathrm{curv}(\bar A))&=& -Z(Lg,\mathrm{curv}(\pr_1^{*}\bar A)) 
\end{eqnarray*}
Indeed, we have
\begin{eqnarray*}
\Delta\omega(\bar A\wedge \bar A) &=& \omega(\pr_2^{*}\bar A\wedge \pr_2^{*}\bar A)-\omega(\pr_1^{*}\bar A\wedge \pr_1^{*}\bar A)
\\&=&\omega((\mathrm{Ad}_{Lg}^{-1}(\pr_1^{*}\bar A) + Lg^{*}\theta)\wedge (\mathrm{Ad}_{Lg}^{-1}(\pr_1^{*}\bar A) + Lg^{*}\theta))-\omega(\pr_1^{*}\bar A\wedge \pr_1^{*}\bar A)
\\&=&\omega(\mathrm{Ad}_{Lg}^{-1}(\pr_1^{*}\bar A) \wedge \mathrm{Ad}_{Lg}^{-1}(\pr_1^{*}\bar A))+2\omega(\mathrm{Ad}_{Lg}^{-1}(\pr_1^{*}\bar A) \wedge  Lg^{*}\theta)\\&&+\omega( Lg^{*}\theta\wedge Lg^{*}\theta)-\omega(\pr_1^{*}\bar A\wedge \pr_1^{*}\bar A)
\\&=&\omega(\pr_1^{*}\bar A \wedge \pr_1^{*}\bar A)+Z(Lg,[\pr_1^{*}\bar A\wedge\pr_1^{*}\bar A])+2\omega(\mathrm{Ad}_{Lg}^{-1}(\pr_1^{*}\bar A) \wedge  Lg^{*}\theta)\\&&+\omega( Lg^{*}\theta\wedge Lg^{*}\theta)-\omega(\pr_1^{*}\bar A\wedge \pr_1^{*}\bar A)
\end{eqnarray*}
Further we have
\begin{eqnarray*}
\Delta r(\mathrm{curv}(\bar A)) &=& r(\pr_2,\mathrm{curv}(\pr_2^{*}\bar A))-r(\pr_1,\mathrm{curv}(\pr_1^{*}\bar A))
\\
&=& r(\pr_1 \cdot Lg,\mathrm{Ad}_{Lg}^{-1}(\mathrm{curv}(\pr_1^{*}\bar A)))-r(\pr_1,\mathrm{curv}(\pr_1^{*}\bar A))
\\
&=& -Z(Lg,\mathrm{curv}(\pr_1^{*}\bar A))  
\end{eqnarray*}
\end{comment}
These show the required identity $\Delta B_{\corr}=\mathrm{curv}(\chi_{\corr})$.
\begin{comment}
Indeed:
\begin{eqnarray*}
\Delta B &=& \frac{1}{2}\Delta\omega(\bar A \wedge \bar A) +\Delta r(\mathrm{curv}(\bar A))-\frac{1}{2}\mathrm{d}\Delta\zeta
\\&=&-\frac{1}{2}Z(L\delta,[\pr_2^{*}\bar A\wedge\pr_2^{*}\bar A])-\omega(\mathrm{Ad}_{L\delta}^{-1}(\pr_2^{*}\bar A) \wedge  L\delta^{*}\theta)-\frac{1}{2}\omega( L\delta^{*}\theta\wedge L\delta^{*}\theta)\\&& +Z(L\delta,\mathrm{curv}(\pr_2^{*}\bar A))  -\frac{1}{2}\mathrm{d}\Delta\zeta
\\&=& -\frac{1}{2}L\delta^{*}\omega(\theta\wedge\theta)+Z(L\delta,\pr_2^{*}\mathrm{d}\bar A)+\omega(L\delta^{*}\theta\wedge \mathrm{Ad}^{-1}_{L\delta}(\pr_2^{*}\bar A))-\frac{1}{2}\Delta\mathrm{d}\zeta
\\&=& \mathrm{curv}(\chi_{\corr})
\end{eqnarray*}
\end{comment}
\end{proof}

It is worthwhile to compare the connection $(\chi_{\corr},B_{\corr})$ on $\mathcal{S}_{LM}$ with another connection developed by Gomi \cite{gomi3} for general lifting gerbes (not only for \emph{loop} group extensions). That connection  takes as input data just the splitting $\split$ of the Lie algebra extension and the reduction $r$ adapted to $\split$. It is defined by
\begin{equation*}
\chi_{Go} = L\delta^{*}\nu_{\split} + Z(L\delta,\pr_2^{*}\bar A)\in \Omega^1(\q)\text{,}
\end{equation*}
where $\nu_{\split}$ is the connection on $\lspinhat n$  determined by $\split$. The corresponding curving is given by
\begin{equation*}
B_{Go} := \frac{1}{2}\omega(\bar A \wedge \bar A) +r(\mathrm{curv}(\bar A)) \in \Omega^2(LFM)\text{.}
\end{equation*}
Since connections on bundle gerbes form an affine space \cite{murray}, we obtain the following.

\begin{lemma}
The assignment
\begin{equation*}
x \mapsto (\chi_x,B_{x}) := ( \chi_{Go} -\frac{x}{2}\Delta\zeta,B_{Go}-\frac{x}{2}\mathrm{d}\zeta)
\end{equation*}
is a one-parameter family of connections on the spin lifting gerbe $\mathcal{S}_{LM}$, which contains the connection of Gomi at $x=0$ and the connection $(\chi_{\corr},B_{\corr})$ at $x=1$. 
\end{lemma}

\begin{comment}
The connection is superficial if and only if $t=1$. 
\end{comment}

We recall from Section \ref{sec:lifting} that  spin structures on $LM$ correspond to trivializations of the spin lifting gerbe $\mathcal{S}_{LM}$, under the assignment of sending a spin structure $(\inf S,\sigma)$ to the trivialization $(T_{\inf S},\kappa_{\inf S})$ consisting of the principal $\ueins$-bundle $T_\inf S$ over $LFM$, and of the bundle isomorphism $\kappa_{\inf S}: \pr_2^{*}T_{\inf S} \otimes \q \to \pr_1^{*}T_{\inf S}$ defined by $\kappa_{\inf S}(t \otimes q):=t\cdot q$. 
To Gomi's connection on the spin lifting gerbe, and to the connection $\omega_{\Omega,0}$ on $T_{\inf S}$ applies a general lifting theorem, see \cite{gomi3} and \cite[Theorem 2.2]{waldorf13}, which in the present situation has the following form.

\begin{proposition}
\label{prop:liftcon0}
The assignment $(\inf S,\sigma,\Omega) \mapsto (T_{\inf S},\kappa_{\inf S},\omega_{\Omega,0})$
induces an equivalence of categories:
\begin{equation*}
\spstcon
\cong
\bigset{12em}{Trivializations of $\mathcal{S}_{LM}$ with connection compatible with  $(\chi_{Go},B_{Go})$}\text{.}
\end{equation*}
\end{proposition}

We recall that a connection on a trivialization $\mathcal{T}=(T,\kappa)$ is a connection $\omega$ on $T$, and it is called \emph{compatible} with a connection $(\chi,B)$ on the lifting gerbe if $\kappa$ is connection-preserving.  The curving $B$ is  used in order to associate to each compatible connection on $\mathcal{T}$ a \emph{covariant derivative}: a 2-form $\rho_{\mathcal{T}}\in \Omega^2(LM)$ uniquely determined by the condition that  $L\pi^{*}\rho_{\mathcal{T}}=\mathrm{curv}(\omega)+ B$.
\begin{comment}
That sign is good:
\begin{eqnarray*}
\pr_2^{*}(\mathrm{curv}(\omega) + B) 
&=& 
\pr_2^{*}\mathrm{curv}(\omega) + \pr_2^{*}B \\
&=& 
-\mathrm{curv}(\q) + \pr_1^{*}\mathrm{curv}(\omega) + \pr_2^{*}B
\\&=& \pr_1^{*}(B+\mathrm{curv}(\omega))\text{.}
\end{eqnarray*}
\end{comment}

Together with Lemma \ref{lem:omegafamily} we deduce the following result.

\begin{corollary}
Under the equivalence of Proposition \ref{prop:liftcon0}, the scalar curvature of a geometric spin connection corresponds to  the covariant derivative of a trivialization, i.e.
\begin{equation*}
L\pi^{*}\mathrm{scurv}(\Omega) = \mathrm{curv}(\omega_{\Omega,0})+B_{Go}\text{.}
\end{equation*}
\end{corollary}

Trivializations with compatible connections together with the connection-preserving isomorphisms between trivializations form a category, which is, analogously to \erf{eq:actiontrivspin} a torsor over the monoidal category $\ubuncon{LM}$ of principal $\ueins$-bundles with connection over $LM$.
The equivalence of Proposition \ref{prop:liftcon0} is equivariant with respect to the $\ubuncon{LM}$-actions on both categories; in particular, we have the following consequence.

\begin{corollary}
The category $\spstcon$ is a torsor over the monoidal category $\ubuncon{LM}$ of principal $\ueins$-bundles with connection over $LM$. 
\end{corollary}

We want to generalize the equivalence of Proposition \ref{prop:liftcon0} to a version for the connection $\omega_{\Omega,x}$ for all $x\in \R$, and so in particular to the case $x=1$. In order to do so, we have the following result.  

\begin{lemma}
\label{lem:connshift}
The assignment $(T,\kappa,\omega) \mapsto (T,\kappa,\omega + \frac{x}{2}\zeta)$
induces  an equivalence of categories:
\begin{equation*}
\bigset{12em}{Trivializations of $\mathcal{S}_{LM}$ with connection compatible with   $(\chi_{0},B_{0})$} \cong \bigset{10.5em}{Trivializations of $\mathcal{S}_{LM}$ with connection compatible with $(\chi_{x},B_{x})$}\text{.}
\end{equation*}
Moreover, the equivalence is equivariant with respect to the $\ubuncon{LM}$-actions, and it preserves the covariant derivative of trivializations.
\end{lemma}

\begin{proof}
It is enough to show that the given functor is well-defined for all $x\in \R$; it is then invertible by the functor associated to $-x$.
For well-definedness it suffices to show that the given isomorphism $\kappa$ is connection-preserving for the shifted connections. Indeed, in 
\begin{equation*}
\kappa: \pr_2^{*}T \otimes \q \to \pr_1^{*}T
\end{equation*}
we shift on the right hand side by $\frac{x}{2}\pr_1^{*}\zeta$ and on the left hand side by $\frac{x}{2}\pr_2^{*}\zeta -\frac{x}{2}\Delta \zeta=\frac{x}{2}\pr_1^{*}\zeta$; thus, $\kappa$ is connection-preserving. 

The equivariance under the $\ubuncon {LM}$-actions follows directly from the definitions. If $\rho$ is the covariant derivative of $(T,\kappa,\omega)$ with respect to $B_0$, i.e. $L\pi^{*}\rho = \mathrm{curv}(\omega) + B_{0}$, then the covariant derivative of $(T,\kappa,\omega+\frac{x}{2}\zeta)$ with respect to $B_{t}=B_{0}-\frac{x}{2}\mathrm{d}\zeta$ is the same $\rho$, since 
\begin{equation*}
L\pi^{*}\rho = \mathrm{curv}(\omega) + B_{0} = \mathrm{curv}(\omega+\frac{x}{2}\zeta) - \frac{x}{2}\mathrm{d}\zeta+B_0=\mathrm{curv}(\omega+\frac{x}{2}\zeta)+B_t\text{.}
\end{equation*}
\end{proof}

From Proposition \ref{prop:liftcon0} and Lemma \ref{lem:connshift} we obtain for each $x\in \R$ an equivalence between geometric spin structures on $LM$ and trivializations of the spin lifting gerbe equipped with the connection $(\chi_x,B_x)$. In particular, we have for $x=1$:

\begin{theorem}
\label{th:spinconlift}
The assignment $(\inf S,\sigma,\Omega) \mapsto (T_{\inf S},\kappa_{\inf S},\omega_{\Omega})$
induces an equivalence of categories
\begin{equation*}
\spstcon  \cong 
\bigset{12em}{Trivializations of  $\mathcal{S}_{LM}$ with connection compatible with $(\chi_{\corr},B_{\corr})$}\text{.}
\end{equation*}
This equivalence is equivariant for the $\ubuncon{LM}$-actions, and the scalar curvature of a geometric spin structure corresponds to the covariant derivative of the trivialization. 
\end{theorem}

Definition \ref{def:superficialspinconnection} of superficial spin connections shows   that the equivalence of Theorem \ref{th:spinconlift} exchanges superficial spin connections $\Omega$ with trivializations of $\mathcal{S}_{LM}$ whose connection $\omega_{\Omega}$ is superficial. Likewise, Definition \ref{def:geometricfusion} of geometric fusion spin structures shows that the equivalence persists in the setting with fusion products, where it becomes the following result.

\begin{corollary}
\label{co:equivsupfusion}
The assignment $(\inf S,\sigma,\lambda,\Omega) \mapsto (T_{\inf S},\kappa_{\inf S},\lambda,\omega_{\Omega})$
induces an equivalence of categories,
\begin{equation*}
\spstconsffus \cong \bigset{15em}{Fusion trivializations of $\mathcal{S}_{LM}$ with superficial fusive connection compatible with $(\chi_{\corr},B_{\corr})$}\text{.}
\end{equation*}
This equivalence is equivariant under the action of the monoidal category $\ufusbunconsf{LM}$.
\end{corollary}

To close, we observe by inspection that the passage from a setting with connections to a setting without connections is consistent with the lifting theory, i.e. with Proposition \ref{prop:equivthinfusion} and Corollary \ref{co:equivsupfusion}): there is a commutative diagram of categories and functors 
\begin{equation}
\label{eq:commdiagleft}
\alxydim{@=\xypicst}{\spstconsffus  \ar[r] \ar[d] & *++{\bigset{15em}{Fusion trivializations of $\mathcal{S}_{LM}$ with superficial fusive connection compatible with $(\chi_{\corr},B_{\corr})$}} \ar@{->}[d] \\  \spstthfus \ar[r] & \trivthfus{\mathcal{S}_{LM}}\text{,}}
\end{equation}
with the horizontal functors equivalences of categories.

\section{String structures and decategorification}

\label{sec:stringstructures}

String structures as defined in \cite{waldorf8} form a bicategory.  The main point of this section is to introduce several decategorified versions of this bicategory of string structures,  which are tailored into a form that allows a direct application of the duality between gerbes and $\ueins$-bundles over loop spaces, see Section \ref{sec:fusionextensions}. 

\subsection{String structures as trivializations}

The idea behind string structures as defined in \cite{waldorf8} is to realize the  class $\frac{1}{2}p_1(M)\in \h^4(M,\Z)$ using bundle 2-gerbes. 
\begin{comment}
Since we anyway need some details in Section \ref{sec:trivializations}, we shall recall the definition of a bundle 2-gerbe.  
\end{comment}

\begin{definition}[{{\cite[Definition 5.3]{stevenson2}}}]
\label{def:bundle2gerbe}
A \emph{bundle 2-gerbe} over a smooth manifold $M$ is a surjective submersion $\pi:Y \to M$ together with a bundle gerbe $\mathcal{P}$  over $Y^{[2]}$, an isomorphism 
\begin{equation*}
\mathcal{M}: \pr_{23}^{*}\mathcal{P} \otimes \pr_{12}^{*}\mathcal{P} \to \pr_{13}^{*}\mathcal{P}
\end{equation*}
of bundle gerbes over $Y^{[3]}$, and a transformation 
\begin{equation*}
\alxydim{@=\xypicst@C=5em}{\pr_{34}^{*}\mathcal{P} \otimes \pr_{23}^{*}\mathcal{P} \otimes \pr_{12}^{*}\mathcal{P} \ar[r]^-{\pr_{234}^{*}\mathcal{M} \otimes \id} \ar[d]_{\id \otimes \pr_{123}^{*}\mathcal{M}} & \pr_{24}^{*}\mathcal{P} \otimes \pr_{12}^{*}\mathcal{P} \ar@{=>}[dl]|*+{\mu} \ar[d]^{\pr_{124}^{*}\mathcal{M}} \\ \pr_{34}^{*}\mathcal{P} \otimes \pr_{13}^{*}\mathcal{P} \ar[r]_-{\pr_{134}^{*}\mathcal{M}} & \pr_{14}^{*}\mathcal{P}}
\end{equation*}
over $Y^{[4]}$ that satisfies a  pentagon axiom. 
\begin{comment}
It is shown in Figure \ref{fig:pentagon}.
\end{comment}
\end{definition}

\begin{comment}
\begin{figure}[h]
\begin{equation*}
\alxydim{@C=0.5cm@R=1.2cm}{&&\ast \ar@{=>}[dll]_*+{\id \circ (\pr_{1234}^{*}\mu \otimes \id)} \ar@{=>}[drr]^*+{\pr_{1345}^{*}\mu \circ \id}&&\\\ast \ar@{=>}[dr]_*+{\pr_{1245}^{*}\mu \circ \id}&&&&\ast \ar@{=>}[dl]^*+{\pr_{1235}^{*}\mu \circ \id}\\&\ast \ar@{=>}[rr]_*+{\id \circ (\id \otimes \pr_{2345}^{*}\mu)}&&\ast&}
\end{equation*}
\caption{The pentagon axiom for a bundle gerbe product $\mu$. It is an equation between transformations over $Y^{[5]}$. }
\label{fig:pentagon}
\end{figure}
\end{comment}

The isomorphism $\mathcal{M}$ is called the \emph{bundle 2-gerbe product} and the transformation $\mu$ is called the  \emph{associator}. The pentagon axiom implies the cocycle condition for a certain degree three \v Cech cocycle on $M$ with values in $\ueins$, which defines -- via the exponential sequence -- a class 
\begin{equation*}
\mathrm{CC}(\mathbb{G}) \in \h^4(M,\Z)\text{;}
\end{equation*}
see \cite[Proposition 7.2]{stevenson2} for the details.

We recall from \cite{carey4} the construction of the Chern-Simons  2-gerbe $\mathbb{CS}_M$, whose characteristic class is $\mathrm{CC}(\mathbb{CS}_M)=\frac{1}{2}p_1(M)$.
It uses the basic gerbe $\gbas$ over $\spin n$, together with its multiplicative structure $(\mathcal{M},\alpha)$ described in Section \ref{sec:centralextension}. Here we first ignore the connections -- they  become relevant in Section \ref{sec:geometricstringstructures}. The Chern-Simons  2-gerbe $\mathbb{CS}_M$ consists of the following structure: 
\begin{itemize}

\item 
Its surjective submersion is the frame bundle $\pi:FM \to M$.  

\item
Its bundle gerbe $\mathcal{P}$ over $FM^{[2]}$ is $\mathcal{P} := \delta^{*}\mathcal{G}_{bas}$, where $\delta:FM^{[2]} \to \spin n$ is the difference map (i.e. $p' \cdot \delta(p,p') = p$).

\item
Its bundle 2-gerbe product is
\begin{equation*}
\mathcal{M}' :=\delta_2^{*}\mathcal{M}: \pr_{23}^{*}\mathcal{P} \otimes \pr_{12}^{*}\mathcal{P} \to \pr_{13}^{*}\mathcal{P}\text{,}
\end{equation*}
where  $\delta_2: FM^{[3]} \to \spin n^2$ is defined by $(p'',p') \cdot \delta_2(p,p',p'') = (p',p)$.

\item
Its associator is $\mu := \delta_3^{*}\alpha$, where $\delta_3: FM^{[4]}\to \spin n$ is defined analogously.
The pentagon axiom for $\alpha$ implies  the  pentagon  axiom 
\begin{comment}
(Figure \ref{fig:pentagon}) 
\end{comment}
for $\mu$. 

\end{itemize}
More detailed discussions of  the Chern-Simons 2-gerbe are given in \cite{carey4,waldorf5,waldorf8,Nikolausa}.

\begin{definition}[{{\cite[Definition 11.1]{stevenson2}}}]
\label{deftriv}
A \emph{trivialization} of a bundle 2-gerbe $\mathbb{G}$ as in Definition \ref{def:bundle2gerbe} is a bundle gerbe $\mathcal{S}$ over $Y$, together with an  isomorphism
\begin{equation*}
\mathcal{A}: \pr_2^{*}\mathcal{S} \otimes \mathcal{P} \to \pr_1^{*}\mathcal{S}
\end{equation*}
of bundle gerbes over $Y^{[2]}$ and a connection-preserving transformation
\begin{equation}
\label{eq:diagfillsigma}
\alxydim{@=\xypicst@C=4em}{\pr_{3}^{*}\mathcal{S} \otimes \pr_{23}^{*}\mathcal{P} \otimes \pr_{12}^{*}\mathcal{P} \ar[r]^-{\pr_{23}^{*}\mathcal{A} \otimes \id} \ar[d]_{\id \otimes \mathcal{M}} & \pr_{2}^{*}\mathcal{S} \otimes \pr_{12}^{*}\mathcal{P} \ar@{=>}[dl]|*+{\sigma} \ar[d]^{\pr_{12}^{*}\mathcal{A}} \\ \pr_{3}^{*}\mathcal{S} \otimes \pr_{13}^{*}\mathcal{P} \ar[r]_-{\pr_{13}^{*}\mathcal{A}} & \pr_1^{*}\mathcal{S}}
\end{equation}
over $Y^{[3]}$ that is compatible with the associator $\mu$ in the sense of Figure \ref{compass}.
\begin{comment}
In full beauty, we have drawn the compatibility condition as a pasting diagram in Figure \ref{compasspasting}.
\end{comment}
\end{definition}

\begin{figure}[h]
\begin{footnotesize}
\begin{equation*}
\alxydim{@C=0.5cm@R=1.2cm}{&&\ast \ar@{=>}[dll]_*+{\pr^{*}_{123}\sigma \circ \id} \ar@{=>}[drr]^*+{\id \circ (\id \otimes \pr_{234}^{*}\sigma)}&&\\\ast \ar@{=>}[dr]_*+{\pr_{134}^{*}\sigma \circ \id}&&&&\ast \ar@{=>}[dl]^*+{\pr_{124}^{*}\sigma \circ \id}\\&\ast \ar@{=>}[rr]_*+{\id \circ (\mu \otimes \id)}&&\ast&}
\end{equation*}
\end{footnotesize}
\caption{The compatibility condition between the associator $\mu$ of a bundle 2-gerbe and the transformation $\sigma$ of a trivialization. It is an equation of transformations over $Y^{[4]}$.}
\label{compass}
\end{figure}

The characteristic class $\mathrm{CC}(\mathbb{G}) \in \h^4(M,\Z)$ of $\mathbb{G}$ vanishes if and only if $\mathbb{G}$ admits a trivialization \cite[Proposition 11.2]{stevenson2}. In particular, the Chern-Simons 2-gerbe $\mathbb{CS}_M$ has trivializations if and only if $M$ is a string manifold.
This is the motivation for the following definition. 

\begin{definition}[{{\cite[Definition 1.1.5]{waldorf8}}}]
\label{def:stringstructure}
A \emph{string structure} on $M$ is a trivialization $\mathbb{T}$ of  $\mathbb{CS}_M$. 
\end{definition}

The main problem with establishing a relation between string structures and  loop space geometry via the transgression and regression functors of Section \ref{sec:fusionextensions} is that these functors are defined on the truncated   \emph{categories} $\hc 1 \ugrbcon X$ and $\hc 1 \ugrb X$ of bundle gerbes and not on the full \emph{bicategories}. This problem is solved in the next subsections by reformulating the notion of trivializations of bundle 2-gerbes internal to these truncated categories.

\subsection{Decategorification of trivializations}

\label{sec:trivializations}

In this section  $\mathbb{G}$ is a general bundle 2-gerbe  over $M$, composed of the same structure as in Definition \ref{def:bundle2gerbe}. According to \cite[Lemma 2.2.4]{waldorf8}, trivializations  of $\mathbb{G}$ form a bicategory, which we denote by $\triv {\mathbb{G}}$. We recall how the 1-morphisms and 2-morphisms are defined. 

Given   trivializations $\mathbb{T}=(\mathcal{S},\mathcal{A},\sigma)$ and $\mathbb{T}'=(\mathcal{S}',\mathcal{A}',\sigma')$ of $\mathbb{G}$, a \emph{1-morphism} $\mathbb{B}: \mathbb{T} \to \mathbb{T}'$ in  $\triv {\mathbb{G}}$ is an  isomorphism $\mathcal{B}: \mathcal{S} \to \mathcal{S}'$ between bundle gerbes over $Y$ together with a  transformation
\begin{equation}
\label{eq:beta}
\alxydim{@=\xypicst}{\pr_2^{*}\mathcal{S} \otimes \mathcal{P} \ar[r]^-{\mathcal{A}} \ar[d]_{\id \otimes \pr_2^{*}\mathcal{B}} & \pr_1^{*}\mathcal{S} \ar@{=>}[dl]|*+{\beta} \ar[d]^{\pr_1^{*}\mathcal{B}} \\ \pr_2^{*}\mathcal{S}' \otimes \mathcal{P}\ar[r]_-{\mathcal{A}'} & \pr_1^{*}\mathcal{S}'}
\end{equation} 
over $Y^{[2]}$ that is compatible with the transformations $\sigma$ and $\sigma'$ in the sense of the pentagon diagram shown in Figure \ref{compmorph}.
\begin{figure}[h]
\begin{footnotesize}
\begin{equation*}
\alxydim{@C=0.5cm@R=1.2cm}{&&\ast \ar@{=>}[dll]_*+{\pr_{12}^{*}\beta \circ \id} \ar@{=>}[drr]^*+{\id \circ \pr_{23}^{*}\beta}&&\\\ast \ar@{=>}[dr]_*+{\id \circ \sigma}&&&&\ast \ar@{=>}[dl]^*+{\sigma' \circ \id}\\&\ast \ar@{=>}[rr]_*+{\pr_{13}^{*}\beta \circ \id}&&\ast&}
\end{equation*}
\end{footnotesize}
\caption{The compatibility between the transformations $\sigma$ and $\sigma'$ of two trivializations $\mathbb{T}$ and $\mathbb{T}'$ and the transformation $\beta$ of a 1-morphism $\mathbb{B}=(\mathcal{B},\beta)$ between $\mathbb{T}$ and $\mathbb{T}'$. It is an equation of transformations over $Y^{[3]}$.}
\label{compmorph}
\end{figure}
If $\mathbb{B}_1=(\mathcal{B}_1,\beta_1)$ and $\mathbb{B}_2=(\mathcal{B}_2,\beta_2)$ are 1-morphisms between $\mathbb{T}$ and $\mathbb{T}'$, a \emph{2-morphism} is a  transformation $\varphi: \mathcal{B}_1 \Rightarrow \mathcal{B}_2$ that is compatible with the transformations $\beta_1$ and $\beta_2$ in such a way  that the diagram
\begin{equation}
\label{eq:2morphtriv}
\alxydim{@=\xypicst}{\pr_1^{*}\mathcal{B}_1 \circ \mathcal{A} \ar@{=>}[d]_{\pr_1^{*}\varphi \circ \id} \ar@{=>}[r]^-{\beta_1} & \mathcal{A}' \circ (\pr_2^{*}\mathcal{B}_1 \otimes \id) \ar@{=>}[d]^{\id \circ (\pr_2^{*}\varphi \otimes \id)} \\ \pr_1^{*}\mathcal{B}_2 \circ \mathcal{A} \ar@{=>}[r]_-{\beta_2} & \mathcal{A}' \circ (\pr_2^{*}\mathcal{B}_2 \otimes \id)}
\end{equation}
of transformations over $Y^{[2]}$ is commutative.

The  bicategory $\triv{\mathbb{G}}$ is a module over the monoidal bicategory $\ugrb M$ in terms of an action 2-functor
\begin{equation}
\label{eq:action}
\ugrb M \otimes \triv{\mathbb{G}} \to \triv{\mathbb{G}}: (\mathcal{K} , \mathbb{T}) \mapsto \mathcal{K} \otimes \mathbb{T}\text{,}
\end{equation}
For a trivialization $\mathbb{T}=(\mathcal{S},\mathcal{A},\sigma)$ and a bundle gerbe $\mathcal{K}$ over $M$, the  trivialization $\mathcal{K} \otimes \mathbb{T}$ is given by the bundle gerbe $\pi^{*}\mathcal{K} \otimes \mathcal{S}$, the isomorphism $\id \otimes \mathcal{A}$ and the transformation $\id \otimes \sigma$. We recall the following result. 

\begin{lemma}[{{\cite[Lemma 2.2.5]{waldorf8}}}]
\label{lem:torsor}
The action \erf{eq:action} exhibits the bicategory $\triv {\mathbb{G}}$ as a torsor over the monoidal bicategory $\ugrb M$.
\end{lemma}

There are two methods to produce a category from the bicategory $\triv {\mathbb{G}}$ of trivializations of a bundle 2-gerbe. The first method is to take the truncation $\hc 1 \triv{\mathbb{G}}$, whose objects are those of $\triv {\mathbb{G}}$, and whose morphisms are 2-isomorphism classes of 1-morphisms in $\triv {\mathbb{G}}$. 

The second method is to consider the truncated presheaf of categories $\hc 1 \ugrb-$ and then formally repeat the definition of trivializations in that ambient category. This gives a category which we denote by $\trivtrunc {\mathbb{G}}$.
An object in $\trivtrunc {\mathbb{G}}$ is a pair $(\mathcal{S},[\mathcal{A}])$ of a bundle gerbe $\mathcal{S}$  over $Y$ and an equivalence class of  isomorphisms
\begin{equation*}
\mathcal{A}: \pr_2^{*}\mathcal{S} \otimes \mathcal{P} \to \pr_1^{*}\mathcal{S}\text{,}
\end{equation*}
such that there exists a  transformation $\sigma$ as in \erf{eq:diagfillsigma}.
\begin{comment}
It is easy to see that if one representative $\mathcal{A}$ of the class $[\mathcal{A}]$ admits such a transformation, than any other representative admits one, too. \end{comment}
Note that it is \emph{not} required that   $\sigma$ makes the diagram of Figure \ref{compass} commutative. 
A morphism $(\mathcal{S}_1,[\mathcal{A}_1]) \to (\mathcal{S}_2,[\mathcal{A}_2])$ in $\trivtrunc{\mathbb{G}}$ is an equivalence class $[\mathcal{B}]$ of  isomorphisms $\mathcal{B}: \mathcal{S}_1 \to \mathcal{S}_2$ such that there exists a  transformation $\beta$ as in \erf{eq:beta}.
\begin{comment}
Again, if one choice $(\mathcal{A}_1,\mathcal{A}_2,\mathcal{B})$ of representatives admits such a transformation, than any other choice admits one, too. 
\end{comment}
Note that it is  \emph{not} required that   $\beta$ makes the diagram of Figure \ref{compmorph} commutative.

The two categories of trivializations are related by a functor
\begin{equation*}
\mathrm{t}_1: \hc 1 \triv{\mathbb{G}} \to \trivtrunc {\mathbb{G}}\text{.}
\end{equation*}
Indeed, an object in $\hc 1 \triv{\mathbb{G}}$ is just an object $(\mathcal{S},\mathcal{A},\sigma)$ in $\triv{\mathbb{G}}$, and the functor sets $\mathrm{t}_1(\mathcal{S},\mathcal{A},\sigma) \df (\mathcal{S},[\mathcal{A}])$. A morphism in $\hc 1 \triv{\mathbb{G}}$ is a 2-isomorphism class $[(\mathcal{B},\beta)]$ of 1-morphisms, and the constructions above show that  $\mathrm{t}_1([(\mathcal{B},\beta)]) := [\mathcal{B}]$ is well-defined.

We recall that $\triv{\mathbb{G}}$ is a torsor  over $\ugrb M$, see Lemma \ref{lem:torsor}. By purely formal reasons, it follows that $\hc 1 \triv {\mathbb{G}}$ is a torsor  over $\hc 1 \ugrb M$. There is a similar action of $\hc 1 \ugrb M$ on $\trivtrunc {\mathbb{G}}$, i.e. a functor
\begin{equation}
\label{eq:actiontrunc}
\hc 1 \ugrb M \times \trivtrunc{\mathbb{G}}   \to \trivtrunc{\mathbb{G}}
\end{equation} 
that exhibits $\trivtrunc{\mathbb{G}}$ as a module category over the monoidal category $\hc 1 \ugrb M$.
On objects, it is given by $(\mathcal{K},(\mathcal{S},[\mathcal{A}])) \mapsto (\pi^{*}\mathcal{K} \otimes \mathcal{S},[\mathcal{A} \otimes \id])$, and on morphisms it is given by $([\mathcal{J}],[\mathcal{B}]) \mapsto [\pi^{*}\mathcal{J} \otimes \mathcal{B}]$.
The  functor $\mathrm{t}_1$ is obviously $\hc 1 \ugrb M$-equivariant.

Next we produce three  sets  from the bicategory $\triv{\mathbb{G}}$ of trivializations. The first is the set $\hc 0 \triv{\mathbb{G}}$ of isomorphism classes of trivializations of $\mathbb{G}$. The second is the set $\hc 0 (\trivtrunc{\mathbb{G}})$ of isomorphism classes of objects in $\trivtrunc{\mathbb{G}}$. The third is the set $\trivttrunc {\mathbb{G}}$ obtained by formally repeating the definition of a trivialization ambient to the  presheaf $\hc 0 \ugrb -$. 
 In detail, an element of  $\trivttrunc {\mathbb{G}}$ is an isomorphism class $[\mathcal{S}]$ of bundle gerbes $\mathcal{S}$  over $Y$, such that there exists an isomorphism
$\pr_2^{*}\mathcal{S} \otimes \mathcal{P} \cong \pr_1^{*}\mathcal{S}$ over $Y^{[2]}$.
\begin{comment}
It is clear that if one representative $\mathcal{S}$ admits such an isomorphism, than any other representative admits one, too.  
\end{comment}

The three sets of trivializations are related by maps
\begin{equation*}
\alxydim{}{\hc 0 \triv{\mathbb{G}}  \ar[r]^-{\hc 0 \mathrm{t}_1} & \hc 0 (\trivtrunc {\mathbb{G}}) \ar[r]^-{\mathrm{t}_2} & \trivttrunc{\mathbb{G}}\text{,}}
\end{equation*}
where $\hc 0 \mathrm{t}_1$ is the map induced by the functor $\mathrm{t}_1$ on isomorphism classes, and $\mathrm{t}_2$ sends an element $[(\mathcal{S},[\mathcal{A}])]$ to $[\mathcal{S}]$. Again by purely formal reasons, $\hc  0 \triv {\mathbb{G}}$ is a torsor over the group $\hc 0 \ugrb M$. Further, the group $\hc 0 \ugrb M$ acts on $\hc 0 (\trivtrunc{\mathbb{G}})$, and $\hc 0 \mathrm{t}_1$ is equivariant. Finally, we have an action of $\hc 0 \ugrb M$ on $\trivttrunc{\mathbb{G}}$, defined by $([\mathcal{K}],[\mathcal{S}]) \mapsto [\pi^{*}\mathcal{K} \otimes \mathcal{S}]$, for which the map $\mathrm{t}_2$ is equivariant.

\begin{comment} 
For a general bundle 2-gerbe, the two categories of trivializations are not  equivalent, and the three sets of trivializations are not the same. 
In the next section we show that in case of the Chern-Simons 2-gerbe we obtain best possible  coincidence.
\end{comment}

\subsection{From string structures to string classes}

\label{sec:stringclasses}

\begin{comment}
We recall that a string structure is a trivialization of the Chern-Simons 2-gerbe $\mathbb{CS}_M$.
\end{comment}
We now have the following versions of string structures:
\begin{enumerate}[(a)]
\item 
a \emph{bicategory of string structures}, $\stst\relax M := \triv{\mathbb{CS}_M}$, 

\item
two \emph{categories of string structures}, $\hc  1 \stst\relax M$ and $\stst1M := \trivtrunc{\mathbb{CS}_M}$, related by a functor 
\begin{equation*}
\mathrm{t}_1: \hc 1 \stst\relax M \to \stst1M\text{,} \end{equation*}

\item
three \emph{sets of string structures}, namely  $\hc 0 \stst\relax M$, $\hc 0 \stst 1M$ and $\stst 0M \df \trivttrunc{\mathbb{CS}_M}$, related by maps
\begin{equation*}
\alxydim{}{\hc 0 \stst\relax M  \ar[r]^-{\hc 0 \mathrm{t}_1} & \hc 0 \stst1M \ar[r]^-{\mathrm{t}_2} & \stst0M\text{.}}
\end{equation*}
\end{enumerate}
Additionally, there is a fourth set consisting of so-called string classes. A \emph{string class} on $M$ is a class $\xi\in \h^3(FM,\Z)$ that restricts on each fibre to the generator  $\gamma \in \h^3(\spin n,\Z)$. We denote the set of string classes on $M$ by $\sc(M)$. We have a map
\begin{equation*}
\mathrm{t}_3:\stst0M \to \sc(M) : [\mathcal{S}] \mapsto \mathrm{DD}(\mathcal{S})\text{.}
\end{equation*}
This map is well-defined: indeed, for a point $p\in FM$ we have the map $\iota_p\maps \spin n \to FM$ defined by $\iota_p(g) := pg$, implementing the \quot{restriction to the fibre of $p$}. We have another map  $j_p\maps \spin n \to FM^{[2]}$ defined by $j_p( g) := (pg,p)$. Recall that an element $[\mathcal{S}]$ in $\stst0M$
is represented by a bundle gerbe $\mathcal{S}$ over $FM$ that admits an isomorphism $\pr_2^{*}\mathcal{S} \otimes \delta^{*}\gbas \cong \pr_1^{*}\mathcal{S}$.  Pullback along $j_p$ followed by taking the Dixmier-Douady class yields $\gamma = \mathrm{DD}(\gbas) = \iota_p^{*}\mathrm{DD}(\mathcal{S})$. Thus, $\mathrm{DD}(\mathcal{S})$ is a string class.

The set $\sc (M)$ of string classes carries an action of $\h^3(M,\Z)$ via pullback to $FM$ and addition, and under the identification $\hc 0 \ugrb M\cong \h^3(M,\Z)$ the map $\mathrm{t}_3$ is equivariant.
\begin{comment}
In the remainder of this section we prove the following theorem, showing that the two categories of string structures are equivalent, and the four sets of string structures are all the same. 
\end{comment}

\begin{theorem}
\label{th:truncation}
The functor
\begin{equation*}
\mathrm{t}_1: \hc 1 \stst\relax M \to \stst1M
\end{equation*}
is an equivalence of categories; in particular, it is an equivariant functor between $\hc 1 \ugrb M$-torsors.
The maps
\begin{equation*}
\alxydim{@C=\xypicst}{\hc 0 \stst\relax M  \ar[r]^-{\hc 0 \mathrm{t}_1} & \hc 0 \stst1M \ar[r]^-{\mathrm{t}_2} & \stst0M \ar[r]^-{\mathrm{t}_3} & \sc(M)}
\end{equation*}
are all bijections; in particular, they are equivariant maps between $\h^3(M,\Z)$-torsors. 
\end{theorem}

The proof is split into a couple of lemmata.
 We start with the following Serre spectral sequence calculation.

\begin{lemma}
\label{lem:serreall}
Let $\pi:P \to M$ be a principal $G$-bundle over $M$, for $G$ a compact, simple, simply-connected Lie group. Let $A$ be an abelian group. For $k=0,1,2$ the pullback map
\begin{equation*}
\pi^{*}: \h^k(M,A) \to \h^k(P,A)
\end{equation*}
is an isomorphism. Let $p\in P$ be a point and $i_p: G \to P: g\mapsto pg$. Then, the sequence
\begin{equation*}
\alxydim{}{0\ar[r] & \h^3(M,A) \ar[r]^-{\pi^{*}} & \h^3(P,A) \ar[r]^-{i_p^{*}} & \h^3(G,A) \ar[r]^{tr} & \h^4(M,A) } 
\end{equation*}
is exact, where $tr$ is the \quot{transgression} homomorphism of the Serre spectral sequence. 
\end{lemma}

In case of $G=\spin n$ and $A=\Z$, the transgression homomorphism $tr$ of the Serre spectral sequence sends the generator $\gamma\in \h^3(\spin n,\Z)$ to the class $\frac{1}{2}p_1(M) \in \h^4(M,\Z)$. Thus, we obtain the following result about string classes.

\begin{corollary}[{{\cite[Proposition 6.1.5]{redden1}}}]
\label{co:stringclasses}
Let $M$ be a spin manifold. 
\begin{enumerate}[(i)]

\item 
\label{co:stringclasses:exist}
String classes exist if and only if $M$ is a string manifold. 

\item
\label{co:stringclasses:torsor}
The set of string classes $\sc (M)$ is a torsor over $\h^3(M,\Z)$.

\end{enumerate}
\end{corollary}

Now we are in position to contribute first partial results to the proof of Theorem \ref{th:truncation}.

\begin{lemma}
\label{lem:t3bijection}
The maps $\mathrm{t}_3$ and $\mathrm{t}_2 \circ \hc 0 \mathrm{t}_1$ are bijections.
\end{lemma} 

\begin{proof}
If $M$ is not string, then $\hc 0 \stst\relax M$ and $\sc(M)$ are empty. Hence $\stst 0M$ is also empty, and both maps in the claim are maps between empty sets, hence bijections. If $M$ is string,  $\hc 0 \stst\relax M$ and $\sc(M)$ are non-empty. 
Then $\hc 0 \stst\relax M$ and $\sc(M)$ are both torsors over $\h^3(M,\Z)$ by Lemma \ref{lem:torsor} and Corollary \ref{co:stringclasses}, and the over all composition $\mathrm{t}_3 \circ \mathrm{t}_2 \circ \hc 0 \mathrm{t}_1$ is an equivariant map between torsors over the same group,  hence a bijection. In particular, $\mathrm{t}_3$ is surjective. The definition  of $\mathrm{t}_3$ shows immediately that it is also injective. Hence, $\mathrm{t}_3$ and $t_2 \circ \hc 0 \mathrm{t}_1$ are bijections.
\end{proof}

\begin{comment}
The next two lemmata extract from the Serre spectral sequence calculation of Lemma \ref{lem:serreall} statements about the cohomology of the fibre products of a principal $G$-bundle.
\end{comment}

\begin{lemma}
\label{lem:serre}
Let $\pi:P \to M$ be a principal $G$-bundle over $M$, for $G$ a compact, simple, simply-connected Lie group. We denote by $P^{[k]}$ the $k$-fold fibre product of $P$ with itself over $M$, and by  $\pi_k: P^{[k]} \to M$ the projection. Then, 
\begin{equation*}
\pi_k^{*}: \h^p(M,A) \to \h^p(P^{[k]},A)
\end{equation*}
is an isomorphism for all $k\in \N$, $p=0,1,2$, and $A=\R,\Z,\ueins$.
\end{lemma}

\begin{proof}
Since $P$ is a principal $G$-bundle, we have diffeomorphisms 
\begin{equation*}
\varphi_k: P^{[k]} \to P \times G^{k-1}: (p_1,...,p_k) \mapsto (p_1,\delta(p_1,p_2),\delta(p_1,p_3),...,\delta(p_1,p_k))
\end{equation*}
with $\pi \circ \pr_1 \circ  \varphi_k = \pi_k$. The projection 
$\pr_1\maps P \times G^{k-1} \to P$
induces an isomorphism in the cohomology with coefficients in $A=\R,\Z$ and in degrees $p=0,1,2$ via the Künneth formula (using that $\R$ and $\Z$ have no torsion) and the 2-connectedness of $G$. For $A=\ueins$ the same statement holds due to the exactness of the exponential sequence and the five lemma.  The bundle projection $\pi: Y \to M$ induces an isomorphism in cohomology in degrees $p=0,1,2$ according to Lemma \ref{lem:serreall}. 
\end{proof}

\begin{lemma}
\label{lem:coboundary}
Suppose that $\mathcal{F}$ is a presheaf of abelian groups over smooth manifolds, and $\pi\maps Y \to M$  is a surjective submersion. We denote by $\pi_k:Y^{[k]} \to M$ the projection, and by
$\Delta: \mathcal{F}(Y^{[k]}) \to \mathcal{F}(Y^{[k+1]})$
the \v Cech coboundary operator. 
\begin{comment}
That is,
\begin{equation*}
\Delta \beta = \sum_{i=1}^{k+1}(-1)^{i} \partial_i^{*}\beta
\end{equation*}
for $\beta \in \mathcal{F}(Y^{[k]})$,
where $\partial_i: Y^{[k+1]} \to Y^{[k]}$ is the map that omits the $i$th component. 
\end{comment}
Then, 
\begin{equation*}
\Delta \circ \pi_k^{*} = 
\begin{cases}
\pi_{k+1}^{*} & k\text{ even} \\
0 & k\text{ odd.} \\
\end{cases}
\end{equation*}
\end{lemma}

\begin{proof}
For $\alpha\in \mathcal{F}(M)$ and $\beta = \pi_k^{*}\alpha$ we have $\partial_i^{*}\beta = \partial_i^{*}\pi_k^{*}\alpha = \pi_{k+1}^{*}\alpha$. Now the claim is proved by counting the number of terms in the alternating sum. 
\end{proof}

Applying Lemma \ref{lem:coboundary} to the presheaf $\mathcal{F}=\h^p(-,\Z)$, we obtain the following.

\begin{corollary}
\label{co:delta}
Let $\pi:P \to M$ be a principal $G$-bundle over $M$, for $G$ a compact, simple, simply-connected Lie group. Then, for $p=0,1,2$, 
\begin{equation*}
\Delta: \h^p(P^{[k]},\Z) \to \h^p(P^{[k+1]},\Z)
\end{equation*}
is an isomorphism if $k$ is even, and the zero map if $k$ is odd. 
\end{corollary}

Now we  finish the proof of Theorem \ref{th:truncation} with the following two lemmata.

\begin{lemma}
\label{lem:t2}
The map $\mathrm{t}_2: \hc 0 \stst1M \to \stst 0M$
is a bijection. 
\end{lemma}

\begin{proof}
By Lemma \ref{lem:t3bijection} it is surjective; hence it remains to prove injectivity. 
Suppose $(\mathcal{S}_1,[\mathcal{A}_1])$ and $(\mathcal{S}_2,[\mathcal{A}_2])$ are objects in $\stst 1M$ such that there images under $\mathrm{t}_2$ are equal, i.e. there exists an isomorphism $\mathcal{B}: \mathcal{S}_1 \to \mathcal{S}_2$. We have to construct a transformation 
\begin{equation*}
\beta: \pr_1^{*}\mathcal{B} \circ \mathcal{A}_1 \Rightarrow \mathcal{A}_2 \circ (\pr_2^{*}\mathcal{B} \otimes \id)
\end{equation*}
over $Y^{[2]}$, see \erf{eq:beta}, so that $\mathbb{B}=(\mathcal{B},\beta)$ is an isomorphism in $\stst 1M$. A priori, the two isomorphisms $\pr_1^{*}\mathcal{B} \circ \mathcal{A}_1$ and $\mathcal{A}_2 \circ (\pr_2^{*}\mathcal{B} \otimes\id)$ are not 2-isomorphic. We recall that the category $\hom(\mathcal{G},\mathcal{H})$ of isomorphisms between two fixed bundle gerbes $\mathcal{G}$ and $\mathcal{H}$ over an arbitrary smooth manifold $X$ is a torsor over the monoidal category $\ubun X$ of principal  $\ueins$-bundles over $X$, under an action functor
\begin{equation}
\label{eq:actionbundles}
\ubun{X} \times \hom(\mathcal{G},\mathcal{H}) \to \hom(\mathcal{G},\mathcal{H}): (B,\mathcal{A}) \mapsto B \otimes \mathcal{A}\text{,}
\end{equation}
see \cite{carey2}. In our situation, this means that there exists a principal $\ueins$-bundle  $B$ over $Y^{[2]}$ such that
with $\mathcal{A}_2' :=B \otimes  \mathcal{A}_2$ we do have a transformation
\begin{equation*}
\beta: \pr_1^{*}\mathcal{B} \circ \mathcal{A}_1 \Rightarrow \mathcal{A}_2' \circ (\pr_2^{*}\mathcal{B} \otimes \id)\text{.}
\end{equation*}
The bundle $B$  has a first Chern class $c_1(B) \in \h^2(Y^{[2]},\Z)$. We claim that $c_1(B)=0$, meaning that $\mathcal{A}_2' \cong \mathcal{A}_2$ and $\beta$ is the claimed transformation; thus,  
$\mathrm{t}_2$ is injective. 

Indeed, since  $(\mathcal{S}_1,[\mathcal{A}_1])$ and $(\mathcal{S}_2,[\mathcal{A}_2])$ are objects in $\stst 1M =\trivtrunc{\mathbb{G}}$, there exist transformations $\sigma_1$, $\sigma_2$ making diagram \erf{eq:diagfillsigma}  commutative. One can paste together $\sigma_1$ and $\beta$ and produce the following transformation:
\begin{equation*}
\alxydim{@C=0.7cm@R=1.3cm}{\pr_3^{*}\mathcal{S}_2 \otimes \pr_{23}^{*}\mathcal{P} \otimes \pr_{12}^{*}\mathcal{P} \ar[rrrr]^-{\pr_{23}^{*}\mathcal{A}_2' \otimes \id} \ar[ddd]_{\id \otimes \mathcal{M}}  &&&& \pr_2^{*}\mathcal{S}_2  \otimes \pr_{12}^{*}\mathcal{P}   \ar[ddd]^{\pr_{12}^{*}\mathcal{A}_2'} \\ & \hspace{-2cm}\pr_3^{*}\mathcal{S}_1 \otimes  \pr_{23}^{*}\mathcal{P} \otimes \pr_{12}^{*}\mathcal{P}  \ar@{=>}@/_1pc/[ddl]_{\id} \ar@{=>}@/^1pc/[urrr]_{\id \otimes \pr_{23}^{*}\beta} \ar[ul]_{\pr_3^{*}\mathcal{B} \otimes \id \otimes \id} \ar[rr]^-{\pr_{23}^{*}\mathcal{A}_1 \otimes \id} \ar[d]_{\mathcal{M} \otimes \id} && \pr_2^{*}\mathcal{S}_1 \otimes \pr_{12}^{*}\mathcal{P}  \ar@{=>}@/^1pc/[ddr]^{\pr_{12}^{*}\beta} \ar@{=>}[dll]|{\sigma_1} \ar[ur]_-{\pr_2^{*}\mathcal{B} \otimes \id} \ar[d]^{\pr_{12}^{*}\mathcal{A}_1} \\ & \pr_3^{*}\mathcal{S}_1  \otimes  \pr_{13}^{*}\mathcal{P} \ar@{=>}@/_/[drrr]_{\pr_{13}^{*}\beta} \ar[dl]^{\pr_3^{*}\mathcal{B} \otimes \id} \ar[rr]_-{\pr_{13}^{*}\mathcal{A}_1} && \pr_1^{*}\mathcal{S}_1 \ar[rd]_{\pr_1^{*}\mathcal{B}} \\ \pr_3^{*}\mathcal{S}_2 \otimes \pr_{13}^{*}\mathcal{P}  \ar[rrrr]_{\pr_{13}^{*}\mathcal{A}_2'} &&&& \pr_{1}^{*}\mathcal{S}_2 }
\end{equation*}
Using the relation $\mathcal{A}_2' = B \otimes  \mathcal{A}_2$ and using that the action \erf{eq:actionbundles} commutes with composition, this transformation induces another transformation
\begin{equation*}
(\pr_{12}^{*}\mathcal{A}_2 \circ (\pr_{23}^{*}\mathcal{A}_2 \otimes \id)) \otimes (\pr_{12}^{*}B \otimes \pr_{23}^{*}B \otimes \pr_{13}^{*}B^{\vee}) \Rightarrow \pr_{13}^{*}\mathcal{A}_2 \circ (\id \otimes \mathcal{M} ) 
\end{equation*}
with $B^{\vee}$ the dual bundle. On the other hand, $\sigma_2$ is a transformation
\begin{equation*}
\sigma_2:\pr_{12}^{*}\mathcal{A}_2 \circ (\pr_{23}^{*}\mathcal{A}_2 \otimes \id) \Rightarrow \pr_{13}^{*}\mathcal{A}_2 \circ(\id \otimes \mathcal{M} ) 
\end{equation*}
It follows that $\pr_{12}^{*}B \otimes \pr_{23}^{*}B \otimes \pr_{13}^{*}B^{\vee}$ must be trivializable, i.e. $\Delta c_1(B)=0$.
By Corollary \ref{co:delta} this means $c_1(B)=0$. 
\end{proof}

\begin{lemma}
\label{lem:t1equiv}
The functor $\mathrm{t}_1$ is an equivalence of categories.
\end{lemma}

\begin{proof}
We know already that
 $\hc 0 \mathrm{t}_1$ is a bijection. This implies that  $\mathrm{t}_1$ is essentially surjective, and it implies that for two objects $\mathbb{T}_1=(\mathcal{S}_1,\mathcal{A}_1,\sigma_1)$ and $\mathbb{T}_2=(\mathcal{S}_2,\mathcal{A}_2,\sigma_2)$ the Hom-sets $\hom_{\hc 1 \triv{\mathbb{CS}_M}}(\mathbb{T}_1,\mathbb{T}_2)$ and $\hom_{\trivtrunc{\mathbb{CS}_M}}(\mathrm{t}_1(\mathbb{T}_1),\mathrm{t}_1(\mathbb{T}_2))$ are either both empty or both non-empty. It remains to show that $\mathrm{t}_1$ induces in the non-empty case a bijection between these sets. We have already seen that $\mathrm{t}_1$ is equivariant with respect to the $\hc 1 \ugrb M$-action, namely the one induced from \erf{eq:action}, and the action \erf{eq:actiontrunc}. Over the objects $(\mathcal{I},\mathcal{I})$ and $(\mathbb{T}_1,\mathbb{T}_2)$ these induce actions of the group 
\begin{equation*}
\hom_{\hc 1 \ugrb M}(\mathcal{I},\mathcal{I}) = \hc 0 \hom_{\ugrb M}(\mathcal{I},\mathcal{I})\stackerf{eq:actionbundles}{=}\hc 0 \ubun M
\end{equation*}
on the sets $\hom_{\hc 1 \triv{\mathbb{CS}_M}}(\mathbb{T}_1,\mathbb{T}_2)$ and $\hom_{\trivtrunc{\mathbb{CS}_M}}(\mathrm{t}_1(\mathbb{T}_1),\mathrm{t}_1(\mathbb{T}_2))$, respectively. By Lemma \ref{lem:torsor}, the first set is even a torsor under this action. We prove that the second is also a torsor, so that $\mathrm{t}_1$ is an equivariant map between torsors, hence a bijection.

We recall that the elements of $\hom_{\trivtrunc{\mathbb{CS}_M}}(\mathrm{t}_1(\mathbb{T}_1),\mathrm{t}_1(\mathbb{T}_2))$ are equivalence classes $[\mathcal{B}]$ of  isomorphisms $\mathcal{B}: \mathcal{S}_1 \to \mathcal{S}_2$ such that there exists a  transformation $\beta$ as in \erf{eq:beta}.
The action of $[K]\in\hc 0 \ubun M$ sends $[\mathcal{B}]$ to $[\pi^{*}K \otimes \mathcal{B}]$. This action is free  because $\pi^{*}: \h^2(M,\Z) \to \h^2(FM,\Z)$ is an isomorphism by Lemma \ref{lem:serre}.  If $\mathcal{B}$ and $\mathcal{B}'$ are both isomorphisms from $\mathcal{S}_1$ to $\mathcal{S}_2$, then by \erf{eq:actionbundles} there exists a principal $\ueins$-bundle $P$ on $FM$ such that $\mathcal{B}'\cong \mathcal{B} \otimes P$. But again, since $\pi^{*}$ is an isomorphism, $P \cong \pi^{*}K$ for some $K$ in $\ubun M$. Hence the action is transitive. 
\end{proof}

\setsecnumdepth{2}

\section{String connections and decategorification}
\label{sec:geomstringstructures}

This section is the  analogue of  Section \ref{sec:stringstructures} in the setting with connections. We first recall the definition of string connections and geometric string structures on the basis of \cite{waldorf8}.
Geometric string structures form a bicategory, of which we discuss various decategorified versions.
At the end of a sequence of decategorification we naturally find the notion of a differential string class.

\subsection{String connections as connections on trivializations}

\label{sec:geometricstringstructures}

The Levi-Cevita connection on $M$ lifts to a spin connection $A$ on $FM$; in turn it defines a connection on the Chern-Simons 2-gerbe $\mathbb{CS}_M$ \cite[Theorem 1.2.1]{waldorf8}. In the following we recall this construction.
We suppose first that $\mathbb{G}$ is a bundle 2-gerbe over a smooth manifold $M$ as in Definition \ref{def:bundle2gerbe}. 
\begin{comment}
That is, it has a surjective submersion $\pi:Y \to M$, over $Y^{[2]}$ a bundle gerbe $\mathcal{P}$ and over $Y^{[3]}$ a bundle 2-gerbe product $\mathcal{M}$ with associator $\mu$.
\end{comment}

\begin{definition}
\label{def:connection2gerbe}
A \emph{connection} on  $\mathbb{G}$  is a 3-form $C \in \Omega^3(Y)$ and  a connection on the bundle gerbe $\mathcal{P}$ such that
\begin{equation}
\label{eq:curving2gerbe}
\pr_2^{*}C - \pr_1^{*}C = \mathrm{curv}(\mathcal{P})\text{,}
\end{equation}
and the bundle 2-gerbe product $\mathcal{M}$ as well as the associator $\mu$ are connection-preserving.
\end{definition}

The 3-form $C$ is called the \emph{curving} of the connection.
In case of the Chern-Simons 2-gerbe, the announced connection is constructed using the connection $A$ on $FM$ and the connection on the basic bundle gerbe $\gbas$ described in Section \ref{sec:centralextension}:
\begin{itemize}

\item 
The curving is the Chern-Simons 3-form $CS(A) \in \Omega^3(FM)$ associated to $A$,
\begin{equation}
\label{eq:CSA}
CS(A)=\left \langle  A\wedge \mathrm{d}A \right \rangle + \frac{1}{3} \left \langle  A\wedge [A\wedge A] \right \rangle\text{,}
\end{equation} 
where $\left \langle -,-  \right \rangle$ is the same symmetric invariant bilinear form on the Lie algebra $\mathfrak{g}$ of $\spin n$ that was used to fix the curvature of the basic gerbe $\gbas$ in Section \ref{sec:centralextension}.

\item
The connection on the bundle gerbe $\mathcal{P}=\delta^{*}\gbas$ over $FM^{[2]}$ is given by the connection on $\delta^{*}\mathcal{G}_{bas}$ shifted by the 2-form
\begin{equation}
\label{eq:CSomega}
\omega :=  \left \langle   \delta^{*} \theta \wedge \pr_1^{*}A  \right \rangle \in
\Omega^2(FM^{[2]})\text{,}
\end{equation}
i.e. we have $\mathcal{P} = \delta^{*}\mathcal{G}_{bas} \otimes \mathcal{I}_{\omega}$ as bundle gerbes with connection. The well-known identity
\begin{equation}
\label{eq:csarule}
CS(\pr_2^{*}A) = CS(\pr_1^{*}A) + \delta^{*}H+\mathrm{d}\omega
\end{equation} 
for the Chern-Simons 3-form implies  the condition  \erf{eq:curving2gerbe} for the curving.

\item
We recall that the isomorphism
\begin{equation*}
\mathcal{M}: \pr_1^{*}\mathcal{G} \otimes \pr_2^{*}\mathcal{G} \to m^{*}\mathcal{G} \otimes \mathcal{I}_{\rho}
\end{equation*}
from the multiplicative structure on $\gbas$ is connection-preserving. Under pullback with $\delta_2: FM^{[3]} \to \spin n^2$, the 2-form $\rho$ satisfies
\begin{equation}
\label{eq:identityrhoomega}
\pr_{13}^{*}\omega=\delta_2^{*}\rho + \pr_{12}^{*} \omega  + \pr_{23}^{*}\omega\text{.}
\end{equation}
\begin{comment}
This identity is derived as follows. First we recall that
\begin{equation*}
\pr_{13}^{*}\delta^{*}\theta = \delta_2^{*}m^{*}\theta = \delta_2^{*}(\mathrm{Ad}_{\pr_2}^{-1}(\pr_1^{*}\theta)+\pr_2^{*}\theta)=\mathrm{Ad}^{-1}_{\pr_2 \circ \delta_2}(\delta_2^{*}\pr_1^{*}\theta)+\delta_2^{*}\pr_2^{*}\theta\text{.}
\end{equation*}
Then we get
\begin{eqnarray*}
&&\hspace{-1cm}\pr_{12}^{*} \omega - \pr_{13}^{*}\omega + \pr_{23}^{*}\omega 
\\&=& 
 \left \langle  \pr_{12}^{*} \delta^{*} \theta \wedge \pr_1^{*}A  \right \rangle -  \left \langle  \pr_{13}^{*} \delta^{*} \theta \wedge \pr_1^{*}A  \right \rangle+ \left \langle  \pr_{23}^{*} \delta^{*} \theta \wedge \pr_2^{*}A  \right \rangle
\\ &=&  
\left \langle  \pr_{12}^{*} \delta^{*} \theta -  \pr_{13}^{*} \delta^{*} \theta \wedge \pr_1^{*}A  \right \rangle+ \left \langle  \pr_{23}^{*} \delta^{*} \theta \wedge \mathrm{Ad}_{\delta \circ \pr_{12}}(\pr_1^{*}A)-\pr_{12}^{*}\delta^{*}\bar\theta  \right \rangle
\\ &=&  
\left \langle  -\mathrm{Ad}_{\pr_2\circ \delta_2}^{-1}(\delta_2^{*}\pr_1^{*}\theta) \wedge \pr_1^{*}A  \right \rangle+ \left \langle  \pr_{23}^{*} \delta^{*} \theta \wedge \mathrm{Ad}_{\pr_2\circ \delta_2}(\pr_1^{*}A)-\pr_{12}^{*}\delta^{*}\bar\theta  \right \rangle
 \\&=&  - \left \langle  \pr_{23}^{*}\delta^{*}\theta \wedge \pr_{12}^{*}\delta^{*}\bar\theta  \right \rangle
 \\&    =&  -\delta_2^{*} \left \langle  \pr_{1}^{*}\theta \wedge \pr_{2}^{*}\bar\theta  \right \rangle
\end{eqnarray*}
\end{comment}
This permits to define a connection-preserving bundle 2-gerbe product $\mathcal{M}'$ by
\begin{equation*}
\alxydim{@C=1.3cm@R=1.1cm}{\pr_{23}^{*}\mathcal{P} \otimes \pr_{12}^{*}\mathcal{P}= \delta_2^{*}(\pr_{1}^{*}\mathcal{G} \otimes \pr_2^{*}\mathcal{G})
\otimes \mathcal{I}_{\pr_{13}^{*}\omega-\delta_2^{*}\rho} \ar[r]^-{\delta_2^{*}\mathcal{M} \otimes
\id}  & \delta_2^{*}m^{*}\mathcal{G} \otimes \mathcal{I}_{\pr_{13}^{*}\omega}
= \pr_{13}^{*}\mathcal{P}\text{.}}
\end{equation*}

\item
The connection-preserving transformation $\alpha$ from the multiplicative structure on $\gbas$ gives a connection-preserving associator. 

\end{itemize}

If a bundle 2-gerbe $\mathbb{G}$ is equipped with a connection, and $\mathbb{T}=(\mathcal{S},\mathcal{A},\sigma)$ is a trivialization of $\mathbb{G}$, then a \emph{compatible connection} on $\mathbb{T}$ is a connection on the bundle gerbe $\mathcal{S}$ such that the isomorphism $\mathcal{A}$ and the transformation $\sigma$ are connection-preserving.

\begin{definition}[{{\cite[Definition 1.2.2]{waldorf8}}}]
Let $\mathbb{T}$ be a string structure on $M$. A \emph{string connection} on $\mathbb{T}$ is a compatible connection on $\mathbb{T}$. A \emph{geometric string structure} on $M$ is a pair of a string structure on $M$ and a string connection.
\end{definition}

Geometric string structures form a bicategory $\ststcon \relax M:=\trivcon{\mathbb{CS}_M}$, with  1-morphisms and 2-morphisms defined exactly as in the setting without connections, just that all occurring isomorphisms and transformations are  connection-preserving \cite[Remark 6.1.1]{waldorf8}. 
We recall the following results about string connections  for later reference. 

\begin{theorem}[{{\cite[Theorems 1.3.4 and 1.3.3]{waldorf8}}}]
\label{th:stringconnections}
\begin{enumerate}[(i)]

\item 
\label{th:stringconnections:exist}
Every string structure admits a string connection.

\item
\label{th:stringconnections:3form}
For every geometric string structure $\mathbb{T}=(\mathcal{S},\mathcal{A},\sigma)$, there exists a unique 3-form $K\in\Omega^3(M)$ such that $\pi^{*}K=CS(A) + \mathrm{curv}(\mathcal{S})$. 

\end{enumerate}
\end{theorem}

\subsection{Decategorified  string connections}

\label{sec:decatstringcon}

Just like in the setting without connections, we may consider various truncations of the bicategory $\trivcon {\mathbb{G}}$ for a general bundle 2-gerbe $\mathbb{G}$ with connection. So we have  categories  $\hc 1 \trivcon{\mathbb{G}}$ and $\trivcontrunc{\mathbb{G}}$ and a functor
\begin{equation*}
\mathrm{t}_1^{\nabla}:\hc 1 \trivcon{\mathbb{G}} \to \trivcontrunc{\mathbb{G}}\text{.}
\end{equation*}
Further we have three sets $\hc 0 \trivcon{\mathbb{G}}$, $\hc 0 \trivcontrunc{\mathbb{G}}$ and $\trivconttrunc{\mathbb{G}}$, and maps
\begin{equation*}
\alxydim{}{\hc 0 \trivcon{\mathbb{G}}  \ar[r]^-{\hc 0 \mathrm{t}_1^{\nabla}} & \hc 0 (\trivcontrunc {\mathbb{G}}) \ar[r]^-{\mathrm{t}_2^{\nabla}} & \trivconttrunc{\mathbb{G}}\text{.}}
\end{equation*}

The passage from the setting with connections to the one without connections is attended by a 2-functor $F_2: \trivcon{\mathbb{G}} \to \triv{\mathbb{G}}$, a functor $F_1: \trivcontrunc{\mathbb{G}} \to \trivtrunc{\mathbb{G}}$, and a map $F_0\maps \trivconttrunc{\mathbb{G}} \to \trivttrunc{\mathbb{G}}$. It is fully consistent with the various functors and maps introduced above, in the sense that the diagrams
\begin{equation}
\label{eq:passage}
\hspace{-1em}\alxydim{@=\xypicst@C=2em}{\hc 1 \trivcon{\mathbb{G}} \ar[r] \ar[d]_{\hc 1 F_{2}} & \trivcontrunc{\mathbb{G}} \ar[d]^{F_1} \\ \hc 1 \triv{\mathbb{G}} \ar[r] & \trivtrunc{\mathbb{G}}}
\quad\quad
\alxydim{@R=\xypicst@C=2em}{\hc 0 \trivcon{\mathbb{G}} \ar[d]_{\hc  0 F_{2}} \ar[r] & \hc 0 \trivcontrunc{\mathbb{G}} \ar[d]^{\hc 0 F_1} \ar[r] & \trivconttrunc{\mathbb{G}} \ar[d]^{F_0} \\ \hc 0 \triv{\mathbb{G}} \ar[r] & \hc 0 \trivtrunc{\mathbb{G}} \ar[r] & \trivttrunc{\mathbb{G}} }
\end{equation}
of functors and maps, respectively, are commutative.

We have again various actions. The bicategory $\trivcon {\mathbb{G}}$ is a torsor for the monoidal bicategory $\ugrbcon M$ of bundle gerbes with connection over $M$ \cite[Lemma 2.2.5]{waldorf8}. As before, a bundle gerbe $\mathcal{K}$ with connection over $M$ acts on a trivialization $\mathbb{T}=(\mathcal{S},\mathcal{A},\sigma)$ with compatible connection by sending it to $(\pi^{*}\mathcal{K} \otimes \mathcal{S},\id \otimes \mathcal{A}, \id \otimes \sigma)$. Correspondingly, the category $\hc 1 \trivcon{\mathbb{G}}$ is a torsor category over the monoidal category $\hc 1 \ugrbcon M$, and the set $\hc 0 \trivcon {\mathbb{G}}$ is a torsor over the group $\hc 0 \ugrbcon M$.

In the same natural way, the monoidal category $\hc 1 \ugrbcon M$ acts on the category $\trivcontrunc{\mathbb{G}}$ such that the functor $\mathrm{t}_1^{\nabla}$ is equivariant, and the group $\hc 0 \ugrbcon M$ acts on the set $\trivconttrunc{\mathbb{G}}$ such that the map $\mathrm{t}_2^{\nabla}$ is equivariant.

Now we specialize to the case of the Chern-Simons 2-gerbe $\mathbb{G}=\mathbb{CS}_M$, and discuss the bicategory $\ststcon\relax M =\trivcon{\mathbb{CS}_M}$ of geometric string structures, the two categories $\hc 1 \ststcon\relax M$ and $\ststcon 1M := \trivcontrunc{\mathbb{CS}_M}$ of geometric string structures, and the three sets $\hc 0 \ststcon\relax M$, $\hc 0 \ststcon 1M$, and $\ststcon 0 M:= \trivconttrunc{\mathbb{CS}_M}$ of geometric string structures. We first note the following result about the passage from the setting with connections to the setting without connections.   

\begin{proposition}
\label{prop:passagesur}
\begin{enumerate}[(i)]

\item
\label{prop:passagesur:2}
The 2-functor $F_2: \ststcon\relax M \to \stst\relax M$ is essentially surjective.

\item
\label{prop:passagesur:1}
The functor $F_{1}\maps\ststcon 1M \to \stst 1M$ is essentially surjective.

\item
\label{prop:passagesur:0}
The map $F_0\maps \ststcon 0M \to \stst 0M$ is surjective.

\end{enumerate}
\end{proposition}

\begin{proof}
\erf{prop:passagesur:2} is Theorem \ref{th:stringconnections} \erf{th:stringconnections:exist}. It follows that the functor $\hc 1 F_2$ is essentially surjective, and that the map $\hc 0 F_2$ is surjective. Then, \erf{prop:passagesur:1} and \erf{prop:passagesur:0} follow from the commutativity of the diagrams \erf{eq:passage} and Theorem \ref{th:truncation}.
\end{proof}

Next we present the main theorem of this section, which is the analogue of Theorem \ref{th:truncation}  in the setting with connections.

\begin{theorem}
\label{th:truncationcon}
The functor 
\begin{equation*}
\mathrm{t}_1^{\nabla}: \hc 1 \ststcon\relax M \to \ststcon 1M
\end{equation*}
is an equivalence of categories; in particular, it is an equivariant functor between $\hc 1 \ugrbcon M$-torsors. The map 
\begin{equation*}
\mathrm{t}_2^{\nabla}: \hc 0 \ststcon 1M \to \ststcon 0M
\end{equation*}
is a bijection; in particular, it is an equivariant map between $\hat\h^3(M)$-torsors.
\end{theorem}

Here, $\hat \h^n(M)$  stands for the \emph{differential cohomology} of $M$. We recall from \cite{brylinski1} that the degree $n$ differential cohomology of a smooth manifold $X$ is a group $\hat\h^n(X)$ that fits into the   exact sequences
\begin{equation*}
\alxydim{}{0 \ar[r] & \Omega^{n-1}_{cl,\Z}(X) \ar[r] & \Omega^{n-1}(X) \ar[r]^-{a} & \hat\h^n(X) \ar[r]^-{c} & \h^n(X,\Z) \ar[r] & 0 }
\end{equation*}
and
\begin{equation*}
\alxydim{}{0 \ar[r] & \h^{n-1}(X,\ueins) \ar[r] & \hat\h^n(X) \ar[r]^-{R} & \Omega^{n}_{cl,\Z}(X) \ar[r] & 0\text{,}}
\end{equation*}
in which $\Omega_{cl,\Z}^{n}(X)$ denotes the closed $n$-forms on $X$ with integral periods. Bundle gerbes with connection are classified by degree three differential cohomology in terms of a differential Dixmier-Douady class $\widehat{\mathrm{DD}}: \hc 0 \ugrbcon X \to \hat\h^3(X)$; the map $c:\hat\h^3(X) \to \h^3(X,\Z)$   corresponds to projecting to the underlying (non-differential) Dixmier-Douady class, the map $R:\hat\h^3(X) \to \Omega^3_{cl,\Z}(X)$ corresponds to taking the curvature, and the map $a\maps  \Omega^2(X) \to \hat\h^3(X)$ corresponds to taking the trivial bundle gerbe $\mathcal{I}_{\rho}$ associated to a 2-form $\rho$. 

In the remainder of this section we prove Theorem \ref{th:truncationcon}, see Lemmata \ref{lem:t1nabla} and  \ref{co:t2nabla} below. 
First we generalize one aspect of the Serre spectral sequence calculation of Lemma \ref{lem:serreall}
from ordinary cohomology to differential cohomology.

\begin{lemma}
\label{lem:pulldiff}
The pullback $\pi^{*}:\hat\h^3(M) \to \hat\h^3(FM)$ is injective.         
\end{lemma}

\begin{proof}
Let $\hat\eta\in\hat\h^3(M)$. 
We show that $\hat\eta\neq0$ implies $\pi^{*}\hat\eta\neq0$. Indeed, if the underlying class $\eta := c(\hat\eta)\in\h^3(M,\Z)$ is non-zero, than $\pi^{*}\eta\neq 0$  because of Lemma \ref{lem:serreall}, and so is $\pi^{*}\hat\eta\neq 0$. If $\eta=0$, then $\hat\eta=a(\mu)$ for a 2-form $\mu\in\Omega^2(M)$. Since $\pi$ is a surjective submersion, $\pi^{*}$ is injective. Thus, if $\mu$ is not closed, then $\pi^{*}\mu$ is also not closed and $\pi^{*}\hat\eta=a(\pi^{*}\mu)$ must be non-trivial. It remains to discuss the case that $\mu$ is closed but its class is not integral. Then, $\pi^{*}\mu$ is also closed. By Lemma \ref{lem:serreall}, $\pi^{*}:\h^2(M,\Z) \to \h^2(FM,\Z)$ is an isomorphism, so since $\mu$ is not integral,  $\pi^{*}\mu$ is not integral. Hence $\pi^{*}\hat\eta$ is non-trivial. \end{proof}

\begin{corollary}
The actions of $\hat \h^3(M)\cong \hc 0 \ugrbcon M$ on $\hc 0 \ststcon 1M$ and $\ststcon 0M$ are free. 
\end{corollary}

\begin{proof}
These actions are defined by pullback (injective by Lemma \ref{lem:pulldiff}) and then addition in the group $\hat\h^3(FM)$.  \end{proof}

Now we are in position to prove the first part of Theorem \ref{th:truncationcon}.

\begin{lemma}
\label{lem:t1esssurjective}
\label{lem:t1nabla}
The functor $\mathrm{t}_1^{\nabla}:\hc 1 \ststcon\relax M \to \ststcon 1M$ is an equivalence.
\end{lemma}

\begin{proof}
Part I of the proof is to show that $\mathrm{t}_1^{\nabla}$ is essentially surjective. 
We consider an object $(\mathcal{S},[\mathcal{A}])$ in $\ststcon 1M=\trivcontrunc{\mathbb{CS}_M}$, i.e. $\mathcal{S}$ is a bundle gerbe with connection over $FM$, and $\mathcal{A}: \pr_2^{*}\mathcal{S} \otimes \mathcal{P} \to \pr_1^{*}\mathcal{S}$ is a connection-preserving 1-isomorphism such that there exists a connection-preserving transformation $\sigma$ as in \erf{eq:diagfillsigma}.
A priori, $\sigma$ does not satisfy the compatibility condition with the associator $\mu$ of $\mathbb{CS}_M$, see Figure \ref{compass}. We show that it yet does, using the assumption that it is connection-preserving. The error in the commutativity of the diagram of Figure \ref{compass} is a smooth map $\varepsilon: FM^{[4]} \to \ueins$.   Since both $\mu$ and $\sigma$ are connection-preserving, $\varepsilon$ is locally constant. Since $M$ is connected, $FM$ and all fibre products $FM^{[k]}$ are connected, in particular  $FM^{[4]}$. Thus, $\varepsilon$ is constant. From the pentagon axiom for $\mu$ over $FM^{[5]}$ 
\begin{comment}
, see Figure \ref{fig:pentagon}, 
\end{comment}
it follows that $\Delta\varepsilon=1$. Since $\Delta\varepsilon$ has five  terms, which are all equal as $\varepsilon$ is constant, this implies $\varepsilon=1$. Thus, $\sigma$ automatically satisfies the compatibility condition, $(\mathcal{S},\mathcal{A},\sigma)$ is a geometric string structure, and a preimage of $(\mathcal{S},[\mathcal{A}])$ under $\mathrm{t}_1^{\nabla}$. Hence, $\mathrm{t}_1^{\nabla}$ is essentially surjective.

Part II of the proof is to show that $\mathrm{t}_1^{\nabla}$ is full and faithful. First of all, we note that the map $\hc 0 \mathrm{t}_1^{\nabla}$ is equivariant  under free actions and defined on a torsor, and thus injective. We have just proved that it also is surjective; hence, $\hc 0 \mathrm{t}_1^{\nabla}$ is a bijection. Now we proceed similar to the proof of Lemma \ref{lem:t1equiv}. That $\hc 0 \mathrm{t}_1^{\nabla}$ is a bijection implies that for two objects $\mathbb{T}_1\eq(\mathcal{S}_1,\mathcal{A}_1,\sigma_1)$ and $\mathbb{T}_2\eq(\mathcal{S}_2,\mathcal{A}_2,\sigma_2)$ of $\hc 1 \ststcon\relax M$ the Hom-sets $\hom_{\hc 1 \ststcon\relax M}(\mathbb{T}_1,\mathbb{T}_2)$ and $\hom_{\ststcon 1M}(\mathrm{t}_1^{\nabla}(\mathbb{T}_1),\mathrm{t}_1^{\nabla}(\mathbb{T}_2))$ are either both empty or both non-empty. It remains to show that $\mathrm{t}_1^{\nabla}$ induces in the non-empty case a bijection between these sets.

The action of the monoidal bicategory $\ubun X$ on the category of homomorphisms between two fixed bundle gerbes over $X$, see \ref{eq:actionbundles}, has a counterpart in the setting with connections, namely an action
\begin{equation}
\label{eq:actionbundlescon}
\ubunconflat X \times \homcon(\mathcal{G},\mathcal{H}) \to \homcon(\mathcal{G},\mathcal{H})
\end{equation}
of the monoidal category of principal $\ueins$-bundles with \emph{flat} connections on the category of connection preserving isomorphisms between $\mathcal{G}$ and $\mathcal{H}$, and connection-preserving transformations. This action exhibits again $\homcon(\mathcal{G},\mathcal{H})$ as a torsor over $\ubunconflat X$.

The functor $\mathrm{t}_1^{\nabla}$ is equivariant with respect to the $\hc 1 \ugrbcon M$-actions. Over the objects $(\mathcal{I}_0,\mathcal{I}_0)$ and $(\mathbb{T}_1,\mathbb{T}_2)$ these induce actions of the group 
\begin{equation*}
\hom_{\hc 1 \ugrbcon M}(\mathcal{I}_0,\mathcal{I}_0) = \hc 0 \hom_{\ugrbcon M}(\mathcal{I}_0,\mathcal{I}_0)\stackerf{eq:actionbundlescon}{=}\hc 0 \ubunconflat M
\end{equation*}
on the sets $\hom_{\hc 1 \ststcon\relax M}(\mathbb{T}_1,\mathbb{T}_2)$ and $\hom_{\ststcon 1M}(\mathrm{t}_1(\mathbb{T}_1),\mathrm{t}_1(\mathbb{T}_2))$, respectively.  The first set is  a torsor under this action. We prove that the second is also a torsor, so that $\mathrm{t}_1^{\nabla}$ is an equivariant map between torsors, hence a bijection.

We recall that the elements of $\hom_{\ststcon 1M}(\mathrm{t}_1(\mathbb{T}_1),\mathrm{t}_1(\mathbb{T}_2))$ are equivalence classes $[\mathcal{B}]$ of  connection-preserving isomorphisms $\mathcal{B}: \mathcal{S}_1 \to \mathcal{S}_2$ such that there exists a  connection-preserving transformation $\beta$ as in \erf{eq:beta}.
\begin{comment}
\begin{equation}
\alxydim{@=\xypicst}{\pr_2^{*}\mathcal{S} \otimes \mathcal{P} \ar[r]^-{\mathcal{A}} \ar[d]_{\pr_2^{*}\mathcal{B}\otimes \id} & \pi_1^{*}\mathcal{S} \ar@{=>}[dl]|*+{\beta}  \ar[d]^{\pr_1^{*}\mathcal{B}} \\ \pr_2^{*}\mathcal{S}'\otimes \mathcal{P} \ar[r]_-{\mathcal{A}'} & \pi_1^{*}\mathcal{S}'\text{.}}
\end{equation} 
\end{comment}
The action of $[K]\in\hc 0 \ubunconflat M \cong \h^1(M,\ueins)$ sends $[\mathcal{B}]$ to $[\pi^{*}K \otimes \mathcal{B}]$. This action is free  because $\pi^{*}: \h^1(M,\ueins) \to \h^1(FM,\ueins)$ is an isomorphism by Lemma \ref{lem:serreall}.  If $\mathcal{B}$ and $\mathcal{B}'$ are both connection-preserving isomorphisms from $\mathcal{S}_1$ to $\mathcal{S}_2$, then by \erf{eq:actionbundlescon} there exists a flat principal $\ueins$-bundle $P$ on $FM$ such that $\mathcal{B}'\cong \mathcal{B} \otimes P$. But again, since $\pi^{*}$ is an isomorphism, $P \cong \pi^{*}K$ for some $K$ in $\ubunconflat M$. Hence the action is transitive. 
\end{proof}

The second part of Theorem \ref{th:truncationcon} is proved by the following lemma.

\begin{lemma}
\label{lem:t2nablasurjective}
\label{co:t2nabla}
The map $\mathrm{t}_2^{\nabla}: \hc 0 \ststcon 1M \to \ststcon 0M$ is a bijection. 
\end{lemma}

\begin{proof}
$\mathrm{t}_2^{\nabla}$ is  equivariant with respect to  free actions and is defined on a torsor. Hence it is injective. Now we prove that it is surjective. 
Consider an element in $\ststcon 0M$, represented by a bundle gerbe $\mathcal{S}$ with connection over $FM$ that admits  a connection-preserving isomorphism $\mathcal{A}:\pr_2^{*}\mathcal{S} \otimes \mathcal{P} \to \pr_1^{*}\mathcal{S}$. The existence of a connection-preserving transformation $\sigma$ as in \erf{eq:diagfillsigma}
\begin{comment}
\begin{equation*}
\alxydim{@=\xypicst}{\pr_{3}^{*}\mathcal{S} \otimes \pr_{23}^{*}\mathcal{P} \otimes \pr_{12}^{*}\mathcal{P} \ar[r]^-{\pr_{23}^{*}\mathcal{A} \otimes \id} \ar[d]_{\id \otimes \mathcal{M}} & \pr_{2}^{*}\mathcal{S} \otimes \pr_{12}^{*}\mathcal{P} \ar@{=>}[dl]|*+{\sigma} \ar[d]^{\pr_{12}^{*}\mathcal{A}} \\ \pr_{3}^{*}\mathcal{S} \otimes \pr_{13}^{*}\mathcal{P} \ar[r]_-{\pr_{13}^{*}\mathcal{A}} & \pr_1^{*}\mathcal{S}}
\end{equation*}
\end{comment}
is obstructed by a  flat principal $\ueins$-bundle $A$ over $Y^{[3]}$,  i.e. there exists a connection-preserving transformation
\begin{equation*}
\sigma:(\pr_{12}^{*}\mathcal{A} \circ (\pr_{23}^{*}\mathcal{A} \otimes\id)) \otimes A \Rightarrow \pr_{13}^{*}\mathcal{A} \circ (\id \otimes \mathcal{M})\text{.}
\end{equation*}
Since flat principal $\ueins$-bundles are classified by $\h^1(X,\ueins)$, we infer from Lemma \ref{lem:serre} that this bundle is the pullback of a flat bundle $A'$ over $M$ along $\pi_3\maps Y^{[3]} \to M$. Now we consider $\mathcal{A}' := \mathcal{A} \otimes \pi_2^{*}A'$, which is another connection-preserving isomorphism  $\mathcal{A}'\maps \pr_2^{*}\mathcal{S} \otimes \mathcal{P} \to \pr_1^{*}\mathcal{S}$. Then, $\sigma$ induces the required transformation for $\mathcal{A}'$.
This means that $(\mathcal{S},[\mathcal{A}'])$ is an object in $\ststcon 1M=\trivcontrunc {\mathbb{CS}_M}$ with $\mathrm{t}_2^{\nabla}(\mathcal{S},[\mathcal{A}])=[\mathcal{S}]$. \end{proof}

\subsection{Differential string classes}

\label{sec:differentialstringclasses}

With Theorem \ref{th:truncationcon} proved in the previous subsection we are well prepared to introduce an analog of string classes in the setting with connections -- we call it \emph{differential string classes}. Like  string classes, differential string classes have the advantage to be based solely on differential cohomology theory, and  no bundle gerbe theory is needed. 
\begin{comment}
Differential string classes have the disadvantage that they only form a \emph{set}, and do not allow a discussion of a \emph{category} of geometric string structures.
\end{comment}

We let $\hat\gamma \df \widehat{\mathrm{DD}}(\mathcal{G}_{bas}) \in\hat\h^3(\spin n)$ denote the differential cohomology class of the basic gerbe, with underlying class $\gamma \in \h^3(\spin n,\Z)$. As explained in Section \ref{sec:centralextension} this class is uniquely determined by just the 3-form $H$. We let $\hat\omega := a(\omega) \in\hat\h^3(FM^{[2]})$ denote the differential cohomology class associated to the 2-form $\omega$ of \erf{eq:CSomega}. 

\begin{definition}
\label{def:diffstringclass}
Let $M$ be a spin manifold with spin-oriented frame bundle $FM$. A \emph{differential string class} is a class $\hat\xi\in\hat\h^3(FM)$ such that the condition
\begin{equation}
\label{eq:conddiffsc}
\pr_2^{*}\hat\xi + \delta^{*}\hat\gamma + \hat\omega  = \pr_1^{*}\hat\xi
\end{equation}
over $FM^{[2]}$ is satisfied, where $\pr_1,\pr_2:FM^{[2]} \to FM$ are the two projections, and  $\delta\maps FM^{[2]} \to \spin n$ is the difference map (i.e. $p' \cdot \delta(p,p') = p$).
\end{definition}

We denote by $\sccon(M)$ the set of differential string classes. Condition \erf{eq:conddiffsc} implies
\begin{equation}
\label{eq:conddiffscwrong}
\hat\gamma = i_p^{*}\hat\xi \in \hat\h^3(\spin n)
\end{equation}
for all $p\in FM$ and $i_p: \spin n\to FM:g\mapsto pg$ the inclusion of the fibre of $p$. Indeed, for $j_p: \spin n\to FM^{[2]}:g \mapsto (pg,p)$ we have
\begin{equation*}
j_p^{*}\omega=j_p^{*}\left \langle \delta^{*}\theta \wedge \pr_1^{*}A  \right \rangle = \left \langle  \theta \wedge \iota_p^{*}A  \right \rangle = \left \langle \theta \wedge \theta  \right \rangle=0\text{.}
\end{equation*}
Further we have $j_p^{*}\delta^{*}\hat\gamma=\hat\gamma$ and $j_p^{*}\pr_2^{*}\hat\xi=0$ and $j_p^{*}\pr_1^{*}\hat\xi=\iota_p^{*}\hat\xi$ and so the pullback of \erf{eq:conddiffsc} is \erf{eq:conddiffscwrong}. 

Under the projection $c: \hat\h^3(X) \to \h^3(X,\Z)$ from differential to ordinary cohomology, condition \erf{eq:conddiffscwrong}  becomes the condition for string classes. In other words, $c$ induces a well-defined map
\begin{equation*}
c: \sccon(M) \to \sc(M)
\end{equation*}
from differential string classes to ordinary string classes.

\begin{remark}
Above considerations raise the question of whether  we could replace the defining condition \erf{eq:conddiffsc} by  condition \erf{eq:conddiffscwrong}. However, this is not the case. In order to see this, let $\hat\xi$ be a differential string class. Let $\mu\in\Omega^2(M)$ and $f\in C^{\infty}(FM,\R)$ satisfy the following assumptions:
\begin{itemize}

\item 
$\mu$ is closed and non-zero at at least one point $x\in M$.

\item
$\mathrm{d}f$ is non-zero at a point $p\in FM$ in the fibre over $x$, and $\mathrm{d}f=0$ at another point $p'\in FM$ in the same fibre. 

\end{itemize}
Such $\mu$, $f$ clearly exist. We consider the class $\hat\epsilon=a(\varepsilon)\in \hat\h^3(FM)$ associated to the 2-form  $\varepsilon \df f\cdot\nobr \pi^{*}\mu\in\Omega^2(FM)$. We have $\mathrm{d}\varepsilon = \mathrm{d}f \wedge \pi^{*}\mu + f\cdot \pi^{*}\mathrm{d}\mu=\mathrm{d}f\wedge \pi^{*}\mu$. This is non-zero at the point $p$, in particular, $\varepsilon$ is not closed and $\hat\epsilon\neq0\in\hat\h^3(FM)$. We show that $\hat\xi' := \hat\xi + \hat\varepsilon$ satisfies \erf{eq:conddiffscwrong} but is not a  differential string class. Firstly, we have $i_p^{*}\varepsilon = i_p^{*}f \cdot i_p^{*}\pi^{*}\mu = 0$, since $\pi \circ i_p$ is constant. Thus, $i_p^{*}\hat\xi'=i_p^{*}\hat\xi=\hat\gamma$. Secondly, we have $\mathrm{d}(\pr_1^{*}\varepsilon-\pr_2^{*}\varepsilon)=\mathrm{d}(\pr_1^{*}f-\pr_2^{*}f)\cdot \pr^{*}\mu$, which is  non-zero at the point $(p,p') \in FM^{[2]}$. In particular, $\pr_1^{*}\hat\varepsilon\neq \pr_2^{*}\hat\varepsilon\in \hat\h^3(FM^{[2]})$. Thus, condition \erf{eq:conddiffsc} is not satisfied.
\end{remark}

We have  an  action of $\hat\h^3(M)$ on the set $\sccon(M)$ of differential string classes, defined by $(\hat\eta,\hat\xi)\mapsto \pi^{*}\hat\eta + \hat\xi$. Under the projection to ordinary string classes, it covers the action of $\h^3(M,\Z)$ on $\sc (M)$. 
The differential Dixmier-Douady class gives a $\hat\h^3(M)$-equivariant map
\begin{equation*}
\ststcon 0M \to  \sccon(M): [\mathcal{S}] \mapsto \widehat{\mathrm{DD}}(\mathcal{S})\text{.}
\end{equation*}
This map is a bijection, because $\widehat{\mathrm{DD}}$ is a bijection and the conditions on both sides are the same. Thus,  we may identify the set of differential string classes with the set $\ststcon 0M$  introduced in the previous section. With this identification,   Theorem \ref{th:truncationcon}
applies to differential string classes, and we obtain the following result.

\begin{theorem}
\label{th:differentialstringclasses}
\begin{enumerate}[(i)]
\item 
The set $\sccon(M)$ of differential string classes is non-empty if and only if $M$ is a string manifold; in this case it is a torsor over $\hat\h^3(M)$.

\item
The map
\begin{equation*}
\hc 0 \ststcon\relax M \to \sccon (M): (\mathcal{S},\mathcal{A},\sigma) \mapsto \widehat{\mathrm{DD}}(S)
\end{equation*}
is a $\hat\h^3(M)$-equivariant bijection between isomorphism classes of geometric string structures and differential string classes.

\item
The projection from differential string classes to ordinary string classes, 
\begin{equation}
\label{eq:diffstringclassforget}
\sccon(M) \to \sc(M) \text{,}
\end{equation}
is surjective and its  fibres are torsors over $\Omega^2(M)/\Omega^2_{cl,\Z}(M)$.

\item
For every differential string class $\hat\xi$ there exists a unique 3-form $K\in\Omega^3(M)$ such that $\pi^{*}K\eq CS(A)+R(\hat\xi)$ as 3-forms over $FM$.

\end{enumerate}
\end{theorem}

\begin{proof}
By Theorem \ref{th:stringconnections} \erf{th:stringconnections:exist} and Corollary \ref{co:stringclasses} \erf{co:stringclasses:exist}, $M$ is a string manifold if and only if it admits geometric string structures. By Theorem \ref{th:truncationcon} we have $\hat\h^3(M)$-equivariant bijections $\hc 0 \ststcon\relax M\cong \ststcon 0M\cong \sccon(M)$; this shows (i) and (ii). The projection $\sccon(M) \to \sc(M)$ is a map from a $\hat \h^3 (M)$-torsor to a $\h^3(M,\Z)$-torsor, and equivariant along the projection $c\maps \h^3(M) \to \h^3(M,\Z)$. This projection is surjective and has kernel $\Omega^2(M)/\Omega^2_{cl,\Z}(M)$; this shows (iii). Assertion (iv) is Theorem \ref{th:stringconnections} \erf{th:stringconnections:3form}. 
\end{proof}

\setsecnumdepth{1}

\section{Transgression of string geometry}

\label{sec:transgression}

In \cite[Section 5.2]{Waldorfa} and \cite{Nikolausa} we have described  transgression for bundle 2-gerbes $\mathbb{G}$ with connections and with loopable surjective submersions, i.e. surjective submersions $\pi:Y \to M$ for which $L\pi\maps LP\to LM$ is again a surjective submersion. 
\begin{comment}
It only uses the transgression functor 
\begin{equation*}
\tr: \hc 1 \ugrbcon M \to \ufusbuncon {LM}
\end{equation*}
and the fibre integration $\tau_{\Omega}$ of differential forms, see \erf{eq:difftrans}, and it  sends a bundle 2-gerbe $\mathbb{G}$ to a bundle gerbe $\tr_{\mathcal{G}}$ with connection and internal fusion product over $LM$.
\end{comment}
Suppose   $\mathbb{G}$ is such a bundle 2-gerbe over  $M$, with  loopable surjective submersion $\pi\maps Y \to M$, a curving 3-form $C\in\Omega^3(Y)$, over $Y^{[2]}$ a bundle gerbe $\mathcal{P}$ with connection, and over $Y^{[3]}$ a connection-preserving bundle 2-gerbe product $\mathcal{M}$ with associator $\mu$.
Then, the bundle gerbe $\tr_{\mathcal{G}}$ over $LM$ with connection and internal fusion product is given as follows: the surjective submersion  is $L\pi:LY \to LM$, the curving is $-\tau_{\Omega}(C)\in\Omega^2(LY)$, the principal $\ueins$-bundle with connection and fusion product over $LY^{[2]}$ is $\tr_{\mathcal{P}}$, and the connection-preserving, fusion-preserving  bundle gerbe product over $LY^{[3]}$ is $\tr_{\mathcal{M}}$. 

\begin{comment}
We promptly apply this transgression procedure to the Chern-Simons 2-gerbe, and obtain the following result.
\end{comment}

\begin{theorem}
\label{th:spinlift}
The transgression of the Chern-Simons 2-gerbe $\mathbb{CS}_M$ is canonically isomorphic to the spin lifting gerbe $\mathcal{S}_{LM}$ as bundle gerbes with connections and internal fusion products, where $\mathcal{S}_{LM}$ is equipped with the connection $(\chi_{\corr},B_{\corr})$ constructed in Section \ref{sec:liftingtheoryspinconnections}.
\end{theorem}

\begin{proof}
The bare isomorphism has been constructed in \cite[Proposition 6.2.1]{Nikolausa}, and in \cite[Proposition 5.2.3]{Waldorfa} it is  proved that it is fusion-preserving. We only have to prove that it is connection-preserving, and for this purpose we have to recall the construction.

  We start by noticing that both bundle gerbes, $\tr_{\mathbb{CS}_M}$ and $\mathcal{S}_{LM}$, have the same surjective submersion $L\pi:LFM \to LM$. The curving of $\tr_{\mathbb{CS}_M}$ is $-\tau_{\Omega}(CS(A))$, and the curving of $\mathcal{S}_{LM}$ is $B_{\corr}$ from Proposition \ref{prop:DeltaB}. These two 2-forms on $LFM$ coincide  \cite[Eq. 24]{Coquereaux1998}.
 The bundle gerbe $\tr_{\mathbb{CS}_M}$ has  over $LFM^{[2]}$ the principal $\ueins$-bundle  $\tr_{\mathcal{P}}$, where $\mathcal{P} \df \delta^{*}\gbas \otimes \mathcal{I}_{\omega}$, equipped with a connection $\nu_{\mathcal{P}}$ and an internal fusion product $\lambda_{\mathcal{P}}$ induced from transgression.    The bundle gerbe $\mathcal{S}_{LM}$ has over $LFM^{[2]}$ the principal $\ueins$-bundle $\q=L\delta^{*}\lspinhat n = L\delta^{*}\tr_{\gbas}$, equipped with the fusion product $L\delta^{*}\lambda_{\gbas}$ and the connection $\chi_{\corr}$ defined in \erf{eq:nukorr}. Naturality of transgression and the canonical connection-preserving, fusion-preserving bundle isomorphism $\tr_{\mathcal{I}_{\omega}}\cong \trivlin_{\tau_{\Omega}(\omega)}$ provide a connection-preserving, fusion-preserving isomorphism
\begin{equation*}
\tr_{\mathcal{P}}\cong L\delta ^{*}\tr_{\mathcal{G}_{bas}} \otimes \tr_{\mathcal{I}_{\omega}}\cong L\delta^{*}\tr_{\gbas} \otimes \trivlin_{\tau_{\Omega}(\omega)}\text{.} \end{equation*}
In the first place, this is an isomorphism \begin{equation}
\label{eq:bundleid}
\tr_{\mathcal{P}} \cong \q
\end{equation}
between the principal $\ueins$-bundles of the two bundle gerbes.
It commutes with the bundle gerbe products (\cite[Proposition 6.2.1]{Nikolausa}) and so we have completed the construction of an isomorphism $\tr_{\mathbb{CS}_M} \cong \mathcal{S}_{LM}$. Moreover, the isomorphism \erf{eq:bundleid} is fusion-preserving for the fusion product $\lambda_{\mathcal{P}}$ on the left and $L\delta^{*}\lambda_{\gbas}$ on the right, as $\trivlin_{\tau_{\Omega}(\omega)}$ is equipped with the trivial fusion product. Finally, it is connection-preserving for the connection $\nu_{\mathcal{P}}$ on the left and $L\delta^{*}\nu + \tau_{\Omega}(\omega)$ on the right: we show below in Lemma \ref{lem:tromega} the equality
\begin{equation*}
\tau_{\Omega}(\omega)=\xi-\frac{1}{2}\Delta\zeta
\end{equation*}
of 2-forms on $LFM^{[2]}$, where $\xi$ and $\zeta$ are defined in \erf{eq:zeta} and \erf{eq:xi}, respectively. This shows that $L\delta^{*}\nu + \tau_{\Omega}(\omega)=\chi_{\corr}$. Thus,  \erf{eq:bundleid} is connection-preserving.
\end{proof}

\begin{lemma}
\label{lem:tromega}
The transgression of the 2-form $\omega\in\Omega^2(P^{[2]})$ of \erf{eq:CSomega} is
\begin{equation*}
\tau_{\Omega}(\omega) =\xi-\frac{1}{2}\Delta\zeta\text{,}
\end{equation*}
with $\xi$ and $\zeta$ the differential forms defined in \erf{eq:zeta} and \erf{eq:xi}, respectively.
\end{lemma}

\begin{proof}
\begin{comment}
The sign is consistent: taking derivatives reproduces
\begin{multline*}
\tau_{\Omega}(\mathrm{d}\omega)\stackerf{eq:csarule}{=}\tau_{\Omega}(\Delta CS(A)-\delta^{*}H)=-\Delta B-L\delta^{*}\tau_{\Omega}(H)\\=-\mathrm{curv}(\nu_{corr})+L\delta^{*}\mathrm{curv}(\nu)\stackerf{eq:nukorr}{=}-\mathrm{d}(\xi-\frac{1}{2}\Delta\zeta)=-\mathrm{d}\tau_{\Omega}(\omega)\text{.}
\end{multline*}
Applying $\Delta$, we get
\begin{equation*}
\tau_{\Omega}(\Delta\omega) \stackerf{eq:identityrhoomega}{=} \tau_{\Omega}(-\delta_2^{*}\rho)=-\delta_2^{*}\varepsilon_{\nu}\stackerf{lem:epsiloncancel}{=}\Delta\xi=\Delta(\xi+\frac{1}{2}\Delta\zeta)\text{,}
\end{equation*}
which is also consistent.
\end{comment}
We use the reformulation 
\begin{equation*}
\omega =  \left \langle   \delta^{*} \theta \wedge \pr_1^{*}A  \right \rangle=\left \langle \delta^{*}\bar\theta\wedge \pr_2^{*}A  \right \rangle\text{.}
\end{equation*}
\begin{comment}
Indeed, using the transformation formula
\begin{equation*}
\mathrm{Ad}_{\delta}^{-1}(\pr_2^{*}A) =\mathrm{Ad}_{\delta}^{-1}(\mathrm{Ad}_{g}^{-1}(\pr_{1}^{*}A)+g^{*}\theta)= \pr_1^{*}A-\delta^{*}\theta
\end{equation*}
we obtain
\begin{equation*}
\omega =  \left \langle   \delta^{*} \theta \wedge \pr_1^{*}A  \right \rangle =\left \langle  \delta^{*}\theta\wedge \mathrm{Ad}_{\delta}^{-1}(\pr_2^{*}A)  \right \rangle=\left \langle \delta^{*}\bar\theta\wedge \pr_2^{*}A  \right \rangle\text{.}
\end{equation*}
\end{comment}
Then, we calculate for tangent vectors $X_1\in T_{\tau_1}LP$, $X_2\in T_{\tau_2}LP$, and their differences $\delta \df L\delta(\tau_1,\tau_2)\in LG$ and $Y:=\mathrm{d}L\delta(X_1,X_2)\in T_{\delta }LG$:
\begin{eqnarray*}
&&\hspace{-1cm}\tau_{\Omega}(\omega)|_{(\tau_1,\tau_2)}(X_1,X_2) \\&=& \int_0^1 \omega_{(\tau_1(z),\tau_2(z))}((\partial_z\tau_1(z),\partial_z\tau_2(z)),(X_1(z),X_2(z))) \mathrm{d}z
\\&=&  \int_0^1 \left \langle \partial_z\delta(z)\delta(z)^{-1},A_{\tau_2(z)}(X_2(z))  \right \rangle \mathrm{d}z -\int_0^1 \left \langle Y(z)\delta(z)^{-1}, A_{\tau_2(z)}(\partial_z\tau_2(z))   \right \rangle \mathrm{d}z
\\&=& (\xi-\frac{1}{2}\Delta\zeta)|_{\tau_1,\tau_2}(X_1,X_2)\text{,}
\end{eqnarray*}
where the last step is obtained using Lemma \ref{lem:deltazeta}. 
\begin{comment}
Namely, from that Lemma we get
\begin{multline*}
(\xi-\frac{1}{2}\Delta\zeta)_{\tau}(X_1,X_2) =\int_0^1 \left \langle \partial_z\delta(z)\delta(z)^{-1},A_{\tau_2(z)}(X_2(z))  \right \rangle \mathrm{d}z \\- \int_0^1 \left  \langle A_{\tau_2(z)}(\partial_z\tau_2(z)), Y(z)\delta(z)^{-1}    \right \rangle \mathrm{d}z \end{multline*}
in explicit expressions. 
\end{comment}
\end{proof}

\begin{comment}
As a Corollary, we reproduce Proposition \ref{prop:DeltaB}:
\begin{multline*}
\Delta B = -\Delta \tau_{\Omega}(CS(A))=-\tau_{\Omega}(\delta^{*}H + \mathrm{d}\omega)=-L\delta^{*} \tau_{\Omega}(H)+\mathrm{d}\tau_{\Omega}(\omega)\\=L\delta^{*}\mathrm{curv}(\nu)+\mathrm{d}(\xi-\frac{1}{2}\Delta\zeta)=\mathrm{curv}(\nu_{corr})\text{.} \end{multline*}
\end{comment}

\label{sec:equivconn}

We are now in position to provide the second half of the  main result of this article, an equivalence between string structures in $M$ and trivializations of the spin lifting gerbe. 
\begin{comment}
We start in the setting with connections.
\end{comment}

\begin{theorem}
\label{th:trequivcon}
Let $M$ be a connected spin manifold. Then, transgression and regression functors induce an equivalence of categories,
\begin{equation*}
\ststcon 1M  \cong \bigset{15em}{Fusion trivializations of $\mathcal{S}_{LM}$ with superficial fusive connection compatible with $(\chi_{\corr},B_{\corr})$}\text{.} \end{equation*}
This equivalence is equivariant with respect to the action of $\hc 1 \ugrbcon M$ on the left hand side and the action of $\ufusbunconsf{LM}$ on the right hand side, under the equivalence between these monoidal categories. Moreover, if $K\in\Omega^3(M)$ is the 3-form associated to a geometric string structure by Theorem \ref{th:stringconnections} \erf{th:stringconnections:3form}, and $\rho\in\Omega^2(LM)$ is the covariant derivative of the corresponding trivialization of $\mathcal{S}_{LM}$, then $\tau_{\Omega}(K)=-\rho$. 
\end{theorem}

\begin{proof}
The purpose of the category $\ststcon 1M$ introduced in Section \ref{sec:decatstringcon} was that its definition is purely in terms of the presheaf  $\hc 1 \ugrbcon-$: its objects are pairs $(\mathcal{S},[\mathcal{A}])$ consisting of an object $\mathcal{S}$ in $\hc 1 \ugrbcon {FM}$ and a morphism 
\begin{equation*}
[\mathcal{A}]: \pr_2^{*}\mathcal{S} \otimes \mathcal{P} \to \pr_1^{*}\mathcal{S}
\end{equation*}
of $\hc 1 \ugrbcon{FM^{[2]}}$ such that an equality of morphisms of $\hc 1 \ugrbcon{FM^{[3]}}$ holds, namely \erf{eq:diagfillsigma}. Likewise, the morphisms of $\ststcon 1M$ are morphisms $[\mathcal{B}]:\mathcal{S} \to \mathcal{S}'$ of $\hc 1 \ugrbcon {FM^{[2]}}$ such that an equality of morphisms of $\hc 1 \ugrbcon{FM^{[3]}}$ holds, namely \erf{eq:beta}. Now, we recall from Theorem \ref{th:equiv} that transgression and regression form an equivalence
\begin{equation*}
\hc 1 \ugrbcon X \cong \ufusbunconsf {LX}
\end{equation*}
that is natural in $X$ and monoidal. Hence, $\ststcon 1M$ is equivalent to the following category:
\begin{itemize}

\item 
An object is a pair $(T,\kappa)$ of an object $T$ in $\ufusbunconsf{LFM}$
and a morphism 
\begin{equation*}
\kappa: \pr_2^{*}T \otimes \tr_{\mathcal{P}} \to \pr_1^{*}T
\end{equation*}
in $\ufusbunconsf{LFM^{[2]}}$
such that the  equality 
\begin{equation*}
\alxydim{@C=2cm@R=\xypicst}{\pr_3^{*}T\otimes \pr_{23}^{*}\tr_{\mathcal{P}} \otimes \pr_{12}^{*}\tr_{\mathcal{P}}  \ar[r]^-{\pr_{23}^{*}\kappa \otimes \id} \ar[d]_{\id \otimes \tr_{\mathcal{M}'}} & \pr_2^{*}T \otimes \pr_{12}^{*}\tr_{\mathcal{P}} \ar[d]^{\pr_{12}^{*}\kappa} \\ \pr_3^{*}T \otimes \pr_{13}^{*}\tr_{\mathcal{P}} \ar[r]_-{\pr_{13}^{*}\kappa} & \pr_1^{*}T}
\end{equation*}
of morphisms of $\ufusbunconsf{LFM^{[3]}}$ holds.

\item
A morphism is a morphism $\varphi: T \to T'$  in $\ufusbunconsf{LFM^{[2]}}$, such that the equality
\begin{equation*}
\alxydim{@=\xypicst}{\pr_2^{*}T\otimes \tr_{\mathcal{P}}  \ar[r]^-{\kappa} \ar[d]_{\pr_2^{*}\varphi} & \pr_1^{*}T \ar[d]^{\pr_1^{*}\varphi} \\ \pr_2^{*}T' \otimes \tr_{\mathcal{P}} \ar[r]_-{\kappa'} & \pr_1^{*}T'}
\end{equation*}
of morphisms in $\ufusbunconsf{LFM^{[3]}}$ holds.

\end{itemize}
This is precisely the category of fusion trivializations of $\tr_{\mathbb{CS}_M}$ with superficial fusive connection compatible with $(\nu_P,\tau_{\Omega}(CS(A)))$, the connection on the transgression of the Chern-Simons 2-gerbe. The connection-preserving, fusion-preserving isomorphism of Theorem \ref{th:spinlift}
identifies this category with the claimed one.

The equivariance under the $\hc 1 \ugrbcon M$-actions follows immediately from the definitions. Suppose $K\in \Omega^3(M)$ is the 3-form associated to a string structure $(\mathcal{S},[\mathcal{A}])$, i.e. $\pi^{*}K=CS(A) + \mathrm{curv}(\mathcal{S})$, then $\tau_{\Omega}(K)$ satisfies 
\begin{equation*}
L\pi^{*}\tau_{\Omega}(K)=\tau_{\Omega}(CS(A))+\tau_{\Omega}(\mathrm{curv}(\mathcal{S}))=-B_{\corr}-\mathrm{curv}(\tr_{\mathcal{S}})\text{,}
\end{equation*}
see \cite[Eq. 24]{Coquereaux1998} and \erf{eq:trcurv}. Thus, $-\tau_{\Omega}(K)$ is the covariant derivative of $(\tr_{\mathcal{S}},\tr_{\mathcal{A}})$. \end{proof}

Now we come to the correspondence between string structures and trivializations of the spin lifting gerbe in the setting without connections. 

\begin{theorem}
\label{th:trequiv}
Let $M$ be a connected spin manifold. Then, regression  induces an equivalence
\begin{equation*}
h\trivthfus{\mathcal{S}_{LM}}\cong \stst 1M
\end{equation*}
between the homotopy category of thin fusion trivializations of the spin lifting gerbe $\mathcal{S}_{LM}$ and the category of string structures on $LM$.
This equivalence is equivariant with respect to the action of $h\ufusbunth{LM}$ on the left hand side and the action of  $\hc 1 \ugrb M$ on the right hand side, along the equivalence between the two monoidal categories.
\end{theorem}

\begin{proof}
We proceed as in the proof of Theorem \ref{th:trequivcon}, now using the equivalence 
\begin{equation*}
h\ufusbunth {LX} \cong \hc 1 \ugrb X
\end{equation*}
established by the regression functor $\un_x$
which is natural in $X$, monoidal, and depends on the choice of a point $x\in X$. Choosing a point $p\in FM$ (and then using the point $(p,...,p)\in FM^{[k]}$ in all higher fibre products) we find an equivalence $K: \mathcal{C} \to \stst 1M$, where $\mathcal{C}$ stands for the following category:
\begin{itemize}

\item 
An object is a pair $(T,\kappa)$ of an object $T$ in $h\ufusbunth{LFM}$
and a morphism 
\begin{equation*}
\kappa: \pr_2^{*}T \otimes \tr_{\mathcal{P}} \to \pr_1^{*}T
\end{equation*}
in $h\ufusbunth{LFM^{[2]}}$
such that the  equality 
\begin{equation}
\label{eq:proofth1}
\alxydim{@C=2cm@R=\xypicst}{\pr_3^{*}T\otimes \pr_{23}^{*}\tr_{\mathcal{P}} \otimes \pr_{12}^{*}\tr_{\mathcal{P}}  \ar[r]^-{\pr_{23}^{*}\kappa \otimes \id} \ar[d]_{\id \otimes \tr_{\mathcal{M}'}} & \pr_2^{*}T \otimes \pr_{12}^{*}\tr_{\mathcal{P}} \ar[d]^{\pr_{12}^{*}\kappa} \\ \pr_3^{*}T \otimes \pr_{13}^{*}\tr_{\mathcal{P}} \ar[r]_-{\pr_{13}^{*}\kappa} & \pr_1^{*}T}
\end{equation}
of morphisms of $h\ufusbunth{LFM^{[3]}}$ holds.

\item
A morphism is a morphism $\varphi: T \to T'$  in $h\ufusbunth{LFM^{[2]}}$, such that the equality
\begin{equation}
\label{eq:proofth2}
\alxydim{@=\xypicst}{\pr_2^{*}T\otimes \tr_{\mathcal{P}}  \ar[r]^-{\kappa} \ar[d]_{\pr_2^{*}\varphi} & \pr_1^{*}T \ar[d]^{\pr_1^{*}\varphi} \\ \pr_2^{*}T' \otimes \tr_{\mathcal{P}} \ar[r]_-{\kappa'} & \pr_1^{*}T'}
\end{equation}
of morphisms in $h\ufusbunth{LFM^{[2]}}$ holds.
\end{itemize}
The connection-preserving, fusion-preserving isomorphism between $\tr_{\mathbb{CS}_M}$ and $\mathcal{S}_{LM}$ induces an equivalence between $\mathcal{C}$ and a category $\mathcal{C}'$ obtained from $\mathcal{C}$ by replacing $\tr_{\mathcal{P}}$ by the principal $\ueins$-bundle $\q$ of $\mathcal{S}_{LM}$ and $\tr_{\mathcal{M}'}$ by the bundle gerbe product $\mu$ of $\mathcal{S}_{LM}$.

Now we go into the details of the category $\trivthfus{\mathcal{S}_{LM}}$ of thin fusion trivializations of the spin lifting gerbe, as introduced in Definition \ref{def:thinfusiontriv}. 
The category $\mathcal{C}'$
receives a functor 
\begin{equation*}
K': \trivthfus{\mathcal{S}_{LM}} \to \mathcal{C}'
\end{equation*}
defined in the following natural way: 
\begin{itemize}

\item 
It takes a thin fusion trivialization $(T,\kappa,\lambda,d)$ and sends it to the pair composed of the object $(T,\lambda,d)$ in $\ufusbunth{FM}$ and of the homotopy class of $\kappa$, which is a morphism in $h\ufusbunth{LFM^{[2]}}$ making diagram \erf{eq:proofth1} commutative. 
\item
A morphism $\varphi$ between thin fusion trivializations $(T,\kappa,\lambda,d)$ and $(T',\kappa',\lambda',d')$ is sent to its homotopy class, which is a morphism in $h\ufusbunth {LFM}$. The condition, i.e. the commutativity of diagram \erf{eq:proofth2}, is exactly the same as in Definition \ref{def:thinfusiontriv}. 
\end{itemize}
We obtain a commutative diagram of categories and functors,
\begin{equation}
\label{eq:commdiagright}
\alxydim{@=\xypicst}{\bigset{13em}{Fusion trivializations of $\mathcal{S}_{LM}$ with compatible superficial fusive connection} \ar[rr]^-{\text{Theorem \ref{th:trequivcon}}} \ar[d] &&  \ststcon 1M\ar[d]^{F_1} \\ \trivthfus{\mathcal{S}_{LM}} \ar[r]_-{K'} & \mathcal{C}' \cong \mathcal{C} \ar[r]_-{K} & \stst 1M\text{.}}
\end{equation}
Since the functor  $F_1$ is essentially surjective (Proposition \ref{prop:passagesur} \erf{prop:passagesur:1}), and the functor on the top is an equivalence (Theorem \ref{th:trequivcon}), it follows that $K'$ is essentially surjective. It is also full: suppose $\varphi$ is a morphism in $\mathcal{C}'$. It is represented by a fusion-preserving, thin bundle morphism $\varphi: T \to T'$  such that diagram \erf{eq:proofth2} is commutative in $h\ufusbunth{LFM^{[2]}}$. This means that the representative $\varphi$ is a morphism in $\trivthfus{\mathcal{S}_{LM}}$.

However, the functor $K'$ is not faithful. This problem is solved by passing  to the \emph{homotopy category} of thin fusion trivializations, $h\trivthfus{\mathcal{S}_{LM}}$. We note that the functor $K'$ is well-defined on the homotopy category, i.e. it induces a functor $hK'$ making the diagram
\begin{equation*}
\alxydim{@=\xypicst}{\trivthfus{\mathcal{S}_{LM}} \ar[r]^-{K'} \ar[d] & \mathcal{C}' \\ h\trivthfus{\mathcal{S}_{LM}} \ar[ur]_-{hK'}}
\end{equation*}
strictly commutative.
Indeed, a morphism on the left is a class $[\varphi]$ of morphisms between thin fusion trivializations, in which two morphisms $\varphi_0$ and $\varphi_1$ are identified, if there exists a homotopy through morphisms of $\trivthfus{\mathcal{S}_{LM}}$. Such a homotopy is, in particular, a  homotopy through fusion-preserving, thin bundle morphisms. Thus, $hK'([\varphi]) := [\varphi]$ is well-defined. 

As the projection $\trivthfus{\mathcal{S}_{LM}} \to h\trivthfus{\mathcal{S}_{LM}}$ is  surjective, we  deduce from the fact that $K'$  is essentially surjective and full, that $hK'$ is essentially surjective and full, too. It remains to show that $hK'$ is faithful. Suppose $\varphi_0$   and $\varphi_1$ are morphisms in $\trivthfus{\mathcal{S}_{LM}}$, such that they are equal in $\mathcal{C}'$. That is, there exists a homotopy $h$ between $\varphi_0=h_0$ and $\varphi_1=h_1$ through fusion-preserving, thin bundle morphisms $h_t: T \to T'$, identifying $\varphi_0$ and $\varphi_1$ in $h\ufusbunth{LFM}$. We  show that the same homotopy $h$ is a homotopy through morphisms in $\trivthfus{\mathcal{S}_{LM}}$, i.e.  the diagram
\begin{equation}
\label{eq:proofth3}
\alxydim{@=\xypicst}{\pr_2^{*}T  \otimes \q  \ar[r]^-{\kappa} \ar[d]_{\pr_2^{*}h_t \otimes \id} & \pr_1^{*}T \ar[d]^{\pr_1^{*}h_t} \\ \pr_2^{*}T' \otimes \q \ar[r]_-{\kappa'} & \pr_1^{*}T'}
\end{equation}
commutes in $h\ufusbunth {LFM^{[2]}}$ for all $t\in [0,1]$.
Indeed, as $\varphi_0$ and $\varphi_1$ are morphisms in $\trivthfus{\mathcal{S}_{LM}}$ the diagram commutes for $t=0$ (and $t=1$).  That is, there is a homotopy $H$ between $\pr_1^{*}\varphi_0 \circ \kappa = H_0$ and $\kappa' \circ (\pr_2^{*}\varphi_0 \otimes \id)=H_1$ through fusion-preserving, thin bundle morphisms $H_s: \pr_2^{*}T \otimes \q \to \pr_1^{*}T$. Now we have the following homotopies:
\begin{enumerate}
\item 
from $\pr_1^{*}h_t \circ \kappa$ to $\pr_1^{*}\varphi_0 \circ \kappa$, namely $s\mapsto \pr_1^{*}h_{t(1-s)} \circ \kappa$

\item
from $\pr_1^{*}\varphi_0 \circ \kappa$ to $\kappa' \circ (\pr_2^{*}\varphi_0 \otimes \id)$, namely $s\mapsto H_s$

\item
from $\kappa' \circ (\pr_2^{*}\varphi_0 \otimes \id)$ to $\kappa' \circ (\pr_2^{*}h_t \otimes \id)$, namely $s \mapsto \kappa' \circ (\pr_2^{*}h_{st} \otimes \id)$
\end{enumerate}
These can be concatenated to a smooth homotopy showing that diagram \erf{eq:proofth3} is commutative in $h\ufusbunth{LFM}$.
\end{proof}

\section{Proof of Theorem \ref{th:main}}

\label{sec:proof}

The various functors we have introduced  form the commutative diagram of Theorem \ref{th:main}:
\begin{equation*}
\alxydim{@R=\xypicst@C=2.4em}{\spstconsffus \ar[d] \ar[r]^-{\text{Cor. \ref{co:equivsupfusion}}} & *++{\bigset{10em}{Fusion trivializations of the spin lifting gerbe $\mathcal{S}_{LM}$ with superficial connection}}  \ar[r]^-{\text{Th. \ref{th:trequivcon}}} \ar[d] & \ststcon 1M \ar@{<-}[r]^-{\text{Th. \ref{th:truncationcon}}} \ar[d]_{F_1} & \hc 1 \ststcon\relax M \ar[d]^{\hc 1 F_2} \\ h\spstthfus \ar[r]_-{\text{Prop. \ref{prop:equivthinfusion}}} &  *++++++{\hspace{-1em}\bigset{10em}{Homotopy category of thin fusion trivializations of the spin lifting gerbe $\mathcal{S}_{LM}$}\hspace{-1em}} \ar[r]_-{\text{Th. \ref{th:trequiv}}} & \stst 1M \ar@{<-}[r]_-{\text{Th. \ref{th:truncation}}} & \hc 1 \stst\relax M\text{.}}
\end{equation*}
The separate diagrams have been discussed in Section \ref{sec:liftingtheoryspinconnections}, see \erf{eq:commdiagleft}, in Section \ref{sec:transgression}, see \erf{eq:commdiagright}, and in Section \ref{sec:decatstringcon}, see \erf{eq:passage}. The arrows are labelled with references to those statements where we have proved that they are equivalences of categories and have the claimed equivariance properties.

If $M$ is not a string manifold, $\stst\relax M$ is empty. Thus,  all categories in above diagram must be empty since they have functors to $\hc 1 \stst\relax M$. Now suppose  $M$ is a string manifold. Then, $\stst\relax M$ is non-empty. Since the functor $F_2$ is essentially surjective by Theorem \ref{th:stringconnections} \erf{th:stringconnections:exist}, $\ststcon\relax M$  is non-empty, too, and so are all other categories in the diagram. It also follows that all vertical functors in the diagram are essentially surjective.

\tocsection{Table of notation}

\newcommand{\notation}[3]{
  \noindent
  \begin{minipage}[t]{0.18\textwidth}#1\end{minipage}
  \begin{minipage}[t]{0.81\textwidth}#2\if!#3!\else(#3)\fi\end{minipage}\vspace{0.8em}}

\notation{$\ubun{X}$}{The category of Fréchet principal $\ueins$-bundles over $X$}{}
\notation{$\ubuncon{X}$}{The category of Fréchet principal $\ueins$-bundles with connection}{}
\notation{$\ugrb X$}{The bicategory of bundle gerbes over $X$}{}
\notation{$\ugrbcon X$}{The bicategory of bundle gerbes with connection over $X$}{}
\notation{$h\mathcal{C}$}{The homotopy category of a (topological) category, defined by identifying homotopic morphisms.}{}
\notation{$\hc k\mathcal{C}$}{The $k$-truncation of a (higher) category, defined using $(k+1)$-isomorphism classes of $k$-morphisms. }{}
\notation{$\spstthfus$}{The category of thin fusion spin structures on $LM$}{Section \ref{sec:fusionspin}}
\notation{$\spstconsffus$}{The category of superficial geometric fusion spin structures }{Section  \ref{sec:spinconnections}}
\notation{$\sc (M)$}{The set of string classes on $M$}{Section \ref{sec:stringclasses}}
\notation{$\sccon (M)$}{The set of differential string classes on $M$}{Section \ref{sec:differentialstringclasses}}
\notation{$\stst \relax M$}{The bicategory of  string structures on $M$}{Section \ref{sec:stringclasses}}
\notation{$\stst k M$}{A decategorification of the bicategory of  string structures }{Section \ref{sec:stringclasses}}
\notation{$\ststcon \relax M$}{The bicategory of geometric string structures on $M$}{Section \ref{sec:decatstringcon}}
\notation{$\ststcon k M$}{A decategorification of the bicategory of geometric string structures on $M$}{Section \ref{sec:decatstringcon}}
\notation{$\ufusbunth{LX}$}{The category of Fréchet principal $\ueins$-bundles over $LX$ with fusion product and thin structure}{}
\notation{$\ufusbunconsf{LX}$}{The category of Fréchet principal $\ueins$-bundles over $LX$ with fusion product and superficial connection}{}

\bibliographystyle{kobib}
\bibliography{kobib}

\end{document}